\documentclass[10pt,russian]{article}

\usepackage{a4wide}
\usepackage[english]{babel}
\usepackage{amsfonts}
\usepackage{amssymb, amsthm}
\usepackage{amsmath}
\usepackage{color}
\usepackage[utf8]{inputenc}
\usepackage[pdftex]{graphicx}
\usepackage{dsfont}
\usepackage{hyperref}
\usepackage{url}
\usepackage[titletoc,toc,title]{appendix}

\newtheorem{Le}{Lemma}[section]
\newtheorem{Prop}[Le]{Proposition}
\newtheorem{Th}[Le]{Theorem}
\newtheorem{Cor}[Le]{Corollary}
\numberwithin{equation}{section}

\renewenvironment{proof}{\paragraph{Proof.}}{\hfill$\square$}

\theoremstyle{definition}
\newtheorem{Rem}[Le]{Remark}

\DeclareMathOperator{\sign}{sign}
\DeclareMathOperator{\grad}{grad}

\newcommand{\RNumb}[1]{\uppercase\expandafter{\romannumeral #1\relax}}

\renewcommand{\le}{\leq}
\renewcommand{\ge}{\geq}
\newcommand{\ci}[1]{_{ {}_{\scriptstyle #1}}}
\newcommand{\cii}[1]{_{ {}_{ #1}}}
\newcommand{\av}[2]{\langle #1\rangle\cii {#2}}
\newcommand{\df}{\buildrel\rm{def}\over=}

\newcommand{\R}{\mathbb{R}}
\newcommand{\AAA}{A}

\newcommand{\Set}[2]{\Big\{{#1}\colon{#2}\Big\}}
\newcommand{\Bell}{\boldsymbol{B}}

\newcommand{\diff}[2]{\frac{\partial #1}{\partial #2}}

\newcommand{\eqlb}[2]{\begin{equation}\label{#1} #2 \end{equation}}

\renewcommand{\leq}{\leqslant}
\renewcommand{\geq}{\geqslant}
\newcommand{\Am}[1]{A_{m_#1}}
\newcommand{\Ak}[1]{A_{k_#1}}
\newcommand{\half}{\tfrac12}
\newcommand{\PsiL}{\Psi\cii{\mathrm L}}
\newcommand{\PsiR}{\Psi\cii{\mathrm R}}
\newcommand{\FL}{F\cii{\mathrm L}}
\newcommand{\FR}{F\cii{\mathrm R}}
\newcommand{\HL}{H\cii{\mathrm L}}
\newcommand{\HR}{H\cii{\mathrm R}}
\newcommand{\KL}{K\cii{\mathrm L}}
\newcommand{\KR}{K\cii{\mathrm R}}
\newcommand{\wL}{w\cii{\mathrm L}}
\newcommand{\wR}{w\cii{\mathrm R}}

\newcommand{\BMO}{\mathrm{BMO}}
\newcommand{\vf}{\varphi}
\newcommand{\eps}{\varepsilon}

\newcommand{\condim}{\mathfrak{C}}
\newcommand{\qqq}{Q}
\newcommand{\KerP}{P}
\newcommand{\KerH}{H}
\newcommand{\Kerr}{K}

\newcommand{\DDD}{D}
\newcommand{\EEE}{E}

\begin{document}
\author{Vasily Vasyunin \and Pavel Zatitskiy \and Ilya Zlotnikov}
\date{
    \today
}
\title{Sharp multiplicative inequalities with $\BMO$~$\mathrm{II}$
\thanks{Support by the Russian Science Foundation grant 19-71-10023.}}

\maketitle
\begin{abstract}
We find the best possible constant~$C$ in the inequality 
$$
\|\vf\|_{L^r}^{\phantom{\frac{p}{r}}}\leq C\|\vf\|_{L^p}^{\frac{p}{r}}\|\vf\|_{\BMO}^{1-\frac{p}{r}}
$$
for all possible values of parameters $p$ and~$r$ such that $1 \le p < r < +\infty$. We employ the Bellman function technique to solve 
this problem. The Bellman function of three variables corresponding to this problem has a rather complicated structure, however, we managed to provide the explicit formulas for this function.
First, we solve the problem on an interval and then transfer our results to the circle and the line. We also obtain explicit estimates in multi-dimensional cases.  

\medskip


\end{abstract}

\section{Introduction}
This is a continuation of the paper~\cite{SVZ},
where a partial case of the following problem was considered: for fixed parameters $p$ and $r$, $1\leq p < r < \infty$, 
find the sharp constant $C=C(p,r)$ such that the inequality
\eqlb{FirstMultiplicative}{
\|\vf\|_{L^r}^{\phantom{\frac{p}{r}}}\leq C\|\vf\|_{L^p}^{\frac{p}{r}}\|\vf\|_{\BMO}^{1-\frac{p}{r}}
}
is true for all functions $\vf$ from~$\BMO$. It was proved in~\cite{SVZ} that in the case $r\geq 2$ the best
possible constant is
\eqlb{eq091101}{
C(p,r)=\Big(\frac{\Gamma(r+1)}{\Gamma(p+1)}\Big)^\frac 1r.
}
We also mention that the same estimate was obtained in~\cite{SV2} for the partial case $1\leq p\leq2\leq r<+\infty$.
Here we find the best constant $C(p,r)$ for the remaining case $1\leq p<r<2$. The expression for this constant is implicit and too difficult to be presented at the very beginning of the paper. All the details can be found in Section~\ref{const}.
So we complete the work announced in Remark~1.4 of~\cite{SVZ}.

Lebesgue spaces and~$\BMO$ here can be considered either on the line $\mathbb{R}$, or on the unit 
circle $\mathbb{T}$, or on an interval $I$. In order to speak about  sharp constant, we need to specify the 
$\BMO$-norm. We will consider $L^2$-based $\BMO$-norm, namely,
\eqlb{BMOnorm}{
\|\vf\|_{\BMO}^2\df\sup\limits_J\av{\,|\vf-\av{\vf}{J}|^2}{J},
}
where the supremum is taken over all subintervals $J$.
Here and in what follows we use the notation~$\av{\psi}{E}$ to denote the average of a function~$\psi$ 
over a set~$E$ of positive finite measure, that is
\begin{equation*}
\av{\psi}{E}\df\frac1{|E|}\int\limits_E\psi.
\end{equation*}
Since the relation~\eqref{BMOnorm} defines a seminorm and for constant function 
it is zero, we need to impose the additional 
restriction $\av\vf{}=0$ for the cases of a circle and an interval to obtain \eqref{FirstMultiplicative}. 

We will not repeat here the motivation and the references concerning possible applications. All of 
that may be found in~\cite{SVZ}. We stress that we are not so much interested in the inequality itself but rather in the corresponding Bellman function due to its importance for the future development of the Bellman function method.

We state now formally the main results of the paper.

\begin{Th}\label{MultiplicativeTheoremInterval}
Let $I,$ $I \subset \R,$ be an interval. 
The inequality
\eqlb{eq141202}{
\|\vf\|_{L^r(I)}^{\phantom{\frac{p}{r}}}\leq
C(p,r)\|\vf\|_{L^p(I)}^{\frac{p}{r}}\|\vf\|_{\BMO(I)}^{1-\frac{p}{r}},
\qquad\vf\in\BMO,\quad\av{\vf}{I}=0,
}
holds true for $1 \le p \le r < 2$ with the sharp constant $C(p,r)$ described in Section~{\rm\ref{const}}.
\end{Th}

\begin{Rem}\label{Rem1}
For $p=1$  we have explicit expression for the constant 
$$
C^r(1,r) = \begin{cases}
 \ \ 2^{r-1},& 1<r\leq 2;\\
 \Gamma(r+1),& \ \ 2\leq  r.
\end{cases}
$$
\end{Rem}

Theorem~\ref{MultiplicativeTheoremInterval} implies the corresponding inequality for the circle and for the line.

\begin{Th}\label{Ctheorem}
For $1 \leq p \leq r < 2$ the inequality
\eqlb{MultCircle}{
\|\vf\|_{L^r(\mathbb{T})}^{\phantom{\frac{p}{r}}} \leq C(p,r)\|\vf\|_{L^p(\mathbb{T})}^{\frac{p}{r}}
\|\vf\|_{\BMO(\mathbb{T})}^{1-\frac{p}{r}},\qquad \vf \in \BMO(\mathbb{T}),\ \int\limits_{\mathbb{T}}\vf = 0,
}
holds true with the same sharp constant $C(p,r) $.
\end{Th}

\begin{Th}\label{Rtheorem}
For $1 \leq p \leq r < 2$ the inequality
\eqlb{eq141201}{
\|\vf\|_{L^r(\R)}^{\phantom{\frac{p}{r}}} \leq C(p,r)\|\vf\|_{L^p(\R)}^{\frac{p}{r}}
\|\vf\|_{\BMO(\R)}^{1-\frac{p}{r}},\qquad \vf \in L^p(\R),
}
holds true with the same sharp constant $C(p,r)$.
\end{Th}

We deduce Theorems~\ref{Ctheorem} and~\ref{Rtheorem} from 
Theorem~\ref{MultiplicativeTheoremInterval} in Section~\ref{sec_circle_line}.

We also prove several statements for higher dimensions. Let $\condim(n) = 4(1+2\sqrt{n-1})$ for $n\in \mathbb{N}$. 
\begin{Th}\label{Cubedtheorem}
If $1 \leq p \leq r<\infty$\textup, then the inequality
\eqlb{eq171001}{
\|\vf\|_{L^r(\qqq)}^{\phantom{\frac{p}{r}}}\leq
 C(p,r)\|\vf\|_{L^p(\qqq)}^{\frac{p}{r}}\Big(\condim(n)\|\vf\|_{\BMO(\qqq)}\Big)^{1-\frac{p}{r}}\!\!\!\!\!,
\qquad\vf\in\BMO(\qqq),\quad\av{\vf}{\qqq}=0,
}
holds\textup, where either $\qqq=I^n$ is a cube in $\mathbb{R}^n$ or $\qqq = \mathbb{T}^n$ is an $n$-dimensional torus. The $\BMO$-norm in~\eqref{eq171001} is defined by~\eqref{BMOnorm}\textup, where supremum is taken over subcubes $J$.
\end{Th}
This theorem follows immediately from Theorem~\ref{MultiplicativeTheoremInterval} (and Theorem~1.1 in~\cite{SVZ}) and the fact that the monotone rearrangement operator acts from $\BMO(I^n)$ and $\BMO(\mathbb{T}^n)$ to $\BMO(I)$ with the norm bounded by $\condim(n)$. For this estimate see \cite{BDG}, where the estimate for the monotone rearrangement operator is obtained for $L^1$-based $\BMO$-norm. The corresponding estimate for $L^2$-based norm may be deduced by a straightforward modification. 

One may use the limiting arguments (see the details in paper~\cite{SVZ}, Section~6.1) to obtain the following result.
\begin{Cor}\label{Rdcor}
If $1 \leq p \leq r<\infty$\textup, then we have the following inequality\textup:
$$
\|\vf\|_{L^r(\mathbb{R}^n)}^{\phantom{\frac{p}{r}}}\leq
 C(p,r)\|\vf\|_{L^p(\mathbb{R}^n)}^{\frac{p}{r}} \Big(\condim(n) \|\vf\|_{\BMO(\mathbb{R}^n)}\Big)^{1-\frac{p}{r}}\!\!\!\!\!,
\qquad\vf\in L^p(\mathbb{R}^n).
$$
\end{Cor}

Another natural norm on $\BMO(\mathbb{R}^n)$ is the norm defined by~\eqref{BMOnorm}, where supremum is taken over balls~$J$. We will call it the ball-based norm. In order to obtain estimates like we have in Corollary~\ref{Rdcor} related to such ball-based norm, we apply another approach inspired by~\cite{SlavinZ}. Following~\cite{SlavinZ},  we prove dimension-free estimate using the Garcia-type norm on $\BMO(\mathbb{R}^n)$. For $y \in \mathbb{R}^n$ and $t>0$ consider the Poisson kernel $\KerP_t$ and the heat kernel $\KerH_t$: 
\begin{equation}\label{kernels}
\KerP_t(y)=\frac{\Gamma\big(\frac{n+1}2\big)}{\pi^{\frac{n+1}2}}\, \frac t{(t^2+|y|^2)^{\frac{n+1}2} }\,,\qquad
\KerH_t(y)=\frac1{(4\pi t)^{\frac n2}}\, e^{-\frac{|y|^2}{4t}}.
\end{equation}
We will write $\Kerr_t$ for both $\KerP_t$ and $\KerH_t$. 

For a function $\vf$ on $\mathbb{R}^n$ consider its $\Kerr$-extension onto $\mathbb{R}^n\times\mathbb{R}_+$ given by convolution:
$$
\vf_{\scriptscriptstyle \Kerr}(y,t)=(\Kerr_t*\vf)(y), \qquad y \in \mathbb{R}^n, \ \ t>0.
$$
For $y \in \mathbb{R}^n$, $t>0$, this value may be considered as the average of $\vf$ with the weight $\Kerr_t$ centered at $y$ instead of the average over a ball with radius $t$ centered at $y$. 
It is well-known that 
$$
\|\vf\|_{\scriptscriptstyle \Kerr}:=\sup_{(y,t)\in\mathbb{R}^n \times \mathbb{R}_+}\big((\vf^2)_{\scriptscriptstyle \Kerr}(y,t)-(\vf_{\scriptscriptstyle \Kerr}(y,t))^2\big)^{\frac 12}
$$
is an equivalent norm on $\BMO(\mathbb{R}^n)$. For $\Kerr=\KerP$ this norm is called the Garcia norm. For the case $n=1$ you can refer to~\cite{Gar}, Chapter~\textup{VI}, Theorem~\textup{1.2}.

In~\cite{SlavinZ} it is shown how to prove dimension-free estimates on $\BMO$ and $A_p$ equipped with Garcia-type norms having the corresponding estimates for one-dimensional case. We use this approach to prove the following theorem.
\begin{Th}\label{DimHeatTh}
Let $n\in \mathbb{N}$. If $1 \leq p<r<\infty$\textup, then the inequality
$$
\|\vf\|_{L^r(\mathbb{R}^n)}^{\phantom{\frac{p}{r}}}\leq
C(p,r)\|\vf\|_{L^p(\mathbb{R}^n)}^{\frac{p}{r}}\|\vf\|_{\Kerr}^{1-\frac{p}{r}},
\qquad\vf\in L^p(\mathbb{R}^n),
$$
holds\textup, where the kernel $\Kerr$ is either the Poisson kernel or the heat kernel\textup{,} see~\eqref{kernels}. 
\end{Th}
The following estimate is proved in~\cite{SlavinZ}:
\eqlb{eq171002}{
\|\vf\|_{\scriptscriptstyle \KerH} \leq \tilde{C}  \sqrt{n}\,\|\vf\|_{\BMO(\mathbb{R}^n)}, 
}
where $\tilde C$ is an absolute constant, and the $\BMO$-norm is the ball-based one. This estimate together with Theorem~\ref{DimHeatTh} implies the following corollary.

\begin{Cor}\label{DimCor}
Let $n\in \mathbb{N}$. If $1\leq p<r<\infty$\textup, then the following inequality holds\textup:
\eqlb{eq171003}{
\|\vf\|_{L^r(\mathbb{R}^n)}^{\phantom{\frac{p}{r}}}\leq
C(p,r)\tilde{C} n^{\frac{r-p}{2r}} \|\vf\|_{L^p(\mathbb{R}^n)}^{\frac{p}{r}}\|\vf\|_{\BMO(\mathbb{R}^n)}^{1-\frac{p}{r}},
\qquad\vf\in  L^p(\mathbb{R}^n),
}
where $\BMO$-norm is the ball-based one.
\end{Cor}

In the next section we repeat the definition of the main Bellman function as well as 
the definition and the properties of some auxiliary Bellman functions after~\cite{SVZ} 
for convenience of the reader.

The Bellman function method allows to obtain various estimates (sometimes sharp) in analysis and probability reducing the infinite dimensional extremal problems to  finite dimensional ones by using some auxiliary function, which now is usually called the Bellman function.
For more information regarding this method and its application  we refer the reader to the monographs~\cite{O} and~\cite{VV}.

\section{Optimization problem}
\label{s2}

We introduce the main characters. These are the following Bellman functions:

\eqlb{041102}{
\Bell_{p,r;\eps}^+(x_1,x_2,x_3)=\sup\Big\{\av{|\vf|^r}I\colon\|\vf\|_{\BMO(I)}\le\eps,\ 
\av{\vf}{I}=x_1,\ \av{\vf^2}{I}=x_2,\  \av{|\vf|^p}{I}=x_3\Big\}.
}
and
\eqlb{041103}{
\Bell_{p,r;\eps}^-(x_1,x_2,x_3)=\inf\Big\{\av{|\vf|^r}I\colon\|\vf\|_{\BMO(I)}\le\eps,\ 
\av{\vf}{I}=x_1,\ \av{\vf^2}{I}=x_2,\  \av{|\vf|^p}{I}=x_3\Big\}.
}

We say that $\vf$ is a \emph{test function} for the point $x\in\R^3$ if 
\eqlb{170401}{
\|\vf\|_{\BMO(I)} \leq \eps,\ \av{\vf}{I}=x_1,\  \av{\vf^2}{I}=x_2,\ \av{|\vf|^p}{I}=x_3.
}

The main purpose of this paper is to find explicit formulas for~$\Bell_{p,r;\eps}^\pm$. We state this 
result by referring to the formulas appearing in the forthcoming sections. After rather long preliminaries
we define our Bellman candidates: $B_1(x;p,r,\eps)$ in Subsection~\ref{B1_def} and $B_2(x;p,r,\eps)$ 
in Subsection~\ref{B2_def}. After that we will prove the following theorem.

\begin{Th}\label{BellmanTheorem}
In the case $(r-2)(r-p)>0$ the function~$\Bell_{p,r;\eps}^+$ coincides with~$B_1$ and the 
function~$\Bell_{p,r;\eps}^-$ coincides with $B_2$. In the case $(r-2)(r-p)<0$ the function~$\Bell_{p,r;\eps}^+$ 
coincides with~$B_2$ and the function~$\Bell_{p,r;\eps}^-$ coincides with $B_1$.
\end{Th}

The function $B_1$ was found in~\cite{SVZ}, where the statement of Theorem~\ref{BellmanTheorem} was proved in
the part concerning $B_1$. Our aim here is to find much more complicated function $B_2$ and to
prove the rest of Theorem~\ref{BellmanTheorem}.

To describe the function~$B_i$, we will need two auxiliary Bellman 
functions~$\Bell^\pm_{p;\eps}\colon\R^2\to\R\cup\{\pm\infty\}$ defined as follows:
\begin{align}\label{OldBellman}
\Bell^+_{p;\eps}(x_1,x_2)&=\sup\Set{\av{|\vf|^p}{I}}{\|\vf\|_{\BMO(I)}\leq\eps,\;\av{\vf}{I}=
x_1,\;\av{\vf^2}{I}=x_2},
\\
\Bell^-_{p;\eps}(x_1,x_2)&=\inf\Set{\av{|\vf|^p}{I}}{\|\vf\|_{\BMO(I)}\leq\eps,\;
\av{\vf}{I}=x_1,\;\av{\vf^2}{I}=x_2}.
\label{OldBellman-}
\end{align}
The latter two functions were studied in detail in~\cite{SV2}. We survey these results, since they 
will play an important role in our study. 

\subsection{Description of~$\Bell_{p;\eps}^\pm$}
\label{2dimBell}

The domain of both functions~$\Bell^\pm_{p;\eps}$ is
\begin{equation*}
\Omega^2_\eps = \Set{(x_1,x_2) \in \R^2}{x_1^2 \leq x_2 \leq x_1^2+\eps^2}.
\end{equation*}
By the domain of a Bellman function we mean the set of~$x$
such that the set of test functions~$\vf$ (i.\,e., the functions over which we optimize in formulas~\eqref{OldBellman} 
and~\eqref{OldBellman-}) is non-empty for this~$x$ (compare with~\eqref{170401}). Both functions also 
satisfy the boundary condition~$\Bell^\pm_{p;\eps}(t,t^2)=|t|^p$, $t\in \R$. 
From now on we omit the index~$\eps$ in the notation of domains and functions. 

To describe~$\Bell^\pm$, we need some auxiliary functions.
For $s\geq 1$ define
\eqlb{mp+}{
m_s(u) = \frac{s}{\eps}\int_{u}^{+\infty}\!\!\! e^{(u-t)/\eps} t^{s-1} dt, \qquad u \geq 0;
}
\eqlb{mp-}{
k_s(u) = \ \frac{s}{\eps}\int_\eps^u e^{(t-u)/\eps} t^{s-1} dt, \qquad u \geq \eps.
}

For any $u \in \R$ we denote the segment connecting the points~$(u,u^2)$ 
with~$(u+\eps, (u+\eps)^2+\eps^2)$ by~$S_+(u)$ and the segment connecting~$(u,u^2)$ 
with~$(u-\eps, (u-\eps)^2+\eps^2)$ by~$S_-(u)$. Note that these 
segments touch upon the upper boundary of~$\Omega^2$, that is the parabola~$x_2=x_1^2+\eps^2$. 
For any $(x_1,x_2) \in \Omega^2$ there exist unique $u_\pm = u_\pm(x_1,x_2) \in \R$ such that 
$(x_1,x_2) \in S_{\pm}(u_\pm)$, $u_+\leq x_1 \leq u_-$, see Figure~\ref{fig161101}. Namely,
$$
u_\pm=x_1\mp\Big(\eps-\sqrt{\eps^2+x_1^2-x_2}\Big).
$$

\begin{figure}[h]
    \centering
    \includegraphics[scale = 0.3]{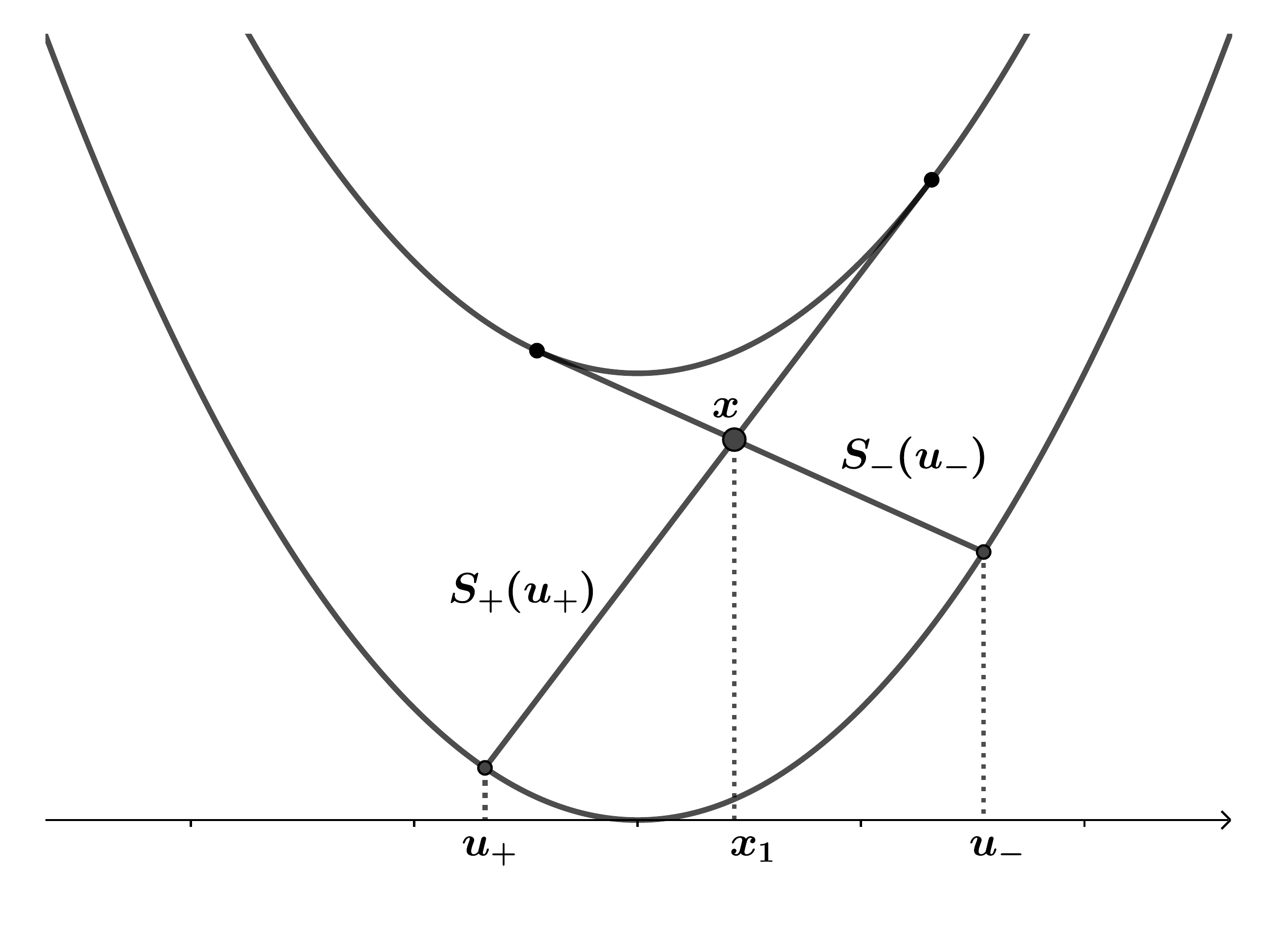}
    \caption{Left and right tangents passing through $x$}
    \label{fig161101}
\end{figure}

Define the function $\Am{p}$ on $\Omega^2$ in the following way. We put 
\eqlb{eqAm}{
\begin{aligned}
\Am{p}(x) =& u^p + m_p(u) (x_1-u), \qquad &x \in S_+(u), \qquad u\geq 0,\\
\Am{p}(x) =& |u|^p - m_p(|u|) (x_1-u),\qquad &x \in S_-(u), \qquad u \leq 0.
\end{aligned}
} 
In the curvilinear triangle between the tangents~$S_-(0)$ and~$S_+(0)$, we set
\eqlb{eqAmang}{
\Am{p}(x) = \frac{m_p(0)}{2\eps} x_2, \qquad |x_1|\leq \eps,\quad 2\eps |x_1| \leq x_2 \leq x_1^2+\eps^2.
} 
Formulas~\eqref{eqAm} and~\eqref{eqAmang} define the function $\Am{p}$ on the entire domain $\Omega^2$. 
Note that $\Am{p}$ is even with respect to~$x_1$ and $C^1$-smooth for $p>1$.

Define the function $\Ak{p}$ on $\Omega^2$ as follows. We put 
\eqlb{eqAk}{
\begin{aligned}
\Ak{p}(x) =& u^p + k_p(u) (x_1-u), \qquad &x \in S_-(u), \qquad u &\geq \eps,\\
\Ak{p}(x) =& |u|^p - k_p(|u|) (x_1-u),\qquad &x \in S_+(u), \qquad u &\leq -\eps.
\end{aligned}
} 
In the domain~$x_2 \leq \eps^2$, we set
\eqlb{eqAkcup}{
\Ak{p}(x_1,x_2) = x_2^{p/2}, \qquad x_1^2\leq x_2 \leq \eps^2.
}
Formulas~\eqref{eqAk} and~\eqref{eqAkcup} define the function $\Ak{p}$ on the entire domain $\Omega^2$. 
This function is also even with respect to~$x_1$ and $C^1$-smooth for $p\ge1$.

Now we are ready to describe the functions $\Bell^\pm$:
\eqlb{}{
\Bell^+_p = 
\begin{cases}
\Am{p}, & \text{if}\quad 2\leq p<\infty,
\\
\Ak{p}, & \text{if}\quad 1 \leq p\leq 2,
\end{cases}
\qquad
\text{and}
\qquad
\Bell^-_p = 
\begin{cases}
\Ak{p}, & \text{if}\quad 2\leq p<\infty,
\\
\Am{p}, & \text{if}\quad 1 \leq p\leq2.
\end{cases}
}

Here we collect some useful relations for derivatives of the functions $m_s$ and~$k_s$:
\eqlb{mpp+}{
m_s''(u)=\frac{s(s-1)(s-2)}{\eps}\int_{u}^{+\infty}\!\!\! e^{(u-t)/\eps} t^{s-3} dt,
}
\eqlb{mpp-}{
k_s''(u)=\ s(s-2)\eps^{s-3}e^{(\eps-u)/\eps}+\frac{s(s-1)(s-2)}{\eps}\int_{\eps}^{u}e^{(t-u)/\eps} t^{s-3} dt,
}
\eqlb{m-diff}{
-\eps m_s'(u)+ m_s(u)= su^{s-1}, \qquad \eps k_s'(u)+ k_s(u)= su^{s-1},
}
\eqlb{mpdiffnew}{
\eps \Big(m_s^{(\ell+1)} + k_s^{(\ell+1)}\Big)=m_s^{(\ell)}-k_s^{(\ell)},\qquad\ell\geq 0,\\
}
where the notation $g^{(k)}$ means the $k$-th derivative of $g$.

\subsection{Simple properties of the optimization problem}

The domain of the functions~$\Bell_{p,r;\eps}^\pm$ introduced in~\eqref{041102} and in~\eqref{041103} 
is described in terms of the functions~$\Bell^\pm_{p;\eps}$ from~\eqref{OldBellman} and~\eqref{OldBellman-}.
\begin{Prop}\label{Prop121201}
The set
\eqlb{DefOmega}{
\Omega_\eps^3=\Set{x\in\R^3}{(x_1,x_2)\in\Omega^2_\eps,\;
x_3\in\big[\Bell^-_{p;\eps}(x_1,x_2),\Bell^+_{p;\eps}(x_1,x_2)\big]}
}
is the domain of the Bellman functions~$\Bell^\pm_{p,r;\eps}$.
\end{Prop}

We note that for $x \notin \Omega_\eps^3$ there are no test functions and we definitely have 
$\Bell^{\pm}_{p,r}(x) = \mp\infty$ in this case. On the other hand, for any $x \in \Omega_\eps^3$  
there is a test function $\vf$. The proof of Proposition~\ref{Prop121201} can be found in~\cite{SVZ} 
(Proposition~2.2 there).

In what follows by the skeleton of $\Omega_{\eps}^3$ we mean the set of points $\{(t, t^2, |t|^p) \colon t \in \R \}$. 

A function~$G\colon \omega \to \R\cup\{\pm \infty\}$, where~$\omega \subset \R^d$ is 
an arbitrary set, is called \emph{locally concave}/\emph{convex} if for any segment~$\ell \subset \omega$, 
the restricted function~$G|_\ell$ is concave/convex. 

We collect standard facts concerning Bellman functions of such kind. 

\begin{Prop}\label{BasicProperties}
\begin{enumerate}
\item[\textup{1)}] The functions~$\Bell^{\pm}_{p,r}$ satisfy the boundary conditions on the skeleton 
of~$\Omega_\eps^3$\textup:
\eqlb{eqBCsceleton}{
\Bell^{\pm}_{p,r}(t,t^2,|t|^p) = |t|^r, \qquad t \in \R.
}
\item[\textup{2)}] The function~$\Bell_{p,r}^+$ is locally concave and the
function~$\Bell_{p,r}^-$ is locally convex on~$\Omega_{\eps}^3$.
\item[\textup{3)}] The function~$\Bell_{p,r}^+$ is the pointwise minimal among all locally concave 
on~$\Omega_{\eps}^3$ functions~$B$ that satisfy the boundary condition~\eqref{eqBCsceleton}. The 
function~$\Bell_{p,r}^-$ is the pointwise maximal among all locally convex
on~$\Omega_{\eps}^3$ functions~$B$ that satisfy the boundary condition~\eqref{eqBCsceleton}.
\end{enumerate}
\end{Prop}

Proposition~\ref{BasicProperties} is proved in~\cite{SVZ} for the case of concave functions 
(see Proposition 2.3 there). The proof of the other case is literally the same. 

In view of Proposition~\ref{BasicProperties}, to find the Bellman function $\Bell^{\pm}_{p,r}$, it suffices to construct a $C^1$-smooth 
function~$B\colon\Omega^3_{\eps}\to\R$ such that
\begin{enumerate}
\item[(i)] the function $B$ is locally concave/convex on~$\Omega^3_\eps$;
\item[(ii)] the function $B$ fulfills the boundary conditions~\eqref{eqBCsceleton};
\item[(iii)] for any point~$x \in \Omega^3_\eps$ there is a function~$\vf_x\in\BMO(I)$ such that 
\eqlb{eqphi_x}{
\|\vf_x\|_{\BMO(I)}\leq\eps,\ \av{\vf_x}{I}=x_1,\ \av{\vf_x^2}{I}=x_2,\ \av{|\vf_x|^p}{I}=x_3,
\ \av{|\vf_x|^r}{I}=B(x);
}	
\end{enumerate}
If all of the above requirements hold, then~$B =\Bell_{p,r}^+$ for locally concave function $B$ and
$B =\Bell_{p,r}^-$ for locally convex function $B$. Indeed, the inequality 
$B(x)\ge\Bell^+_{p,r}(x)$, $x\in\Omega^3_\eps$, follows from  conditions (i) and (ii), 
and the third statement of Proposition~\ref{BasicProperties}. The reverse inequality 
$B(x)\le\Bell^+_{p,r}(x)$ for $x\in\Omega^3_\eps$ follows from
condition~(iii) and the definition of the Bellman function~$\Bell_{p,r}^+$. 
We will provide more details in Section~\ref{SecOptim} and Section~\ref{sec_proofthm21}.

A function~$\vf_x$ satisfying~\eqref{eqphi_x} is called an optimizer for~$B$ at~$x$.

\section{Foliation. Definitions}
\label{SectFoliation}

Our aim is to construct the function~$B$ on~$\Omega_\eps^3$ described at the end of the previous section. 
First, we build a foliation of the domain, i.\,e., we split the whole domain into a union of one-dimensional and two-dimensional sets, where our function is linear and its gradient is constant. These subsets of linearity we call {\it leaves} of the foliation.
Having a foliation we can reconstruct the Bellman function using the 
boundary values. Recall the construction performed in~\cite{SVZ} for the Bellman function $B_1$. We had 
there only two-dimension leaves of foliation. We repeat here description of this foliation.
 
Domain~$\Omega_\eps^3$ was split into three subdomains $\Xi\cii+$, $\Xi\cii0$, $\Xi\cii-$ with foliation of 
different types: 
\eqlb{defXi}{
\begin{aligned}
\Xi\cii0 &=\Set{x\in\Omega_\eps^3}{|x_1|\leq 2\eps,\ 
x_2\geq4\eps|x_1|-3\eps^2\!,\ (p-2)\big(x_3-\eps^p-\frac{x_2-\eps^2}{4\eps}m_p(\eps)\big)\geq 0},
\\
\Xi\cii+ &= \{x \in \Omega_\eps^3 \setminus \Xi\cii0 \colon x_1>0\},
\\
\Xi\cii- &= \{x \in \Omega_\eps^3 \setminus \Xi\cii0 \colon x_1<0\}.
\end{aligned}
}

\subsection{Foliation far from the origin}

Since our problem is symmetric with respect to the change of sign of the first coordinate, we will assume 
that $x_1>0$.  We consider the subdomain $\Xi\cii+$. Let us denote 
by $U(v)$ the point of the skeleton with the first coordinate $v$: 
$U(v) = (v,v^2,v^p)$. Together with the point $U(v)$ we consider two points $W_\pm(v)$ that belong 
to the upper or lower boundary of $\Omega_\eps^3$ and are the second endpoints of tangents from 
the point~$U(v)$ to the parabolic part of the boundary, namely
\eqlb{Upm}{
\begin{split} 
W_+(v)=&\Big(v+\eps,(v+\eps)^2+\eps^2,\Am{p}\big(v+\eps,(v+\eps)^2+\eps^2\big)\Big)
\\
=&\Big(v+\eps,(v+\eps)^2+\eps^2, v^p+\eps m_p(v)\Big),
\\
W_-(v)=&\Big(v-\eps,(v-\eps)^2+\eps^2,\Ak{p}\big(v-\eps,(v-\eps)^2+\eps^2\big)\Big)
\\
=&\Big(v-\eps,(v-\eps)^2+\eps^2, v^p-\eps k_p(v)\Big).
\end{split}
}
Note that the point $W_+(v)$ is on the upper boundary if $p>2$ and on the lower boundary if $p<2$.
The converse position is taken by $W_-(v)$.
Note that the projection of the segments $[U(v),W_\pm(v)]$ onto the $x_1x_2$-plane
 are the segments $S_\pm(v)$. Here and in what follows, by $[{\rm A},{\rm B}]$ we denote the straight line segment with the endpoints ${\rm A}$ and ${\rm B}$.

For $v\geq\eps$ we consider the two-dimensional plane that passes through~$U(v)$ 
and the points $W_{\pm}(v)$. Its equation is
\eqlb{planeeq1}{
x_3 = v^p + \frac{m_p-k_p}{4\eps}\cdot\big(x_2-x_1^2+(x_1-v)^2\big)+\frac{m_p + k_p}{2}\cdot(x_1-v).
}
Here and in what follows, we omit the argument of~$m_p$ and~$k_p$ if this does not lead to ambiguity.

Let $T(v)$ be the intersection of $\Omega^3_\eps$ and the triangle with the vertices $U(v)$, $W_-(v)$, 
$W_+(v)$. So, $T(v)$ is a curvilinear triangle with linear sides on $\partial\Omega_\eps^3$: 
$[U(v),W_+(v)]$ is the graph of $\Am{p}$ restricted on $S_+(v)$ and 
$[U(v),W_-(v)]$ is the graph of $\Ak{p}$ restricted on $S_-(v)$. 
That is for $p>2$ the segment $[U,W_+]$ 
lies on the upper boundary of $\Omega_\eps^3$, 
for $p<2$ it 
lies on the lower boundary. For the segment $[U,W_-]$ we have the opposite situation. This difference of the cases $p<2$ and $p>2$ explains rather cumbersome definition of the domain 
$\Xi\cii0$: depending on $p$ we have to consider the points either above or below the triangle $T(\eps)$. 
The domain $\Xi\cii+$ is foliated by such curvilinear triangles $T(v)$, $v \geq \eps$.

This was the description of the foliation of $\Xi\cii+$ for $B_1$ found in~\cite{SVZ}. We start with some heuristic arguments 
concerning possible foliation for $B_2$. We need to find another 
foliation with the same boundary condition: on the boundary we assume the same extremal lines $[U(v),W_+(v)]$ 
and $[U(v),W_-(v)]$. The unique possibility to have a two-dimensional leaf of linearity is already 
used in the described foliation, therefore, we have to look for one dimensional extremals. All extremals 
have to start from the skeleton $\{x=(v,v^2,|v|^p)\}$ and cannot go transversal to the free parabolic 
boundary $\{x=(x_1,x_1^2+\eps^2,x_3)\}$. Therefore, there are only two possibilities: either a fan 
$\FL(u)$ of left tangents to parabolic boundary from a point $U(u-\eps)$, or a fan $\FR(u)$ of 
right tangents to parabolic boundary from a point $U(u+\eps)$. The tangency points of each such fan 
lie on the line $(u,u^2+\eps^2,\;\cdot\;)$. We would like to pay attention of the reader that symbols~$\mathrm L$ 
and~$\mathrm R$ mean left and right tangents, what means that these lines lies on the left and on the right 
of the tangency point, correspondingly. But if we look from their common point $U(u\pm\eps)$ the 
situation is opposite: the fan $\FL(u)$ lies on the right from $U(u-\eps)$ (its projection on $(x_1,x_2)$-plane 
is $S_+(u-\eps)$) and the fan $\FR(u)$ lies on the left of $U(u+\eps)$ (its projection on $(x_1,x_2)$-plane 
is $S_-(u+\eps)$). We illustrate all of this by Figure~\ref{3_u_infty}. 
\begin{figure}[!h]
\includegraphics[scale=0.5]{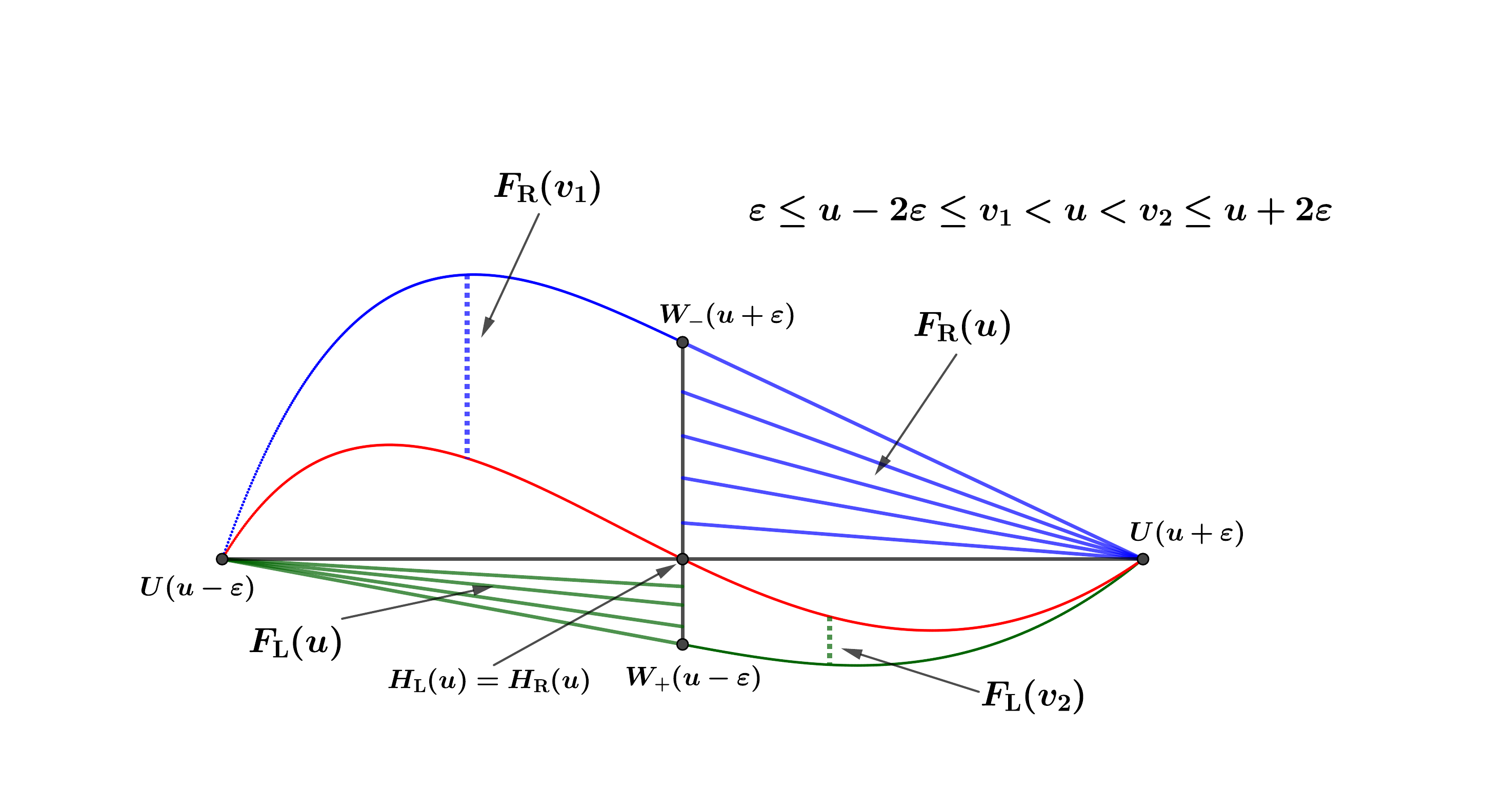}
\caption{Traces of extremals on the plane $P(u)$ far from the origin for $1<p<2$.}
\label{3_u_infty}
\end{figure}
To visualize our construction, it is more natural to consider what happens 
in a tangent plane to the parabolic boundary. Let us intersect our domain $\Omega_\eps^3$ by the 
two-dimensional plane $2ux_1-x_2-u^2+\eps^2=0$, we call this intersection $P(u)$. It touches the parabolic boundary along the vertical 
line $x_1=u$, $x_2=u^2+\eps^2$.  The picture of $P(u)$ on 
Figure~\ref{3_u_infty} illustrates the situation for the case $1<p<2$ and $u>3\eps$. The chord $[U(u-\eps),U(u+\eps)]$ has a very large slope for large values of $u$, but we place
it horizontally, i.\,e., this picture presents an affine transform of the true graph. 

We see here that $P(u)$ has two linear parts of the boundary. These are two straight line segments: the extremal \hbox{$[U(u-\eps),W_+(u-\eps)]$} 
on the upper boundary and the extremal $[U(u+\eps),W_-(u+\eps)]$ on the lower boundary. The fans 
$\FR(u)$ and $\FL(u)$ start from these two extremals. The ``long chord'' $[U(u-\eps),U(u+\eps)]$ is 
a natural second border for both fans. The curvilinear line on the Figure~\ref{3_u_infty} connecting $U(u-\eps)$ and $U(u+\eps)$ is 
the trace of the other long chords on this plane. That is the set of points of intersection of the chords 
$[U(u-\eps+t),U(u+\eps+t)]$, $-2\eps<t<2\eps$, with the plane $P(u)$. In each plane $P(u+t)$ there are
own fans $\FR(u+t)$ and $\FL(u+t)$ that intersect our plane $P(u)$ along some
vertical segments between the boundary point and the corresponding long chord. So, we have described the extremals for
all points of $P(u)$ except of ones between the long chord $[U(u-\eps),U(u+\eps)]$ and the curvilinear line connecting its endpoints. It appears that the extremals for the points between these two lines are short chords of the form 
$$[U(a), U(b)], \quad 0 < b-a < 2\eps,$$
passing through
them, i.\,e., the chords that do not touch the parabolic boundary. In Lemma~\ref{080508} below it will be proved that 
for each such point there exists exactly one chord passing through it. Observe that for $x \in [U(u-\eps),U(u+\eps)]$ we have the following relation:
\begin{equation*}
x_3 = \begin{cases} \displaystyle\frac{\Delta_-(x_1-\Delta_+)^p+\Delta_+(x_1+\Delta_-)^p}{2\eps}, & \text{ if } \  u \leq x_1 \leq u + \eps,\\
\rule{0pt}{23pt}
\displaystyle \frac{\Delta_-(x_1+\Delta_+)^p+\Delta_+(x_1-\Delta_-)^p}{2\eps}, & \text{ if  } \  u - \eps \leq x_1 \leq u,
\end{cases}
\end{equation*}
where $\Delta_\pm\df\eps\pm d$ and $d\df\sqrt{x_1^2+\eps^2-x_2}$.

Now, instead of one domain $\Xi\ci+$ generated by the triangles $T(u)$, we have three domains: $\Xi\cii{\mathrm L+}$ 
is the domain foliated by the fans $\FL(u)$ of left tangents, $\Xi\cii{\mathrm R+}$ is the domain foliated by 
fans $\FR(u)$ of right tangents, and $\Xi\cii{\mathrm{ch}+}$ is the domain between two preceding ones 
foliated by chords. So, approximately, we can describe 
these domains as follows:
\begin{align}
\notag
\Xi\cii{\mathrm L+} = &\Big\{x\in\Omega_\eps^3\colon x_1>0\ \text{and sufficiently far from 0,}
\\
\notag
&\quad(p-2)\frac{\Delta_-(x_1+\Delta_+)^p+\Delta_+(x_1-\Delta_-)^p}{2\eps}\leq(p-2)x_3
\leq(p-2)\Am{p}(x_1,x_2)\Big\};
\\
\label{defXi1}
\Xi\cii{\mathrm R+} = &\Big\{x\in\Omega_\eps^3\colon x_1>0\ \text{and sufficiently far from 0,}
\\
\notag
&\quad(p-2)\Ak{p}(x_1,x_2)\leq
(p-2)x_3\leq(p-2)\frac{\Delta_-(x_1-\Delta_+)^p+\Delta_+(x_1+\Delta_-)^p}{2\eps}\Big\};
\\
\notag
\Xi\cii{\mathrm{ch}+} = &\Big\{x\in\Omega_\eps^3\colon x_1>0\ \text{and sufficiently far from 0,}
\\
\notag
&\quad(p-2)\frac{\Delta_-(x_1\!-\!\Delta_+)^p+\Delta_+(x_1\!+\!\Delta_-)^p}{2\eps}\leq
(p-2)x_3\leq(p-2)\frac{\Delta_-(x_1\!+\!\Delta_+)^p+\Delta_+(x_1\!-\!\Delta_-)^p}{2\eps}\Big\}.
\end{align}

It is clear that we have domains $\Xi\cii{\mathrm R-}$ , $\Xi\cii{\mathrm L-}$, and 
$\Xi\cii{\mathrm{ch}-}$ symmetrical to $\Xi\cii{\mathrm L+}$ , $\Xi\cii{\mathrm R+}$, and 
$\Xi\cii{\mathrm{ch}+}$ respectively. 
And now the problem arises how to connect these domains in a neighbourhood 
of the origin and specify what does it mean ``sufficiently far'' in these formulas. Furthermore, near 
the origin, the description of all these three domains should be slightly changed due to the influence 
of two new specific domains that appear here. The traces of all these domains near the origin can be seen on Figures~\ref{2_u_3}, \ref{1_u_2}, and \ref{0_u_1}.

\begin{figure}[!h]
\includegraphics[scale=0.5]{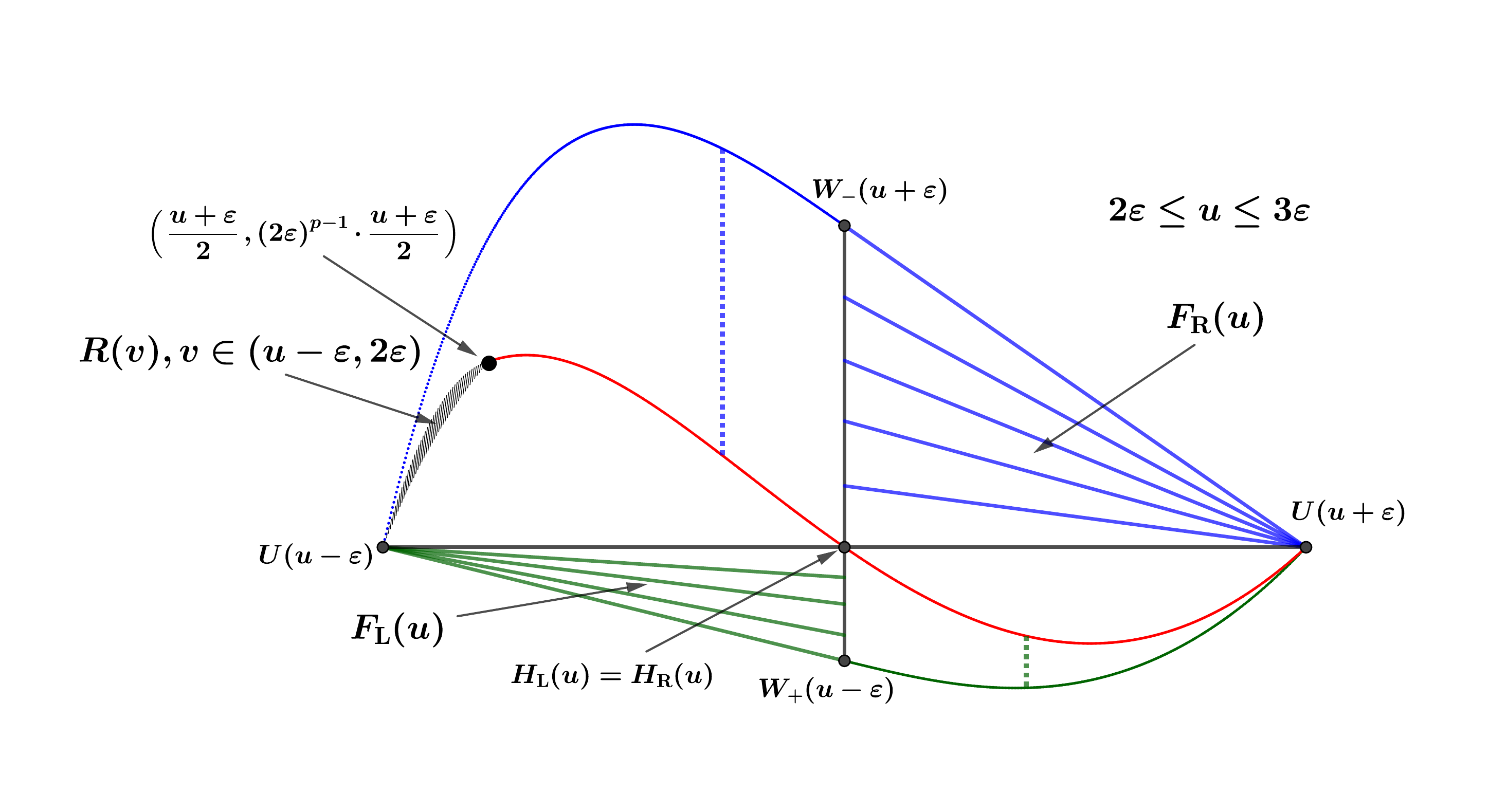}
\caption{Traces of extremals on the plane $P(u)$ for $2\eps<u<3\eps$ and $1<p<2$.}
\label{2_u_3}
\end{figure}
\begin{figure}[!h]
\includegraphics[scale=0.5]{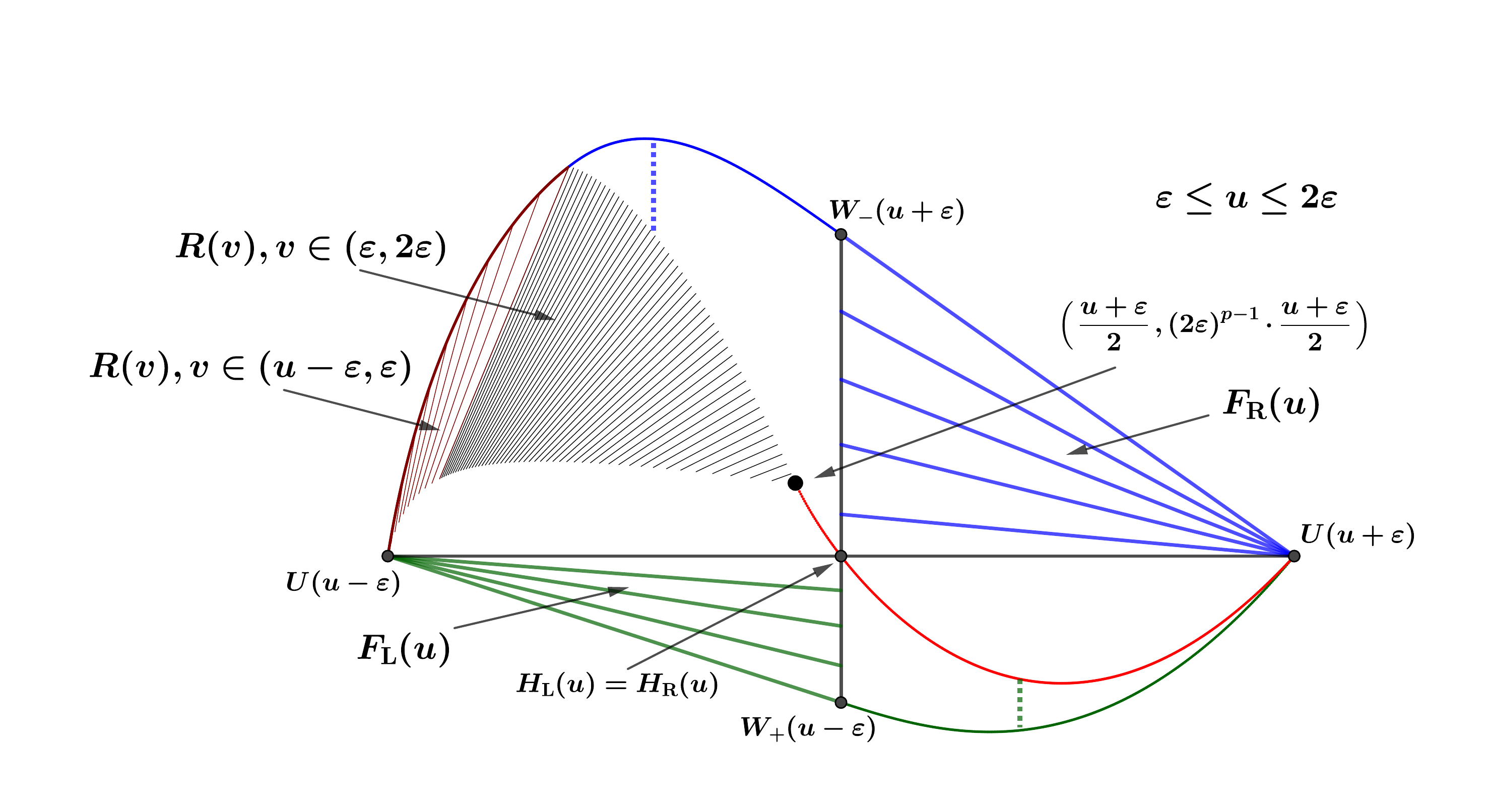}
\caption{Traces of extremals on the plane $P(u)$ for $\eps<u<2\eps$ and  $1<p<2$.}
\label{1_u_2}
\end{figure}
\begin{figure}[!h]
\includegraphics[scale=0.3]{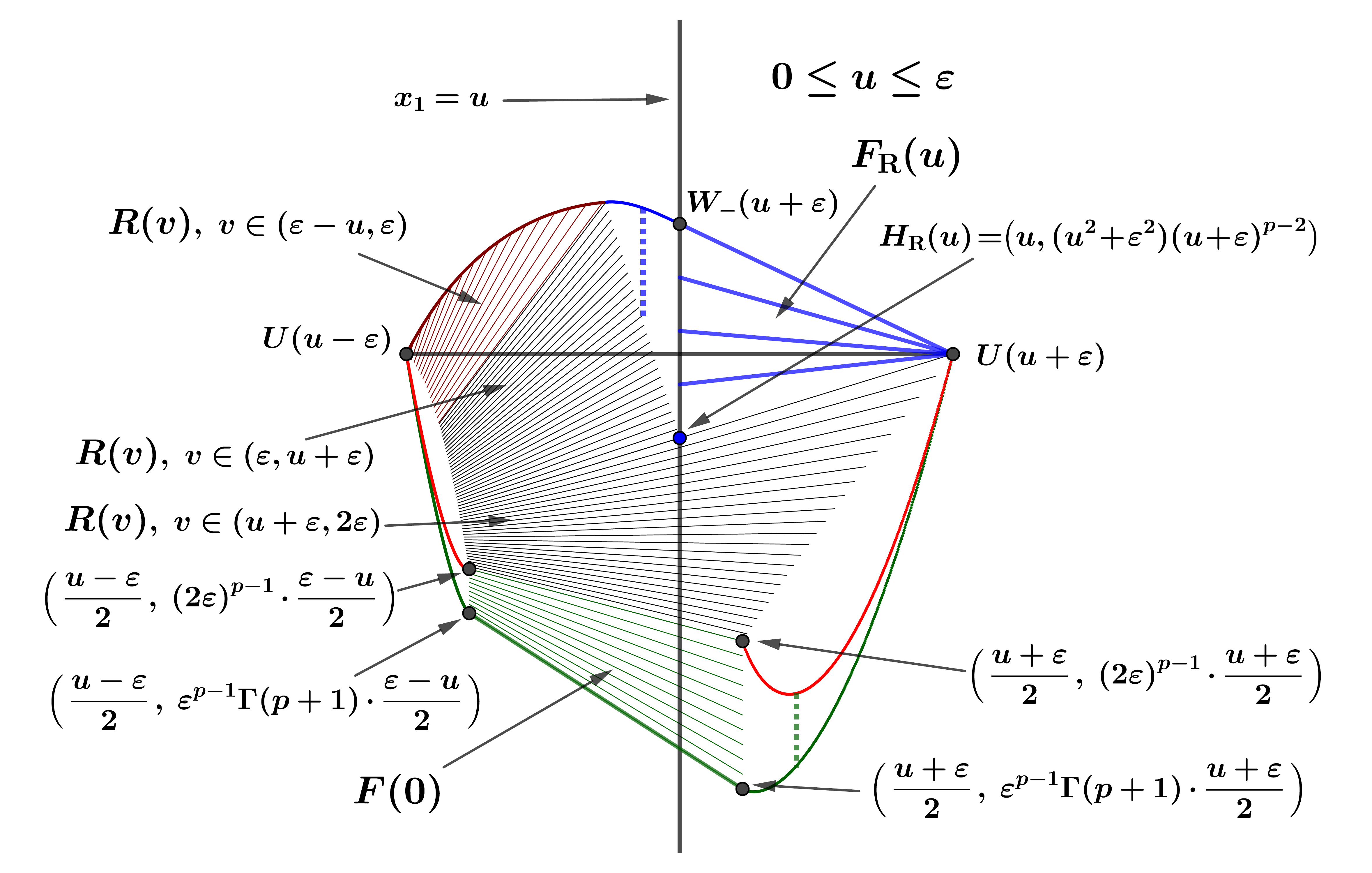}
\caption{Traces of extremals on the plane $P(u)$ for $0<u<\eps$ and  $1<p<2$.}
\label{0_u_1}
\end{figure}

\subsection{Foliation near the origin}
\label{FolNearZero}

In~\cite{SVZ} the domain $\Xi\cii0$ for $B_1$ (see~\eqref{defXi}) was obtained in a rather natural way: when $u$   
decreases, the triangle $T(u)$ and the symmetrical triangle $T(-u)$ get closer and closer, and
at the moment $u=\eps$ they touch each other at the point $(0,\eps^2,\eps^p)$. At this moment they 
lie in the same plane and two of their sides form a single chord $[U(-\eps),U(\eps)]$. 
\footnote{For the reader familiar with~\cite{ISVZ} we give the following analogy. We know that in two-dimensional cases,
when two angles simultaneously touch a cup, they form a birdie.
Here in the three-dimensional case, we get a family of ``birdies''.}
For each $u$, $u<\eps$, we have a two-dimensional domain of 
linearity in the form of curvilinear trapezoid $\tilde T(u)$ with three straight line sides $[U(-u),U(u)]$, 
$[U(u),W_+(u)]$, and $[U(-u),W_-(-u)]$.

Now, for $B_2$, the situation is much less clear, the question is how to gather nearly the origin three domains from the right and 
three domains from the left. It appears that the domains foliated by the chords~$\Xi\cii{\mathrm{ch}+}$ and $\Xi\cii{\mathrm{ch}-}$
are separated automatically, they have the only common point --- the origin. It is rather easy to build an interlacing 
domain between $\Xi\cii{\mathrm L+}$ and $\Xi\cii{\mathrm R-}$. Each of these domains has 
a border fan $\FL(\eps)$ and $\FR(-\eps)$. They have the common point $U(0)$. If we take 
an extremal from $\FL(\eps)$ and the symmetric extremal from $\FR(-\eps)$ it is natural 
to consider a curvilinear triangle between these two extremals as a domain of linearity. In such a way 
we obtain at the origin a fan of two-dimensional extremals. We denote this fan simply by $F(0)$ and 
it is just the interlacing domain between $\Xi\cii{\mathrm L+}$ and~$\Xi\cii{\mathrm R-}$.

The construction of the interlacing domain between $\Xi\cii{\mathrm L-}$ and $\Xi\cii{\mathrm R+}$ is more 
difficult. It appears that the connecting foliation consists of two-dimensional linearity domains 
$R(v)$, $0<v\leq2\eps$, describing as follows: the leaf $R(v)$ is the induced (in $\Omega^3_\eps$) 
convex hull\footnote{We refer to a set $X$, $X\subset\Omega$, as the induced convex set in $\Omega$ if for any pair $x,y \in X$ such that the straight line segment $[x,y]$ lies in $\Omega$, it also lies in $X$. We say that $X$ is an induced convex hull of $X_1$ in $\Omega$ if it is the minimal by inclusion induced convex subset of $\Omega$ containing $X_1$.} of the points $U(0)$ and $U(\pm v)$.
\begin{figure}[!h]
\begin{center}
\includegraphics{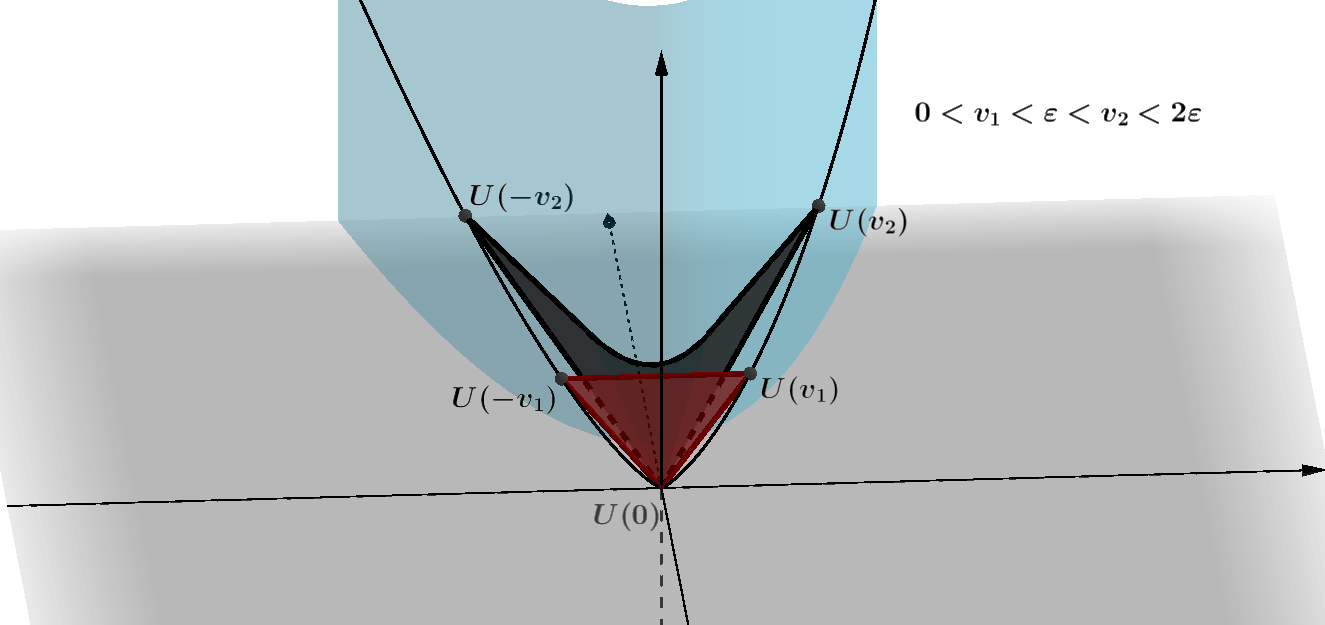}
\caption{Two examples of two-dimensional extremals $R(v)$.}
\label{R(v)_domain}
\end{center}
\end{figure}
Two types of linearity domains $R(v)$ are presented on Fig.~\ref{R(v)_domain}:
\begin{itemize}
\item if $0<v\leq\eps$, then $R(v)$ is simply the convex hull of these points, i.\,e., 
it is the triangle with these vertices;
\item if $\eps\leq v\leq2\eps$, then $R(v)$ is the induced convex hull of these points, 
i.\,e., it is the plane curvilinear triangle with two sides being the chords $[U(0),U(\pm v)]$, and 
the third side consist of two symmetrical line straight segments starting from $U(\pm v)$ and 
tangent to the parabolic boundary at the points with the first coordinates $\pm(v-\eps)$ and 
a curve on the parabolic boundary connecting these two tangency points.  
\end{itemize}

Now we give a formal description of all subdomains.

The fan of two-dimensional extremals foliates the domain
\eqlb{F0}{
F(0)=\Big\{x\in\Omega_\eps^3\colon|x_1|\leq\eps,\ 2\eps|x_1|\leq x_2\leq x_1^2+\eps^2,
\quad (2-p)x_3\leq(2-p)(2\eps)^{p-2}x_2\Big\}.
}

We denote by $R$ the domain foliated by two-dimensional leaves $R(v)$. To describe this domain
analytically, we first split the underlying domain $\Omega^2_\eps$ into the following subregions:
\eqlb{041101}{
\begin{aligned}
\omega_0=&\Big\{x\in\Omega_\eps^2\colon 2\eps|x_1|\le x_2\le\eps^2\Big\};
\\
\omega_1=&\Big\{x\in\Omega_\eps^2\colon x_2\ge\eps^2,\ x_2\ge2\eps x_1,\ 0\le x_1\le\eps\Big\};
\\
\omega_2=&\Big\{x\in\Omega_\eps^2\colon x_2\le\eps^2,\ x_2\le2\eps x_1\Big\};
\\
\omega_3=&\Big\{x\in\Omega_\eps^2\colon \eps^2\leq x_2\leq2\eps x_1\Big\};
\\
\omega_4=&\Big\{x\in\Omega_\eps^2\colon x_1\ge\eps,\ x_2\ge2\eps x_1\Big\};
\end{aligned}
}
and $\omega_{-i}=\Big\{x=(x_1,x_2)\colon (-x_1,x_2)\in\omega_i\Big\}$.
\begin{figure}[!h]
\begin{center}
\includegraphics[width=0.7\textwidth]{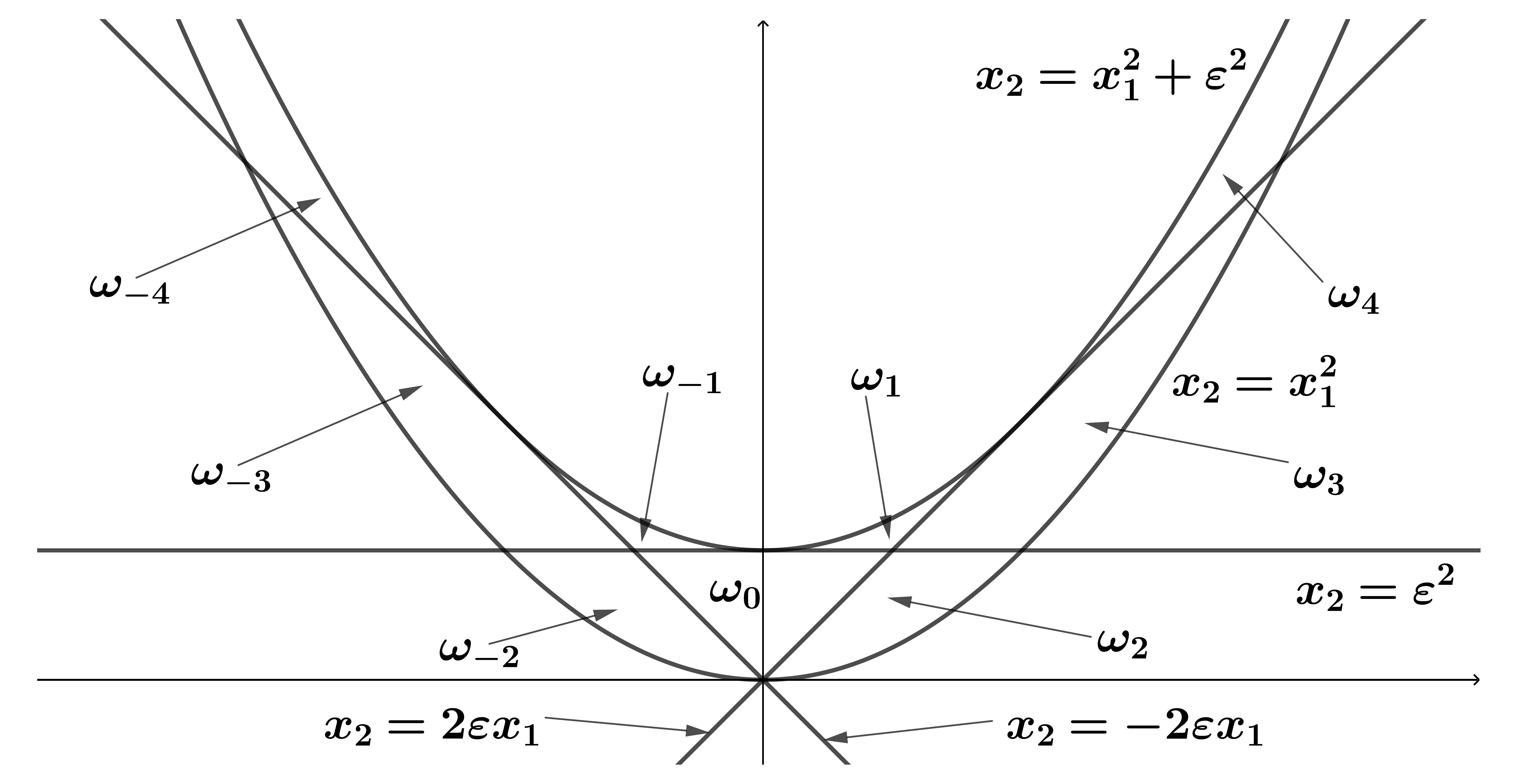}
\caption{Splitting of $\Omega^2_\eps$ into subdomains $\omega_i$.}
\label{omega_domains}
\end{center}
\end{figure}

The projection of the domain $R$ to the first two coordinates covers all subregions $\omega_i$
except $\omega_{\pm4}$.
\eqlb{R}{
R=\bigcup_{i=-3}^{i=3} R_i\,,
}
where
\begin{align}
\label{111201}
R_0=&\Big\{x\colon (x_1,x_2)\in\omega_0,\ 
(2-p)(2\eps)^{p-2}x_2\leq(2-p)x_3\leq(2-p)x_2^{\frac p2}\Big\};
\\
\label{111202}
R_{\pm1}=&\Big\{x\colon (x_1,x_2)\in\omega_{\pm1},\
(2-p)(2\eps)^{p-2}x_2\leq(2-p)x_3\leq(2-p)(|x_1|+\Delta_-)^{p-2}x_2\Big\};
\\
\label{111203}
R_{\pm2}=&\Big\{x\colon (x_1,x_2)\in\omega_{\pm2},\
(2-p)x_2^{p-1}|x_1|^{2-p}\leq(2-p)x_3\leq(2-p)x_2^{\frac p2}\Big\};
\\
\label{111204}
R_{\pm3}=&\Big\{x\colon (x_1,x_2)\in\omega_{\pm3},\ 
(2-p)x_2^{p-1}|x_1|^{2-p}\leq(2-p)x_3\leq(2-p)(|x_1|+\Delta_-)^{p-2}x_2\Big\}.
\end{align}
As before, we use the notation $\Delta_\pm = \eps\pm d$ and $d = \sqrt{x_1^2+\eps^2-x_2}$. The different parts of $R$ have different neighbours in vertical direction: $R_0$ has common boundary with $F(0)$, 
$R_1$ lies between $F(0)$ and $\Xi\cii{\mathrm R+}$,~$R_2$ has common boundary with $\Xi\cii{\mathrm{ch}+}$, 
$R_3$ lies between $\Xi\cii{\mathrm{ch}+}$ and $\Xi\cii{\mathrm R+}$.
We will check in Lemma~\ref{070500} that indeed the domain $R$ is foliated by the leaves $R(v)$, $0\leq v\leq2\eps$.

Now we can complete the definition of the domains~\eqref{defXi1}:
\eqlb{eq080501}{
\begin{aligned}
\Xi\cii{\mathrm L+} =\; &\Big\{x\in\Omega_\eps^3\colon x_1\geq\Delta_-,\
(2-p)\Am{p}(x_1,x_2)\leq(2-p)x_3
\\
&\quad\leq(2-p)\frac{\Delta_-(x_1+\Delta_+)^p+\Delta_+(x_1-\Delta_-)^p}{2\eps}\Big\}\,.
\end{aligned}
}
In Lemma~\ref{080504} we will check that the domain $\Xi\cii{\mathrm L+}$ is foliated by the fans $\FL(u)$, 
$\eps\leq u<\infty$.

Near the origin the domain foliated by the right fans should be changed a bit:
\eqlb{eq080502}{
\begin{aligned}
\Xi\cii{\mathrm R+} =\ &
\Bigg\{
x\in\Omega_\eps^3\colon x_1\geq0,\ x_2\geq\eps^2,\ 
(p-2)\Ak{p}(x_1,x_2)\leq(p-2)x_3 
\\
&\quad\leq(p-2)
\begin{cases}
(x_1+\Delta_-)^{p-2}x_2,&\quad\text{if }\ x_1\leq\Delta_+,\rule[-10pt]{0pt}{10pt}
\\
\displaystyle
\frac{\Delta_-(x_1-\Delta_+)^p+\Delta_+(x_1+\Delta_-)^p}{2\eps},
&\quad\text{if }\ x_1\geq\Delta_+
\end{cases}
\Bigg\}.
\end{aligned}
}
In Lemma~\ref{080506} we will check that the domain $\Xi\cii{\mathrm R+}$ is foliated by the fans $\FR(u)$, 
$0\leq u<\infty$. 

The remaining subdomain
\eqlb{eq080503}{
\begin{aligned}
\Xi\cii{\mathrm{ch}+} =\ &\Bigg\{x\in\Omega_\eps^3\colon x_1\geq\Delta_-,\ 
(2-p)\frac{\Delta_-(x_1+\Delta_+)^p+\Delta_+(x_1-\Delta_-)^p}{2\eps}\leq(2-p)x_3
\\
&\qquad\leq(2-p)
\begin{cases}
x_2^{p-1}x_1^{2-p},&\quad\text{if }\ x_1\leq\Delta_+,\rule[-10pt]{0pt}{10pt}
\\
\displaystyle
\frac{\Delta_-(x_1-\Delta_+)^p+\Delta_+(x_1+\Delta_-)^p}{2\eps},
&\quad\text{if }\ x_1\geq\Delta_+
\end{cases}\Bigg\}
\end{aligned}
}
is foliated by the chords connecting two points of the skeleton.
In Lemma~\ref{080508} we will check that for every point $x$ from the domain $\Xi\cii{\mathrm{ch}+}$ there exists
a unique pair of non-negative numbers $a=a(x)$ and $b=b(x)$ such that $0\leq b-a\leq2\eps$ and the chord 
$[U(a),U(b)]$ passes through $x$.

In the latter two cases we have two different analytic expressions describing the domains, because the different
parts of the domains in~\eqref{eq080502} and~\eqref{eq080503}  have different neighbours. 

We collect the proofs of all the results concerning the described foliation in Section~\ref{FoliationProofs}.

\section{Construction of a Bellman candidate}
\label{Bellman_candidate_def}

We start with the simplest domain $\Xi\cii{\mathrm{ch}+}$ foliated by chords of the form $[U(a),U(b)]$. Since the
Bellman function is assumed to be linear on these chords, we can define a Bellman candidate $B$ by 
the formula
\eqlb{eq080509}{
B(x)=\frac{x_1-a}{b-a}b^r+\frac{b-x_1}{b-a}a^r.
}

It is also easy to find a Bellman candidate on the domain $R$ foliated by two-dimensional leaves 
$R(v)$, that are induced convex hulls of the points $U(\pm v)$ and the origin, because the linear function there is also determined completely by the values at these three 
points:
\eqlb{eq070502}{
B(x)=v^{r-2}x_2=x_3^\frac{r-2}{p-2}x_2^\frac{p-r}{p-2},\quad x\in R(v).
}

Formally, it is easy to define a candidate on the fan of triangles $F(0)$: the function there depends
only on $x_2$ and $x_3$ and is completely determined by the values on the sides of each triangle.
However, the values on the boundary fans $\FL(\eps)$ and $\FR(-\eps)$ should be obtained from the
domains $\Xi\cii{\mathrm L+}$ and $\Xi\cii{\mathrm R-}$, where it is not so simple to find an appropriate
Bellman candidate. Due to the symmetry, it is sufficient to consider only one of these domains, and
we start with $\Xi\cii{\mathrm L+}$.

\subsection{Bellman function in the domain $\Xi\cii{\mathrm L+}$}
\label{sec_def_psiL}
Recall that the domain $\Xi\cii{\mathrm L+}$ is foliated by the fans $\FL(u)$, $u\ge\eps$
(see Fig.~\ref{3_u_infty},~\ref{2_u_3}, and~\ref{1_u_2}). Recall that all the extremals of the fan 
$\FL(u)$ lie in the plane $P(u)$ and are tangent to the parabolic part of the boundary at the points 
$(u,u^2+\eps^2, \;\cdot\;)$. Recall that the plane $P(u)$ is determined by the equation 
\eqlb{190901}{
x_2=2ux_1+\eps^2-u^2.
}
Since we consider the left tangents, $u=u(x)\ge x_1$, i.\,e.,
\eqlb{021101}{
u=x_1+\sqrt{x_1^2+\eps^2-x_2}\,.
}
The tangent lines forming the fan $\FL(u)$ have the common endpoint $U(v)$, $v=u-\eps=x_1-\Delta_-$. 
The second endpoint of the extremal passing through $x$ will be denoted by $H(u,h)$, where $h=h(x)$ 
is determined by the third coordinate of this endpoint by the formula
\eqlb{210902}{
H(u,h)=(u,u^2+\eps^2,v^p+\eps m_p(v)+h).
}
This second endpoint runs over a segment $[W_+(v),\HL(u)]$ on the parabolic part of the boundary, where
\eqlb{210901}{
\HL(u)=[W_+(u-\eps),W_-(u+\eps)]\cap[U(u-\eps),U(u+\eps)]=\Big(u,u^2+\eps^2,\half\big((u+\eps)^p+(u-\eps)^p\Big),
}
and $W_\pm$ were defined in~\eqref{Upm}. 
This means that $h$ runs from $0$ till $\half\big((u+\eps)^p-(u-\eps)^p\big)-\eps m_p(u-\eps)$:
$$
H(u,0)=W_+(u-\eps),\qquad
H\Big(u,\half\big((u+\eps)^p-(u-\eps)^p\big)-\eps m_p(u-\eps)\Big)=\HL(u).
$$
We see that $\sign h=\sign(2-p)$, because the point $W_+(v)$ is on the lower boundary of the domain
if~$p<2$ and on the upper boundary if $p>2$.

Since $B$ has to be linear along the extremal $[U(v),H(u,h)]$, we write
\eqlb{eq3}{
B(x) = v^r+\KL(u,h)(x_1-v),\qquad x\in[U(v),H(u,h)],
}
where $\KL=\KL(u,h)$ is the unknown slope of this linear function. Now, we find the slope $\KL$
using the property of $\grad B$ to be constant along the extremal. First, we have to calculate
the partial derivatives of the variables $u$ and $h$.

The function $u=u(x)$ is defined by~\eqref{190901}, $v=u-\eps=x_1-\Delta_-$, whence
\eqlb{eq4}{
\diff{v}{x_2}=\diff{u}{x_2}=\frac1{2(x_1-u)};\qquad\quad\diff{v}{x_3}=\diff{u}{x_3}=0.
}
The function $h=h(x)$ is determined by the fact that the points $H(u,h)$, $x$, and $U(v)$ lie 
on the same straight line, i.\,e.,
\eqlb{eq5}{
\frac{x_1-v}{\eps}=\frac{x_3-v^p}{\eps m_p(v)+h}.
}
This yields
\eqlb{161001}{
h=\eps\Big(\frac{x_3-v^p}{x_1-v}-m_p(v)\Big),
}
and
\eqlb{291001}{
x_3=v^p+\Big(\frac h\eps+m_p(v)\Big)(x_1-v)=(u-\eps)^p+\Big(\frac h\eps+m_p(u-\eps)\Big)\Delta_-.
}

Since $v$ depends only on $x_1$ and $x_2$, we have
\eqlb{eq6}{
\diff{h}{x_3}=\frac{\eps}{x_1-v}; \qquad\quad
\diff{h}{x_2}\buildrel{\eqref{m-diff}}\over=\frac{h+\eps m_p'(v) (u-x_1)}{x_1-v}\cdot \diff{v}{x_2}.
}

Differentiating the function $B$ with respect to $x_3$, we obtain
\eqlb{eq7}{
\diff{B}{x_3}=\diff{\KL}{h}\cdot(x_1-v)\cdot\diff{h}{x_3}=\eps\diff{\KL}{h}.
}
The expression for the derivative of $B$ with respect to $x_2$ is more complicated:
\eqlb{eq8}{
\begin{aligned}
\diff{B}{x_2}&=\frac{1}{2(x_1-u)}\cdot
\Big[rv^{r-1}+\diff{\KL}{u}\cdot(x_1-v)-\KL+\diff{\KL}{h}\cdot\big(h+\eps m_p'(v)(u-x_1)\big)\Big]
\\
&=\frac{1}{2(x_1-u)}\cdot\Big[rv^{r-1}+\eps\diff{\KL}{u}-\KL+h\diff{\KL}{h}\Big]
+\frac12\Big(\diff{\KL}{u}-\eps m_p'(v)\diff{\KL}{h}\Big).
\end{aligned}
}
Since the derivative of the function $B$ with respect to $x_2$ has to be constant on the 
extremal line with the fixed parameters $(u,h)$, we conclude that 
\eqlb{eq9}{
rv^{r-1}+\eps\diff{\KL}{u}-\KL+h\diff{\KL}{h}=0
}
and
\eqlb{190903}{
\diff{B}{x_2}=\frac12\Big(\diff{\KL}{u}-\eps m_p'(v)\diff{\KL}{h}\Big).
}

From the formula for the function $B$ at the point $W_+(v)$, we get the boundary 
value $\KL(u,0)=m_r(v)$. The general solution of the differential equation~\eqref{eq9} with this
boundary condition has the following form:
\eqlb{eq10}{
\KL(u,h)=e^{\frac{u}{\eps}}\PsiL(e^{-\frac{u}{\eps}}h)+m_r(v),
} 
where $\PsiL$ is an arbitrary sufficiently smooth function with $\PsiL(0) = 0$. We remind the reader that in this subsection we set $v=u-\eps$.

We find $\PsiL$ from the boundary value at the point $\HL(u)$. Recall that $\HL(u)$ coincides with $H(u,h)$ for
$$
h=\half\big((u+\eps)^p-(u-\eps)^p\big)-\eps m_p(u-\eps).
$$
Since we assume the segment $[U(u-\eps),U(u+\eps)]$ to be an extremal line, $B$ 
has to be linear on it, i.\,e., $B=\half\big((u-\eps)^r+(u+\eps)^r\big)$ at the midpoint 
of this segment. As a result, we come to the equation
$$
B\big(\HL(u)\big)=(u-\eps)^r+\KL\Big(u,\half\big((u+\eps)^p-(u-\eps)^p\big)-\eps m_p(v)\Big)\cdot\eps
=\half\big((u+\eps)^r+(u-\eps)^r\big),
$$
whence
\eqlb{KLH_and_HLH}{
\KL\Big(u,\half\big((u+\eps)^p-(u-\eps)^p\big)-\eps m_p(v)\Big)=\frac{(u+\eps)^r-(u-\eps)^r}{2\eps},
}
or
\eqlb{220901}{
\PsiL\Big(e^{-\frac{u}\eps}\big[\half\big((u+\eps)^p-(u-\eps)^p\big)-\eps m_p(u-\eps)\big]\Big)
=\frac1\eps e^{-\frac{u}\eps}\big[\half\big((u+\eps)^r-(u-\eps)^r\big)-\eps m_r(u-\eps)\big].
}
Let us introduce a function $\xi\mapsto\wL(\xi; s,\eps)$ by the formula
\eqlb{w_L_def}{
\wL(\xi;s,\eps) = e^{-\frac\xi\eps}\big[\half\big((\xi+\eps)^s-(\xi-\eps)^s\big)-\eps m_s(\xi-\eps)\big],
\qquad \xi\ge\eps.
}
The symbols $s$ and $\eps$ are considered here as fixed parameters, we will sometimes omit them when it does not lead to misunderstanding. Then we can rewrite 
relation~\eqref{220901} in terms of $\wL$:
\eqlb{190904}{
\PsiL\big(\wL(u;p,\eps)\big)=\frac1\eps\wL(u;r,\eps\big).
} 
In Lemma~\ref{lemmaSignW_r} below  we will prove that the function $\wL$ is monotone, and thereby,
the inverse function is correctly defined.

This representation suggests the following change of parametrization for the extremals of the fan~$\FL(u)$.
Till now, they were parametrized by $h$. Let us introduce a new parameter $\xi$:
\eqlb{021104}{
\xi=\wL^{-1}(e^{-\frac u\eps}h;p,\eps),
}
i.\,e.,
\eqlb{eq081002}{
h=e^{\frac u\eps}\wL(\xi;p,\eps).
}
When the variable $\xi$ is running from $u$ till $\infty$, then $h$ is running from
$\half\big((u+\eps)^p-(u-\eps)^p\big)-\eps m_p(u-\eps)$ till zero. Therefore, the range of $\wL$ 
covers the domain of $\PsiL$, and we can write down the following expression for $\PsiL$:
\eqlb{Psi_L_def}{
\PsiL(\;\cdot\;)=\frac1\eps\wL^{\phantom1}\big(\wL^{-1}(\;\cdot\;;p,\eps);r,\eps\big).
}

Rewriting~\eqref{291001} in terms of $\xi$ we get:
\eqlb{291002l}{
x_3=v^p+\Big(\frac1\eps e^{\frac u\eps}\wL(\xi;p,\eps)+m_p(v)\Big)(x_1-v).
}
Taking into account~\eqref{190904}, we rewrite~\eqref{eq3} as follows:
\eqlb{291003l}{
B(x)=v^r+\Big(\frac1\eps e^{\frac u\eps}\wL(\xi;r,\eps)+m_r(v)\Big)(x_1-v).
}

At the end of this subsection, we collect some formulas for the derivatives of the function $B$. 
We will use them in various proofs. For brevity, we omit the argument 
$e^{-\frac{u}{\eps}}h$ of the function $\PsiL$ and its derivatives, as well as the argument $v$ of the 
functions $m_s$ and their derivatives. From~$\eqref{eq10}$ we see that
\eqlb{eq11}{
\diff{\KL}{h}=\PsiL'; \qquad\quad
\diff{\KL}{u}=\frac{1}{\eps}e^{\frac{u}{\eps}}\PsiL-\frac{h}{\eps}\PsiL'+m_r'.
}
Recall that $\frac{\partial B}{\partial x_3}$ was computed in~\eqref{eq7}. We have
\eqlb{050801}{
\frac{\partial B}{\partial x_3} = \eps\PsiL', \qquad\quad
\frac{\partial^2 B}{\partial x_3^2} = \frac{\eps^2}{x_1-v} e^{-\frac{u}{\eps}}\PsiL''.
}
Now, using relations~\eqref{eq11}, we rewrite formula~\eqref{190903} as follows
\eqlb{eq12}{
\diff{B}{x_2}=\frac12\Big(\frac1\eps e^{\frac{u}\eps}\PsiL-\frac{h}\eps\PsiL'+m_r'-\eps m_p'\PsiL'\Big).
}

Differentiating~\eqref{eq12}, we obtain
\begin{align}
\notag 
\frac{\partial^2 B}{\partial x_2^2}&=-\frac1{2\eps}\diff{h}{x_2}e^{-\frac{u}\eps}(h+\eps^2 m_p')\PsiL''+
\frac1{2\eps^2}\diff{u}{x_2}\Big(e^{\frac{u}\eps}\PsiL-(h+\eps^3 m_p'')\PsiL'+
he^{-\frac{u}\eps}(h+\eps^2 m_p')\PsiL''+\eps^2 m_r''\Big)
\\
\notag
&\buildrel{\eqref{eq6}}\over=\frac1{2\eps^2}\diff{u}{x_2}\Big(e^{\frac{u}{\eps}}\PsiL-(h+\eps^3 m_p'')\PsiL'+\eps^2 m_r''+
e^{-\frac{u}\eps}(h+\eps^2 m_p')\PsiL''\cdot\big(h-\eps\frac{h+\eps(u-x_1)m_p'}{x_1-v}\big)\Big)
\\
\label{eq14}
&\buildrel{\eqref{eq4}}\over=\frac1{4\eps^2(x_1-u)}\Big(e^{\frac{u}{\eps}}\PsiL-(h+\eps^3 m_p'')\PsiL'+\eps^2 m_r''
+e^{-\frac{u}\eps}(h+\eps^2 m_p')^2\PsiL''\cdot\frac{x_1-u}{x_1-v}\Big).
\end{align}

Differentiating~\eqref{eq12} with respect to $x_3$ and using~\eqref{eq6}, we get
\eqlb{180901}{
%\begin{aligned}
\frac{\partial^2 B}{\partial x_2\partial x_3}
% &=\eps\PsiL''\Big(-\frac1\eps e^{-\frac{u}\eps}\diff{u}{x_2}h+e^{-\frac{u}\eps}\diff{h}{x_2}\Big)
% \\
% &=\eps\PsiL''e^{-\frac{u}\eps}\Big(\frac{h+\eps(u-x_1)m_p'}{x_1-v}-\frac{h}\eps\Big)\frac1{2(x_1-u)}
% \\
=\frac{e^{-\frac{u}\eps}(h+\eps^2m_p')}{2(v-x_1)}\PsiL''.
%\end{aligned}
}
 
Collecting~\eqref{050801}, \eqref{eq14}, and~\eqref{180901}, we finally obtain
\eqlb{eq15}{
\det
\begin{pmatrix}
B_{x_2x_2}&B_{x_2x_3}
\\
B_{x_2x_3}&B_{x_3x_3}
\end{pmatrix} 
=\frac{e^{-\frac{u}\eps}\PsiL''}{4(x_1-u)(x_1-v)}
\Big(e^{\frac{u}\eps}\PsiL-(h+\eps^3 m_p'')\PsiL'+\eps^2 m_r''\Big).
}

\subsection{Bellman function in the domain $\Xi\cii{\mathrm R+}$}
\label{sec_def_psiR}

The construction of a Bellman candidate in $\Xi\cii{\mathrm R+}$ is quite similar to that for 
$\Xi\cii{\mathrm L+}$. The additional difficulty  appears here due to the fact that we have to 
distinguish the cases $0\le u\le\eps$ and $u\ge\eps$, compare Fig.~\ref{3_u_infty}, \ref{2_u_3}, 
\ref{1_u_2} with Fig.~\ref{0_u_1}. If $u>\eps$ then the situation is completely the same as in 
$\Xi\cii{\mathrm L+}$. However, if $u<\eps$ the corresponding long chord cannot be an extremal, 
because our function $\Bell_{p,r}^\pm$ is symmetric with respect to the plane $x_1=0$, but this chord intersects 
(not orthogonally) the plane of symmetry. For $u<\eps$ the fan $\FR(u)$ is continued till 
the domain $R$. This gives us another boundary condition for the corresponding function $\PsiR$. 

Since we consider right tangents, now we have $u=u(x)\le x_1$, and we choose the smaller root 
of~\eqref{190901},~i.\,e.,
\eqlb{021102}{
u=x_1-\sqrt{x_1^2+\eps^2-x_2}
}
(compare with~\eqref{021101}). The tangent lines forming the fan $\FR(u)$ have the common endpoint 
$U(v)$, $v=u+\eps=x_1+\Delta_-$. The second endpoint of the extremal passing through $x$ will be 
denoted as before by $H(u,h)$. Now $h=h(x)$ is determined by the third coordinates of this endpoint 
by the formula
\eqlb{210903}{
H(u,h)=(u,u^2+\eps^2,v^p-\eps k_p(v)+h)
}
(compare with~\eqref{210902}).
This second endpoint runs throw the segment $[W_-(v),\HR(u)]$ of the line $x=(u,u^2+\eps^2,x_3)$ 
of the parabolic part of the boundary, where
\eqlb{210904}{
\HR(u)=
\begin{cases}
[U(u-\eps),U(u+\eps)]\cap[W_-(u+\eps),W_+(u-\eps)],\quad&\text{ for \ }u\ge\eps, 
\\
\{x\colon x_1 = u,\ x_2=u^2+\eps^2\}\cap R(v),\quad&\text{ for \ }0\le u\le\eps,
\rule{0pt}{15pt}
\end{cases}
}
or in coordinates
\eqlb{210905}{
\HR(u)=
\begin{cases}
\left(u,u^2+\eps^2,\half\big((u+\eps)^p+(u-\eps)^p\big)\right), \quad&\text{ for \ }u\ge\eps, 
\\
\left(u,u^2+\eps^2,v^{p-2}(u^2+\eps^2)\right),\quad&\text{ for \ }0\le u\le\eps.
\rule{0pt}{15pt}
\end{cases}
}
This means that in the fan $\FR(u)$ the variable $h$ runs from $0$ till 
$\eps k_p(v)-\half\big((u+\eps)^p-(u-\eps)^p\big)$ for $u\ge\eps$ and till $\eps k_p(v)-2\eps uv^{p-2}$
for $0\le u\le\eps$.

We see that $\sign h=\sign(p-2)$, because the point $W_-(u)$ is on the upper boundary of the domain
if $p<2$ and on the lower boundary if $p>2$.

Since $B$ has to be linear along the extremal $[U(v),H(u,h)]$, we write
\eqlb{eq17}{
B(x) = v^r+\KR(u,h)(x_1-v),\qquad x\in[U(v),H(u,h)],
}

In the same way as we got~\eqref{eq5}, we obtain:
\eqlb{eq18}{
\frac{x_1-v}{-\eps} = \frac{x_3-v^p}{-\eps k_p(v)+h}.
}
This yields
\eqlb{011101}{
h=\eps\Big(k_p(v)-\frac{x_3-v^p}{x_1-v}\Big)
}
or
\eqlb{011102}{
x_3=v^p-\Big(\frac h\eps-k_p(v)\Big)(x_1-v)=(u+\eps)^p+\Big(\frac h\eps-k_p(u+\eps)\Big)\Delta_-.
}
We collect the expressions for the derivatives of $u$ and $h$:
\eqlb{eq19}{
\diff{u}{x_2}=\diff{v}{x_2}=\frac1{2(x_1-u)},\qquad\qquad\qquad\diff{u}{x_3}=\diff{v}{x_3}=0,
}
\eqlb{eq20}{
\diff{h}{x_2}\;{\buildrel\eqref{m-diff}\over=}\;\frac{h+\eps (x_1-u)k_p'(v)}{x_1-v}\cdot\diff{v}{x_2},\qquad
\diff{h}{x_3}=-\frac{\eps}{x_1-v}.
}
Similarly to~\eqref{eq7} and~\eqref{eq8}, we have
\eqlb{210906}{
\diff{B}{x_3}=-\eps\diff{\KR}{h},
}
\eqlb{210907}{
\begin{aligned}
\diff{B}{x_2}&=\frac{1}{2(x_1-u)}\cdot
\Big[rv^{r-1}+\diff{\KR}{u}\cdot(x_1-v)-\KR+\diff{\KR}{h}\cdot\big(h-\eps k_p'(v)(u-x_1)\big)\Big]
\\
&=\frac{1}{2(x_1-u)}\cdot\Big[rv^{r-1}-\eps\diff{\KR}{u}-\KR+h\diff{\KR}{h}\Big]
+\frac12\Big(\diff{\KR}{u}+\eps k_p'(v)\diff{\KR}{h}\Big).
\end{aligned}
}

The requirement for $\diff{B}{x_2}$ to be constant along the extremal line yields the following
differential equation (compare with~\eqref{eq9}):
\eqlb{eq2508202}{
rv^{r-1}-\eps\diff{\KR}{u}-\KR+h\diff{\KR}{h}=0.
}
As before, we know the value $B(W_-(v))=\Ak{r}(W_-(v))=v^r-\eps k_r(v)$, and it implies
the boundary condition $\KR(u,0)=k_r(v)$. The general solution of~\eqref{eq2508202} with
this boundary condition has the form
\eqlb{eq21}{
\KR(u,h)=e^{-\frac{u}{\eps}}\PsiR(e^{\frac{u}{\eps}}h)+k_r(v),   
}
where $\PsiR$ is arbitrarily sufficiently smooth function with $\PsiR(0)=0$ (compare with~\eqref{eq10}).
Identities~\eqref{210907} and~\eqref{eq2508202} imply
\eqlb{eq2508203}{
\diff{B}{x_2}=\frac12\Big(\diff{\KR}{u}+\eps k'_p(v)\diff{\KR}{h}\Big).
}

Now, we consider the cases $0\le u\le\eps$ and $u\ge\eps$ separately.

Let $0\le u\le\eps$. From formula~\eqref{eq070502} we have $B(\HR)=v^r-2\eps uv^{r-2}$,  
whence using~\eqref{eq17} we obtain
$$
\KR(u,\eps k_p(v)-2\eps uv^{p-2}) = 2uv^{r-2},
$$
or
\eqlb{220902}{
\PsiR\big(e^{\frac{u}\eps}\big[\eps k_p(v)-2\eps uv^{p-2}\big]\big)
=e^{\frac{u}\eps}\big[2uv^{r-2}-k_r(v)\big].
}

Now, let $u\ge\eps$. Similarly to the case of the function $\PsiL$, we have 
the endpoint $\HR$ at the middle of the extremal segment $[U(u-\eps),U(u+\eps)]$, where $B$ is linear, 
i.\,e., $B(\HR)=\half\big((u+\eps)^r+(u-\eps)^r\big)$. As a result, we come to the equation
$$
B\big(\HR(u)\big)=v^r+\KR\Big(u,\half\big((u-\eps)^p-(u+\eps)^p\big)+\eps k_p(v)\Big)\cdot(-\eps)
=\half\big((u+\eps)^r+(u-\eps)^r\big),
$$
$$
\KR\Big(u,\half\big((u-\eps)^p-(u+\eps)^p\big)+\eps k_p(v)\Big)=\frac{(u+\eps)^r-(u-\eps)^r}{2\eps},
$$
or
\eqlb{220903}{
\PsiR\Big(e^{\frac{u}\eps}\big[\half\big((u-\eps)^p-(u+\eps)^p\big)+\eps k_p(v)\big]\Big)
=\frac1\eps e^{\frac{u}\eps}\big[\half\big((u+\eps)^r-(u-\eps)^r\big)-\eps k_r(v)\big].
}
Let us introduce a function $\xi\mapsto\wR(\xi;s,\eps)$ by the formula
\eqlb{220904}{
\wR(\xi;s,\eps)=
\begin{cases}
\eps e^{\frac{\xi}{\eps}}\big[k_s(\xi+\eps)-2\xi(\xi+\eps)^{s-2}\big],\quad
&\text{ if \ }0\le\xi\le\eps,
\\
e^{\frac\xi\eps}\big[\half\big((\xi-\eps)^s-(\xi+\eps)^s\big)+\eps k_s(\xi+\eps)\big],\quad
&\text{ if \ }\xi\ge\eps.
\rule{0pt}{15pt}
\end{cases}
}
The symbols $s$ and $\eps$ are considered here as fixed parameters, we will sometimes omit them if it does not lead to ambiguity. Then we can rewrite 
relations~\eqref{220902} and~\eqref{220903} in terms of $\wR$:
\eqlb{220905}{
\PsiR\big(\wR(u;p,\eps)\big)=-\frac1\eps\wR(u;r,\eps\big).
} 
In Lemma~\ref{ww_increase} we prove that the function $\wR$ is monotone, and thereby, 
the inverse function is correctly defined.

This representation suggests the following change of parametrization for the extremals of the fan~$\FR(u)$.
Till now, they were parametrized by $h$. Let us introduce a new parameter $\xi$:
\eqlb{250703}{
\xi=\wR^{-1}(e^{\frac u\eps}h;p,\eps),
}
i.\,e.,
\eqlb{eq081003}{
h=e^{-\frac u\eps}\wR(\xi;p,\eps).
}
When the variable $\xi$ is running from zero till $u$, then $h$ is running from zero till
$\eps\big[k_p(v)-2uv^{p-2}\big]$ if $u\le\eps$ and till
$\half\big((u-\eps)^p-(u+\eps)^p\big)+\eps k_p(v)$ if $u\ge\eps$. Since the range of $\wR$ cover
the domain of $\PsiR$, we can write down the following expression for $\PsiR$:
\eqlb{Psi_R_def}{
\PsiR(\;\cdot\;)=-\frac1\eps\wR^{\phantom1}\big(\wR^{-1}(\;\cdot\;; p,\eps);r,\eps\big).
} 

Rewriting~\eqref{011102} in terms of $\xi$, we get:
\eqlb{291002r}{
x_3=v^p-\Big(\frac1\eps e^{-\frac u\eps}\wR(\xi;p,\eps)-k_p(v)\Big)(x_1-v).
}
Taking into account~\eqref{eq21} and~\eqref{220905}, we rewrite~\eqref{eq17} as follows:
\eqlb{291003r}{
B(x)=v^r-\Big(\frac1\eps e^{-\frac u\eps}\wR(\xi;r,\eps)-k_r(v)\Big)(x_1-v).
}

As at the end of the preceding subsection, we collect here some formulas for the derivatives of the function 
$B$ and calculate $\det\{B_{x_i x_j}\}_{2\leq i,j\leq3}$. As before, we omit the argument 
$e^{\frac{u}{\eps}}h$ of the function $\PsiR$ and its derivatives, as well as the argument $v$ of the 
functions $k_s$ and their derivatives.

Differentiating \eqref{eq21}, we get
\eqlb{eq2508204}{
\diff{\KR}{h}=\PsiR';\qquad\quad\diff{\KR}{v}
=-\frac1\eps e^{-\frac{u}\eps}\PsiR+\frac{h}\eps\PsiR'+k'_r.
}
We use these relations to rewrite~\eqref{210906} and~\eqref{eq2508203}:
\eqlb{250601}{
\diff{B}{x_3}=-\eps \PsiR'; \qquad\quad
\diff{B}{x_2}=\frac12\Big(k'_r-\frac1\eps e^{-\frac{u}\eps}\PsiR+\big(\frac{h}\eps+\eps k'_p\big)\PsiR'\Big).
}

Using these formulas together with~\eqref{eq19} and~\eqref{eq20}, we can check the following relations
\eqlb{220906}{
\frac{\partial^2 B}{\partial x_3^2} =\frac{\eps^2}{x_1 - v}e^{\frac{u}\eps}\PsiR'';
}
\eqlb{220907}{
%\begin{aligned}
\frac{\partial^2 B}{\partial x_2\partial x_3}
% &=-\eps\PsiR''\Big(\frac1\eps e^{\frac{u}\eps}\diff{v}{x_2}h+e^{\frac{u}\eps}\diff{h}{x_2}\Big)
% \\
% &=-\eps\PsiR''e^{\frac{u}\eps}\Big(\frac{h+\eps(x_1-u)k_p'}{x_1-v}+\frac{h}\eps\Big)\frac1{2(x_1-u)}
% \\
=\frac{e^{\frac{u}\eps}(h+\eps^2k_p')}{2(v-x_1)}\PsiR''
%\end{aligned}
}
(compare with~\eqref{180901});

\eqlb{240901}
{\begin{aligned}
\frac{\partial^2 B}{\partial x_2^2}
&=\frac1{2\eps}\diff{h}{x_2}e^{\frac{u}\eps}(h+\eps^2 k_p')\PsiR''+
\frac1{2\eps^2}\diff{u}{x_2}\Big(e^{-\frac{u}\eps}\PsiR-(h-\eps^3 k_p'')\PsiR'+
he^{\frac{u}\eps}(h+\eps^2 k_p')\PsiR''+\eps^2 k_r''\Big)
\\
&=\frac1{2\eps^2}\diff{u}{x_2}\Big(e^{-\frac{u}{\eps}}\PsiR-(h-\eps^3 k_p'')\PsiR'+\eps^2 k_r''+
e^{\frac{u}\eps}(h+\eps^2 k_p')\PsiR''\cdot\big(h+\eps\frac{h+\eps(x_1-u)k_p'}{x_1-v}\big)\Big)
\\
&=\frac1{4\eps^2(x_1-u)}\Big(e^{-\frac{u}{\eps}}\PsiR-(h-\eps^3 k_p'')\PsiR'+\eps^2 k_r''
+e^{\frac{u}\eps}(h+\eps^2 k_p')^2\PsiR''\cdot\frac{x_1-u}{x_1-v}\Big)
\end{aligned}
}
(compare with~\eqref{eq14}).

Finally, we obtain
\eqlb{290902}{
\det
\begin{pmatrix}
B_{x_2x_2}&B_{x_2x_3}
\\
B_{x_2x_3}&B_{x_3x_3}
\end{pmatrix} 
=\frac{e^{\frac{u}\eps}\PsiR''}{4(x_1-u)(x_1-v)}
\Big(e^{-\frac{u}\eps}\PsiR-(h-\eps^3 k_p'')\PsiR'+\eps^2 k_r''\Big)
}
(compare with~\eqref{eq15}).

Let us note that instead of calculating all these derivatives we could replace the subindex L by R,
$m$ by $k$, and $\eps$ by $-\eps$ in all formulas we deduced for the case of $\PsiL$.

\subsection{Correctness of definitions}

First, to ensure that the function $\PsiL$ is defined correctly we will show that $\wL$ is strictly 
monotone function.

\begin{Le}\label{lemmaSignW_r}
The equality $\sign \wL'(\xi;s,\eps)=\sign(s-2)$ holds for $\xi\ge\eps$ and $s\in(1,+\infty)$.
\end{Le}

\begin{proof}
Direct differentiation of $\wL$ from~\eqref{w_L_def} leads to
$$
\wL'=e^{-\frac\xi\eps}\Big(\!-\tfrac1{2\eps}\big((\xi+\eps)^s-(\xi-\eps)^s\big)+m_s(\xi-\eps)+
\half s\big((\xi+\eps)^{s-1}-(\xi-\eps)^{s-1}\big)-\eps m_s'(\xi-\eps)\Big).
$$
Applying relation \eqref{m-diff}, we get
\eqlb{dwu}{
\wL'=\tfrac1{2\eps} e^{-\frac\xi\eps}\Big(-(\xi+\eps)^s + \!(\xi-\eps)^s+s\eps(\xi+\eps)^{s-1}+s\eps(\xi-\eps)^{s-1}\Big).
}
We introduce the following function
\eqlb{051003}{
\AAA(\alpha,\beta,s) = 2\Big(\!(\alpha-\beta)^s-(\alpha+\beta)^s+s\beta(\alpha+\beta)^{s-1}+s\beta(\alpha-\beta)^{s-1}\Big),
}
which admits the integral representation
\eqlb{240403}{
\AAA(\alpha,\beta,s)=s(s-1)(s-2)\int\limits_{-\beta}^\beta(\beta^2-\lambda^2)(\lambda+\alpha)^{s-3}d\lambda\,.
}

The sign of $\wL'$ is clear from the relation
\eqlb{280401}{
\wL'(\xi; s,\eps) = \tfrac1{4\eps} e^{-\frac\xi\eps}\AAA(\xi,\eps,s).
}
\end{proof}

\begin{Le}
\label{ww_increase}
The function $\wR$ is $C^1$-smooth and the equality $\sign\wR'(\xi;s,\eps)=\sign(s-2)$ holds for 
$\xi\ge0$ and $s\in(1,+\infty)$.
\end{Le}

\begin{proof}
By the direct differentiation of 
$$
\wR=
\begin{cases}
\eps e^\frac\xi\eps\Big(k_s(\xi+\eps)-2\xi(\xi+\eps)^{s-2}\Big),\quad&\text{ if \ }0\le\xi\le\eps,
\\
e^\frac\xi\eps\Big(\half\big((\xi-\eps)^s-(\xi+\eps)^s\big)+\eps k_s(\xi+\eps)\Big),\quad&\text{ if \ }\xi\ge\eps
\rule{0pt}{15pt}
\end{cases}
$$
and the usage of~\eqref{m-diff}, we get
\eqlb{ww_dif}{
\wR'=
\begin{cases}
e^\frac\xi\eps(s-2)(\xi^2+\eps^2)(\xi+\eps)^{s-3},\quad&\text{ if \ }0\le\xi\le\eps,
\\
\tfrac1{2\eps} e^\frac\xi\eps\Big(\!s\eps(\xi-\eps)^{s-1}+s\eps(\xi+\eps)^{s-1}+(\xi-\eps)^s-(\xi+\eps)^s\Big),
\quad&\text{ if \ }\xi\ge\eps.
\rule{0pt}{15pt}
\end{cases}
}

Continuity of $\wR'$ immediately follows from this formula, and $\wR'(\eps;s,\eps)=e(s-2)2^{s-2}\eps^{s-1}$.
The statement $\sign\wR'=\sign(s-2)$ is clear directly from the first line of~\eqref{ww_dif} if $0\le\xi\le\eps$.
If $\xi\ge\eps$, then comparing~\eqref{ww_dif} with~\eqref{dwu} we note that
\eqlb{w_ww_relation}{
\wR'(\xi)=e^{\frac{2\xi}\eps}\wL'(\xi),\quad\xi\ge\eps.
}
Applying Lemma~\ref{lemmaSignW_r}, we finish the proof.
\end{proof}

\subsection{Definitions. Summary}
\label{Definitions}

In this subsection we collect all formulas defining our Bellman function (we call it $B_2$)
together with the formulas for the Bellman function (we call it $B_1$) constructed in~\cite{SVZ}.

First we recall notation used in the description of foliations.
\begin{itemize}
\item $U(v)$ is the point at the skeleton with the first coordinate $v$, i.\,e., $U(v)=(v,v^2,|v|^p)$.

\item $W_\pm(v)$ are the following points
\eqlb{061002}{
\begin{split} 
W_+(v)=&\Big(v+\eps,(v+\eps)^2+\eps^2,v^p+\eps m_p(v)\Big),\quad v\ge0,
\\
W_-(v)=&\Big(v-\eps,(v-\eps)^2+\eps^2,v^p-\eps k_p(v)\Big),\quad v\ge\eps,
\end{split}
}
and
\begin{align*}
W_+(v)=&\Big(v+\eps,(v+\eps)^2+\eps^2,|v|^p-\eps k_p(|v|)\Big),\quad v\le-\eps,
\\
W_-(v)=&\Big(v-\eps,(v-\eps)^2+\eps^2,|v|^p+\eps m_p(|v|)\Big),\quad v\le0.
\end{align*}

\item The two-dimensional linearity domain $T(v)$ is the induced (in $\Omega^3_\eps$) convex hull of $U(v)$, $W_\pm(v)$, i.\,e., it is the curvilinear triangle being the intersection 
of the triangle whose vertices are $U(v)$, $W_\pm(v)$ with the domain $\Omega^3_\eps$.

\item The two-dimensional linearity domain $\tilde T(v)$, $0<v<\eps$, is the induced 
(in $\Omega^3_\eps$) convex hull of the points $U(\pm v)$, $W_\pm(\pm v)$, i.\,e., it is the curvilinear trapezoid being the intersection 
of the usual trapezoid whose vertices are $U(\pm v)$, $W_\pm(\pm v)$ with the domain $\Omega^3_\eps$.

\item The two-dimensional linearity domain $R(v)$ is the induced 
(in $\Omega^3_\eps$) convex hull of the points $U(0)$ and $U(\pm v)$.

\item $\HL(u)=\big(u,u^2+\eps^2,\half\big((u+\eps)^p+(u-\eps)^p)\big)$,\quad for $u\ge\eps$ (see~\eqref{210901}).

\item $\HR(u)=\begin{cases}
\left(u,u^2+\eps^2,\half\big((u+\eps)^p+(u-\eps)^p)\big)\right),&\text{ for \ }u\ge\eps, 
\\
\left(u,u^2+\eps^2,(u+\eps)^{p-2}(u^2+\eps^2)\right),&\text{ for \ }0\le u\le\eps
\rule{0pt}{15pt}
\end{cases}$ (see~\eqref{210905}).

\item The two-dimensional domain $\FL(u)$ ($u\ge\eps$) is the triangle with the vertices $U(v)$, $W_+(v)$, 
and $\HL(u)$ foliated by a fan of extremals from the point $U(v)$, $v=u-\eps$.

\item The two-dimensional domain $\FR(u)$ ($u\ge0$) is the triangle with the vertices $U(v)$, $W_-(v)$,  
and $\HR(u)$ foliated by a fan of extremals from the point $U(v)$, $v=u+\eps$.

\item $\Delta_\pm=\eps\pm d$, $d=\sqrt{x_1^2+\eps^2-x_2}$\,.
\end{itemize}

\subsubsection{Definition of $B_1$}
\label{B1_def}

\paragraph{Domain $\Xi\cii0.$} The domain
\eqlb{061001}{
\Xi\cii0=\Set{x\in\Omega_\eps^3}{|x_1|\leq 2\eps,\ 
x_2\geq4\eps|x_1|-3\eps^2\!,\ (p-2)\big(x_3-\eps^p-\frac{x_2-\eps^2}{4\eps}m_p(\eps)\big)\geq 0},
}
is foliated by the two-dimensional linearity domains $\tilde T(v)$, $0\le v\le\eps$. The curvilinear trapezoid
$\tilde T(v)$ with three straight line sides $[U(-v),U(v)]$, 
$[U(v),W_+(v)]$, and $[U(-v),W_-(-v)]$ belongs to the plane
\eqlb{091001}{
x_3=v^p+\frac{m_p(v)}{2(v+\eps)}(x_2-v^2).
}
The function $B_1$ on $\tilde T(v)$ is defined by the formula 
\eqlb{071001}{
B_1(x)=v^r+\frac{m_r(v)}{2(v+\eps)}(x_2-v^2).
}

\paragraph{Domain $\Xi\cii+.$} The domain
\eqlb{071002}{
\Xi\cii+=\Set{x\in\Omega_\eps^3\setminus\Xi\cii0}{x_1\ge0}, 
}
is foliated by the two-dimensional extremals $T(v)$, $v\ge\eps$. The curvilinear triangle
$T(v)$ is the intersection of the domain $\Omega_\eps^3$ with the plane
\eqlb{091002}{
x_3=v^p+\frac{m_p(v)-k_p(v)}{4\eps}(x_2-2vx_1+v^2)+\frac{m_p(v)+k_p(v)}{2}(x_1-v).
}
The function $B_1$ on $T(v)$ is defined by the formula 
\eqlb{071003}{
B_1(x)=v^r+\frac{m_r(v)-k_r(v)}{4\eps}(x_2-2vx_1+v^2)+\frac{m_r(v)+k_r(v)}{2}(x_1-v).
}

\paragraph{Domain $\Xi\cii-.$} The domain
\eqlb{071004}{
\Xi\cii-=\Set{x\in\Omega_\eps^3\setminus\Xi\cii0}{x_1\le0}, 
}
is foliated by the two-dimensional extremals $T(v)$, $v\le-\eps$. The function $B_1$ on
$\Xi\cii-$ is defined by the formula 
\eqlb{071005}{
B_1(x_1,x_2,x_3)=B_1(-x_1,x_2,x_3).
}

It is possible to join all the formulas above in the following way. 
Let us introduce a function $\xi\mapsto w_1(\xi;s,\eps,x_1,x_2)$ by the formula
\eqlb{071006}{
w_1(\xi;s,\eps,x_1,x_2)=
\begin{cases}
\displaystyle\xi^s+\frac{m_s(\xi)}{2(\xi+\eps)}(x_2-\xi^2),\quad&\text{ if \ }0\le\xi\le\eps,
\\
\displaystyle\xi^s+\frac{m_s(\xi)-k_s(\xi)}{4\eps}(x_2-2\xi x_1+\xi^2)+\frac{m_s(\xi)+k_s(\xi)}2(x_1-\xi),
\quad&\text{ if \ }\xi\ge\eps. \rule{0pt}{20pt}
\end{cases}
}
The symbols $s$, $\eps$, $x_1$, and $x_2$ are considered here as fixed parameters. Then we can rewrite 
relations~\eqref{071001}, \eqref{071003}, and~\eqref{071005} in terms of $w_1$:
\eqlb{071007}{
B_1(x_1,x_2,w_1(v;p,\eps,|x_1|,x_2))=w_1(v;r,\eps,|x_1|,x_2),
}
in other words 
\eqlb{071008}{
B_1(x_1,x_2,x_3)=w_1(w_1^{-1}(x_3;p,\eps,|x_1|,x_2);r,\eps,|x_1|,x_2).
}

\subsubsection{Definition of $B_2$}
\label{B2_def}

\paragraph{Domain $\Xi\cii{\mathrm L+}.$} The domain
\eqlb{071009}{
\begin{aligned}
\Xi\cii{\mathrm L+} =\; &\Big\{x\in\Omega_\eps^3\colon x_1\geq\Delta_-,
\\
&(2-p)\Am{p}(x_1,x_2)\leq(2-p)x_3
\leq(2-p)\frac{\Delta_-(x_1+\Delta_+)^p+\Delta_+(x_1-\Delta_-)^p}{2\eps}\Big\}\,.
\end{aligned}
}
is foliated by the fans $\FL(u)$, $\eps\le u<\infty$ (see Lemma~\ref{080504} below). 

The function $B_2$ on $\FL(u)$ is defined by~\eqref{291003l}: 
\eqlb{eq080505}{
B_2(x)=v^r+\Big[m_r(v)+\frac1\eps e^{\frac u\eps}\wL(\xi;r,\eps)\Big]\Delta_-,
}
where, $v=x_1-\Delta_-$, $u=v+\eps$, the function $\xi\mapsto\wL(\xi;s,\eps)$ was defined by~\eqref{w_L_def}:
$$
\wL(\xi;s,\eps)=e^{-\frac\xi\eps}\big[\half\big((\xi+\eps)^s-(\xi-\eps)^s\big)-\eps m_s(\xi-\eps)\big],
\qquad \xi \in [\eps,+\infty].
$$
We would like to note that $\xi=+\infty$ is included in the domain of $\wL$ and 
$\wL(+\infty)=0$. The value of the variable $\xi=\xi(x;p,\eps)$ in~\eqref{eq080505} is obtained as the solution of the equation
\eqlb{141001}{
x_3=v^p+\Big[m_p(v)+\frac1\eps e^{\frac u\eps}\wL(\xi;p,\eps)\Big]\Delta_-
}
running from $u=x_1+d$ to $+\infty$.

We are ready to write down the function $B_2$ in the form~\eqref{071008}, as it was made for the function $B_1$.
From~\eqref{eq080505} and~\eqref{141001} we see that we should introduce the following function $w_2$:
\eqlb{141002}{
w_2(\xi;s,\eps,x_1,x_2)=v^s+\Big[m_s(v)+\frac1\eps e^{\frac u\eps}\wL(\xi;s,\eps)\Big]\Delta_-,
\quad u\le\xi\le\infty,\ u=v+\eps=x_1+d.
}
Then, for $x \in \Xi\cii{\mathrm L+}$ we have
\eqlb{051001}{
B_2(x_1,x_2,x_3)=w_2(w_2^{-1}(x_3;p,\eps,x_1,x_2);r,\eps,x_1,x_2).
}

\paragraph{Domain $\Xi\cii{\mathrm R+}.$} The domain
\eqlb{121001}{
\begin{aligned}
\Xi\cii{\mathrm R+} =\; &\Big\{x\in\Omega_\eps^3\colon x_1\geq0,\ x_2\geq\eps^2,\ 
(p-2)\Ak{p}(x_1,x_2)\leq(p-2)x_3
\\
&\quad\leq(p-2)
\begin{cases}
(x_1+\Delta_-)^{p-2}x_2,&\quad\text{if }\ d\le x_1\le\Delta_+,\rule[-10pt]{0pt}{10pt}
\\
\displaystyle
\frac{\Delta_-(x_1-\Delta_+)^p+\Delta_+(x_1+\Delta_-)^p}{2\eps},
&\quad\text{if }\ x_1\ge\Delta_+;
\end{cases}
\end{aligned}
}
is foliated by the fans $\FR(u)$, $0\le u<\infty$ (see Lemma~\ref{080506} below). 

The function $B_2$ on $\FR(u)$ is defined by the formula 
\eqlb{eq080507}{
B_2(x)=v^r-\Big[k_r(v)-\frac1\eps e^{-\frac u\eps}\wR(\xi;r,\eps)\Big]\Delta_-,
}
where $v=x_1+\Delta_-$, $u=v-\eps$, the function $\xi\mapsto\wR(\xi;s,\eps)$ was defined by~\eqref{220904}:
\eqlb{eq060901}{
\wR(\xi; s,\eps)=
\begin{cases}
\eps e^{\frac{\xi}{\eps}}\big[k_s(\xi+\eps)-2\xi(\xi+\eps)^{s-2}\big],\quad
&\text{ if \ }0\le\xi\le\eps,
\\
e^{\frac\xi\eps}\big[\half\big((\xi-\eps)^s-(\xi+\eps)^s\big)+\eps k_s(\xi+\eps)\big],\quad
&\text{ if \ }\xi\ge\eps.
\rule{0pt}{15pt}
\end{cases}
}
The value of the variable $\xi=\xi(x;p,\eps)$ in~\eqref{eq080507} is the solution of the equation
\eqlb{141003}{
x_3=v^p-\Big[k_p(v)-\frac1\eps e^{-\frac u\eps}\wR(\xi;p,\eps)\Big]\Delta_-
}
running from $0$ till $u = x_1-d$.

Comparing~\eqref{eq080507} and~\eqref{141003} we see that we should introduce the following function $w_2$:
\eqlb{141004}{
w_2(\xi;s,\eps,x_1,x_2)=v^s-\Big[k_s(v)-\frac1\eps e^{-\frac u\eps}\wR(\xi;s,\eps)\Big]\Delta_-,
\quad 0\le\xi\le u,\ u=v-\eps=x_1-d.
}
Recall that here $v=x_1+\Delta_-$, not as in~\eqref{141002}, where $v=x_1-\Delta_-$. Again, for $x \in \Xi\cii{\mathrm R+}$ we have~\eqref{051001}.

\paragraph{Domain $F(0).$} The domain
\eqlb{071010}{
F(0)=\Big\{x\in\Omega_\eps^3\colon|x_1|\leq\eps,\ \eps|x_1|\leq x_2\leq x_1^2+\eps^2,
\quad (2-p)x_3\leq(2-p)(2\eps)^{p-2}x_2\Big\}
}
is foliated by the two-dimensional domains of linearity. The function $B_2$ does not depend on $x_1$:
\eqlb{eq051002}{
B_2(x_1,x_2,x_3) = B_2\Big(\frac{x_2}{2\eps},x_2,x_3\Big),
}
and the former point lies in $\FL(\eps)$. We can use~\eqref{eq10}, \eqref{161001}, and~\eqref{eq3} for $v=0$ to rewrite the right-hand side of~\eqref{eq051002} to obtain
\eqlb{071011}{
B_2(x)=\Big[e\PsiL\Big(\frac{2\eps^2 x_3}{ex_2}-\frac{\eps^p}e\Gamma(p+1)\Big)+
\eps^{r-1}\Gamma(r+1)\Big]\frac{x_2}{2\eps},\qquad x \in F(0).
}

However, we would like to write down this formula in another form. We prefer to have a description on each leaf of the foliation of~$F(0)$ separately. 
Now we have 
appropriate tools to describe the foliation of~$F(0)$. The right boundary of $F(0)$ (that is the 
boundary between $F(0)$ and $\Xi\cii{\mathrm L+}$) consists of the fan~$\FL(\eps)$, and the symmetrical 
boundary (that is the boundary between $F(0)$ and $\Xi\cii{\mathrm R-}$) consists of the symmetrical fan 
$\FR(-\eps)$. The extremals of a fan $\FL(u)$ are segments $[U(u-\eps),H(u,h)]$ that are parametrized 
either by the parameter $h$, or by the corresponding parameter $\xi$ (see~\eqref{021104}). 
The parameter~$h$ is running from zero (when the endpoint $H(u,h)$ is on the boundary) till 
$\half\big((u+\eps)^p-(u-\eps)^p\big)-\eps m_p(u-\eps)$ (when $H(u,h)$ is at the middle of the chord 
$[U(u-\eps),U(u+\eps)]$). Therefore, the extremals of the fan $\FL(\eps)$ are 
parametrized by the parameter $h$ running from zero till $\eps^p\big(2^{p-1}-\Gamma(p+1)\big)$. 
Each such extremal is a boundary of the two-dimensional domain of linearity $F_0(h)$ being the induced convex hull of the points 
$(0,0,0)$ and $(\pm\eps,2\eps^2,\eps m_p(0)+h)$. The whole domain $F(0)$ is foliated by $F_0(h)$. 
The triangle $F_0(0)$ is on the boundary of $\Omega^3_\eps$ and the triangle 
$F_0\big((2^{p-1}-\Gamma(p+1))\eps^p\big)$ separates $F(0)$ from $R$.

Let us compare now formula~\eqref{161001} with~\eqref{141001} taking into account that 
$\Delta_-=x_1=\frac{x_2}{2\eps}$ for $\FL(\eps)$ (see~\eqref{190901}). We see that $h=e\cdot \wL(\xi;p,\eps)$,
and~\eqref{eq080505} supplies us with the formula for $B_2$ on $F_0(h)$:
$$
B_2(x)=\left[m_r(0)+\frac{e}{\eps}\wL\Big(\wL^{-1}\Big(\frac{h}{e};p,\eps\Big);r,\eps\Big)\right]\frac{x_2}{2\eps}.
$$
This formula coincides with~\eqref{071011}, since $m_s(0)=\eps^{s-1}\Gamma(s+1)$.

Instead of parameter $h$ we can parametrize the leaves of $F(0)$ by the parameter $\xi$ from~\eqref{021104}
running from $\eps$ to $+\infty$. Then
\eqlb{220601}{
x_3=\Big[m_p(0)+\frac e\eps\wL(\xi;p,\eps)\Big]\frac{x_2}{2\eps}
}
and
\eqlb{220602}{
B_2(x)=\Big[m_r(0)+\frac e\eps\wL(\xi;r,\eps)\Big]\frac{x_2}{2\eps}.
}
Therefore, for this domain it is natural to introduce the function $w_2$ by the formula
\eqlb{181001}{
w_2(\xi;s,\eps,x_1,x_2)=\Big[\eps^s\Gamma(s+1)+e \wL(\xi;s,\eps)\Big]\frac{x_2}{2\eps^2},
\quad \eps\le\xi\le+\infty,
}
and then, as before,
\eqlb{181002}{
B_2(x_1,x_2,x_3)=w_2(w_2^{-1}(x_3;p,\eps,|x_1|,x_2);r,\eps,|x_1|,x_2).
}

\paragraph{Domain $R$.} The function $B_2$ has a very simple description in the domain $R$ (see~\eqref{eq070502}): 
\eqlb{191001}{
B_2(x_1,x_2,x_3)=x_3^{\frac{r-2}{p-2}}x_2^{\frac{p-r}{p-2}}.
}
However, if we would like to present $B_2$ in  a form such as~\eqref{181002}, we should use
another representation from~\eqref{eq070502}:
\eqlb{191002}{
B_2(x_1,x_2,x_3)=v^{r-2}x_2,\qquad x\in R(v),\quad 0\le v\le2\eps,
}
together with formula for the third coordinate
\eqlb{191003}{
x_3=v^{p-2}x_2,\qquad x\in R(v),\quad 0\le v\le2\eps.
}
This gives us the same formula~\eqref{181002}, if we take
$$
w_2(\xi;s,\eps,x_1,x_2)=v(\xi)^{s-2}x_2,
$$
with some monotone parametrization $v(\xi)$. In what follows we will use not the definition of the 
form~\eqref{181002}, but the direct definition~\eqref{191001}. It is the reason why we will not specify the function
$v(\xi)$, this may be done in many different ways.

\paragraph{Domain $\Xi\cii{\mathrm{ch}+}$.}  Finally, on the domain $\Xi\cii{\mathrm{ch}+}$ we will use formula~\eqref{eq080509} for our Bellman 
candidate $B_2$ on a chord $[U(a),U(b)]$:
\eqlb{090601}{
B_2(x)=\frac{x_1-a}{b-a}b^r+\frac{b-x_1}{b-a}a^r,
}
where the parameters $a$ and $b$ are considered as functions of $x$ determined by the pair of equations
\eqlb{231001}{
x_2=(a+b)x_1-ab
}
and
\eqlb{231002}{
x_3=\frac{x_1-a}{b-a}b^p+\frac{b-x_1}{b-a}a^p.
}
We have used the fact that $x\in[U(a),U(b)]$.

Comparing~\eqref{090601} with~\eqref{231002}, it is easy to represent the function $B_2$ in 
the same form~\eqref{181002}. To do this it suffices to parametrize all chords passing through a 
point $x$ with fixed $x_1$ and $x_2$ by some parameter $\xi$ and consider some functions 
$\xi\mapsto a(\xi;x_1,x_2)$ and $\xi\mapsto b(\xi;x_1,x_2)$ rather than functions $a$ and $b$ dependent
on $x$. After such a choice we can put
$$
w_2(\xi;s,\eps,x_1,x_2)=\frac{x_1-a}{b-a}b^s+\frac{b-x_1}{b-a}a^s
$$
to obtain the representation~\eqref{181002}.
This also may be done in many different ways. We will not use such a representation,
since we have not found the choice of the parameter $\xi$ that essentially simplifies all calculations.

\section{Foliation. Proofs}
\label{FoliationProofs}

In this section we formally prove that the foliation corresponding to the candidate~$B_2$ is precisely the foliation announced in Section~\ref{SectFoliation}.

\begin{Le} 
\label{le300401}
Fix a point $(x_1,x_2)\in\omega_2 \cup \omega_3 \cup \omega_4$ \textup(see~\eqref{041101}\textup). Consider all the pairs $a,b >0$ such that the chord $[U(a),U(b)]$ contains the point of the form $(x_1,x_2, \;\cdot\;)$. Then\textup, the third coordinate of this point\textup, considered as a function of $a,$ 
is strictly increasing for $p>2$ and strictly decreasing for $p<2$.
\end{Le}

\begin{proof}
First we differentiate relation~\eqref{231001} with respect to $x_3$. As a result, we get
\eqlb{eq300403}{
b_{x_3}(x_1-a)=a_{x_3}(b-x_1)\,.
}
Therefore, both endpoints of the chord $[U(a),U(b)]$ move in one and the same direction. Without lost of generality, we may
assume $a<x_1<b$. 

Now, we calculate the derivative of the function
\eqlb{230601}{
w(a,b;s)=\frac{(b-x_1)a^s+(x_1-a)b^s}{b-a}.
}
We have
\eqlb{261211}{
\begin{aligned}
w_{x_3}\buildrel\phantom{\eqref{eq300403}}\over=
&\frac{\big[a^sb_{x_3}\!\!+\!(b\!-\!x_1)sa^{s-1}a_{x_3}\!\!-\!
b^sa_{x_3}\!\!+\!(x_1\!-\!a)sb^{s-1}b_{x_3}\big](b\!-\!a)\!-\!
(b_{x_3}\!\!-\!a_{x_3})\big[(b\!-\!x_1)a^s\!+\!(x_1\!-\!a)b^s\big]}{(b-a)^2}
\\
\buildrel\phantom{\eqref{eq300403}}\over=
&\frac{a_{x_3}(b-x_1)\big[sa^{s-1}(b-a)+a^s-b^s\big]+b_{x_3}(x_1-a)\big[sb^{s-1}(b-a)+a^s-b^s\big]}{(b-a)^2}
\\
\buildrel\eqref{eq300403}\over=
&\frac{a_{x_3}(b-x_1)}{(b-a)^2}\big[s(b-a)(a^{s-1}+b^{s-1})+2a^s-2b^s\big]=
\frac{a_{x_3}(b-x_1)}{(b-a)^2}\AAA\big(\half(b+a),\half(b-a),s\big)\,,
\end{aligned}
}
where the function $\AAA$ was defined in~\eqref{051003}. We see that
$\sign\AAA(\alpha,\beta,s)=\sign(s-2)$ from the integral representation~\eqref{240403}. Since $x_3=w(a,b;p)$ (see~\eqref{231002}) and $x_1<b$, we immediately conclude that 
$\sign a_{x_3}=\sign(p-2)$. This gives us the conclusion of the lemma.
\end{proof}

\begin{Le} 
\label{070500}
The domain $R$ is foliated by leaves $R(v)$\textup, $0\le v\le2\eps$. If we keep the
first two coordinates of a point $x\in R(v)$ fixed\textup, we get a function $v\mapsto x_3$. 
This function $x_3(v)$ is strictly increasing for $p>2$ and strictly decreasing for $p<2$.
\end{Le}

\begin{proof} Since by definition $R(v)$ is the induced (in $\Omega^3_\eps$) convex hull of the points 
$U(0)$ and $U(\pm v)$, it is a part of the two-dimensional plane $x_3=v^{p-2}x_2$. From this formula the second statement of the Lemma is clear: the function $x_3(v)$ is strictly increasing for 
$p>2$ and strictly decreasing for $p<2$. 

To prove the first statement we need to consider four subdomains $R_i$ separately (see~\eqref{R}--\eqref{111204}). The projection on the $x_1x_2$-plane of
the points from $R_0$ (see~\eqref{111201}) lie in $\omega_0$ (see~\eqref{041101}), see Fig.~\ref{omega_domains}. Now, we will look at the domains $R(v)$ when $v$ increases from zero
to $2\eps$. For a fixed point $(x_1,x_2)\in\omega_0$ the first moment when the projection of $R(v)$ to the 
first two coordinates contains the point $(x_1,x_2)$ is $v=\sqrt{x_2}$, and the point $(x_1,x_2,x_3(v))$ 
lies in $R(v)$ till $v=2\eps$. That means that $x_3(v)$ continuously varies 
from $x_2^{p/2}$ till $(2\eps)^{p-2}x_2$. This is just what is written in formula~\eqref{111201}, 
where the factor $(2-p)$ reflects the fact that the function $x_3(v)$ is increasing for $p>2$ 
and decreasing for $p<2$.

Now, we consider in a similar way the case when $(x_1,x_2) \in \omega_1$. The first moment when the projection of 
$R(v)$ contains this point is $v=x_1+\Delta_-$, i.\,e., when the point 
$(x_1,x_2)$ is on the tangent line passing through $(v,v^2)$. The point $(x_1,x_2,x_3(v))$ lies 
in $R(v)$ till $v=2\eps$.  That means that $x_3(v)$ continuously varies from 
$(x_1+\Delta_-)^{p-2}x_2$ till $(2\eps)^{p-2}x_2$, as it is written in formula~\eqref{111202}.

For the case when $(x_1,x_2) \in \omega_2$, the first moment when the projection of $R(v)$ contains this point
is the same as for $\omega_0$, i.\,e., $v=\sqrt{x_2}$. But now, the point $(x_1,x_2,x_3(v))$ lies 
in $R(v)$ only till the moment when $x\in[U(0),U(v)]$, i.\,e., till $v=x_2/x_1$.  That means that $x_3(v)$ 
continuously varies from $x_2^{p/2}$ till~$x_1^{2-p}x_2^{p-1}$, as it is written in formula~\eqref{111203}.

Finally, for the case when $(x_1,x_2) \in \omega_3$, the first moment when projection of $R(v)$ contains 
this point is the same as for $\omega_1$, i.\,e., $v=x_1+\Delta_-$. As in the preceding case, the point 
$(x_1,x_2,x_3(v))$ lies in $R(v)$ till the moment when $x\in[U(0),U(v)]$, i.\,e., till $v=x_2/x_1$. 
That means that $x_3(v)$ continuously varies from $(x_1+\Delta_-)^{p-2}x_2$ till $x_1^{2-p}x_2^{p-1}$, 
as it is written in formula~\eqref{111204}.
\end{proof}

\begin{Le} 
\label{080504}
The domain $\Xi\cii{\mathrm L+}$ is foliated by the fans $\FL(u)$\textup, $\eps\le u<\infty$.
\end{Le}

\begin{proof}
First, we note that the fans $\FL(u)$ do not intersect. Indeed, the projection of $\FL(u)$ onto the \hbox{$x_1x_2$-plane}
is the tangent $S_+(u-\eps)$ and these lines do not intersect (see the beginning of Subsection~\ref{2dimBell} and 
Fig.~\ref{fig161101} there). 

The fact that $\FL(u)$ foliate the whole domain $\Xi\cii{\mathrm L+}$ is almost evident.
Indeed, for every point $x\in\Xi\cii{\mathrm L+}$ (see~\eqref{eq080501}) we have $x_1\ge\Delta_-$.
This is equivalent to the assertion that a left tangent line $S_+(v)$ passes through $(x_1,x_2)$
with $v = u_+(x_1,x_2)\ge 0$, see Fig.~\ref{fig161101}. Thus, we have to consider the fan $\FL(u)$ with $u=v+\eps$. Recall that the extremal
lines in this fan are parametrized by $h$ (see~\eqref{210902}), which runs from zero till
$\half\big((u+\eps)^p-(u-\eps)^p\big)-\eps m_p(u-\eps)$. This means that $x_3$ (see~\eqref{291001}) 
runs from $\Am{p}(x_1,x_2)$ till 
$$
(u-\eps)^p+\frac{(u+\eps)^p-(u-\eps)^p}{2\eps}\Delta_-=\frac{(x_1+\Delta_+)^p\Delta_-+(x_1-\Delta_-)^p\Delta_+}{2\eps}\,,
$$
because $u+\eps=x_1+\Delta_+$ and $u-\eps=x_1-\Delta_-$. These are the boundary values for $x_3$ described 
in~\eqref{eq080501}.
\end{proof}

\begin{Le} 
\label{080506}
The domain $\Xi\cii{\mathrm R+}$ is foliated by the fans $\FR(u)$\textup, $0\le u<\infty$.
\end{Le}

\begin{proof}
The reasoning here is the same as in the preceding lemma. The fans $\FR(u)$ do not intersect because 
their projections onto the $x_1x_2$-plane are disjoint tangent lines $S_-(u+\eps)$. Thus, we have to check
that for any $x$, $x\in\Xi\cii{\mathrm R+}$ (see~\eqref{eq080502}), there exists $u\ge0$ such that $x\in\FR(u)$.

First, we note that $x_2\ge\eps^2$ for $x\in\Xi\cii{\mathrm R+}$. This is equivalent to the assertion that some right 
tangent line $S_-(v)$ passes through $(x_1,x_2)$ with $v\ge\eps$. Thus, we have to consider the fan $\FR(u)$ 
with $u=v-\eps$. Recall that the extremal lines in this fan are parametrized by $h$ (see~\eqref{210903}) 
running from zero till $\half\big((u-\eps)^p-(u+\eps)^p\big)+\eps k_p(u+\eps)$ if $u\ge\eps$, and till 
$\eps k_p(u+\eps)-2\eps u(u+\eps)^{p-2}$ if $0\le u\le\eps$. This means that in the case $u\ge\eps$
(i.\,e., $x_1\ge\Delta_+$) $x_3$ runs (see~\eqref{011102}) from $\Ak{p}(x_1,x_2)$ till 
\eqlb{181201}{
(u+\eps)^p+\frac{(u-\eps)^p-(u+\eps)^p}{2\eps}\Delta_-=\frac{(x_1-\Delta_+)^p\Delta_-+(x_1+\Delta_-)^p\Delta_+}{2\eps}\,,
}
because $u+\eps=x_1+\Delta_-$ and $u-\eps=x_1-\Delta_+$. In the case $0\le u\le\eps$ the coordinate
$x_3$ runs from $\Ak{p}(x_1,x_2)$ till 
\eqlb{181202}{
(x_1+\Delta_-)^{p-2}x_2\,,
}
because $v=x_1+\Delta_-=u+\eps$ and $v^2-2u\Delta_-=x_2$.
The boundary values for $x_3$ described in~\eqref{181201} and~\eqref{181202} are given in~\eqref{eq080502}.
\end{proof}

\begin{Le} 
\label{080508}
The domain $\Xi\cii{\mathrm{ch}+}$ is foliated by the chords $[U(a),U(b)]$\textup, $0\le a<\infty$\textup, 
$a\le b\leq a+2\eps$.
\end{Le}

\begin{proof}
The fact that the chords $[U(a),U(b)]$ are disjoint was proved in Lemma~\ref{le300401}. Thus, we only have to verify that for any $x$, $x\in\Xi\cii{\mathrm{ch}+}$ (see~\eqref{eq080503}), there exist $a\ge0$ and 
$b\ge a$ such that $x\in[U(a),U(b)]$.

If $(x_1,x_2)\in\omega_4$ (i.\,e., if $x_1\ge\Delta_+$), then $x\in[U(a),U(b)]$ when $a$ runs from $x_1-\Delta_+$
till $x_1-\Delta_-$, and therefore, $b$ runs from $x_1+\Delta_-$ till $x_1+\Delta_+$. Due to formula~\eqref{231002},
we see that $x_3$ runs from $\big[\Delta_-(x_1-\Delta_+)^p+\Delta_+(x_1+\Delta_-)^p\big]/(2\eps)$ till
$\big[\Delta_-(x_1+\Delta_+)^p+\Delta_+(x_1-\Delta_-)^p\big]/(2\eps)$. These are exactly the bounds for $x_3$ as
stated in~\eqref{eq080503}.

If $(x_1,x_2)\in\omega_2\cup\omega_3$ (i.\,e., if $\Delta_-\le x_1\le\Delta_+$), then $x\in[U(a),U(b)]$ when 
$a$ runs from zero till $x_1-\Delta_-$, and therefore, $b$ runs from $x_2/x_1$ till $x_1+\Delta_+$. 
Again formula~\eqref{231002} yields that $x_3$ runs from $x_2^{p-1}x_1^{2-p}$ till
$\big[\Delta_-(x_1+\Delta_+)^p+\Delta_+(x_1-\Delta_-)^p\big]/(2\eps)$. These are the bounds for $x_3$
given by another part of~\eqref{eq080503}.
\end{proof}

\medskip

Completing this section, we collect some information concerning the order of subdomains with different types 
of foliation over an arbitrary point $(x_1,x_2)\in\Omega^2_\eps$ and the formulas for the boundaries between
different layers. This will be especially important when we glue the solutions in different subdomains and
verify that the obtained candidate is $C^1$-smooth.

For a fixed point $(x_1,x_2) \in \omega_j$ we list the domains and their boundaries over this point separated by~$|$ symbol. These objects are listed from up to down in the case~$p<2$ and in the reverse order in the case~$p>2$.

\eqlb{201202}{
(x_1,x_2)\in\omega_0\colon\quad x_3=\Ak{p}(x_1,x_2)\ \Big|\ R\ \Big|\ x_3=(2\eps)^{p-2}x_2\ \Big|\ F(0)\ 
\Big|\ x_3=\Am{p}(x_1,x_2)\,;
}

\eqlb{201203}{
\begin{aligned}
(x_1,x_2)\in\omega_1\colon\quad x_3=\Ak{p}(x_1,x_2)\ &\Big|\ \Xi\cii{\mathrm R+}
\Big|\ x_3=(x_1+\Delta_-)^{p-2}x_2\ \Big|\ R\ \Big|\ x_3=(2\eps)^{p-2}x_2\ \Big|
\\
&\Big|\ F(0)\ \Big|\ x_3=\Am{p}(x_1,x_2)\,;
\end{aligned}
}

\eqlb{201204}{
\begin{aligned}
(x_1,x_2)\in\omega_2\colon\quad x_3=&\Ak{p}(x_1,x_2)\ \Big|\ R\ \Big|\ x_3=x_2^{p-1}x_1^{2-p}\ \Big|
\ \Xi\cii{\mathrm {ch}+}\ \Big|
\\
&\Big|\ x_3=\frac{\Delta_-(x_1+\Delta_+)^p+\Delta_+(x_1-\Delta_-)^p}{2\eps}\ 
\Big|\ \Xi\cii{\mathrm L+} \Big|\ x_3=\Am{p}(x_1,x_2)\,;
\end{aligned}
}

\eqlb{201205}{
\begin{aligned}
(x_1,x_2)\in\omega_3\colon\quad x_3=&\Ak{p}(x_1,x_2)\ \Big|\ \Xi\cii{\mathrm R+}\Big|\ x_3=
(x_1+\Delta_-)^{p-2}x_2\ \Big|\ R\ \Big|\ x_3=x_2^{p-1}x_1^{2-p}\ \Big|
\ \Xi\cii{\mathrm {ch}+}\ \Big|
\\
&\Big|\ x_3=\frac{\Delta_-(x_1+\Delta_+)^p+\Delta_+(x_1-\Delta_-)^p}{2\eps}\ 
\Big|\ \Xi\cii{\mathrm L+}\Big|\ x_3=\Am{p}(x_1,x_2)\,;
\end{aligned}
}

\eqlb{201206}{
\begin{aligned}
(x_1,x_2)\in\omega_4\colon\quad x_3=&\Ak{p}(x_1,x_2)\ \Big|\ \Xi\cii{\mathrm R+}\Big|
\ x_3=\frac{\Delta_-(x_1-\Delta_+)^p+\Delta_+(x_1+\Delta_-)^p}{2\eps}\ \Big|
\ \Xi\cii{\mathrm {ch}+}\Big|
\\
&\Big|\ x_3=\frac{\Delta_-(x_1+\Delta_+)^p+\Delta_+(x_1-\Delta_-)^p}{2\eps}\ 
\Big|\ \Xi\cii{\mathrm L+}\Big|\ x_3=\Am{p}(x_1,x_2)\,.
\end{aligned}
}

\section{Properties of the Bellman candidate $B_2$}
\label{Bellman_candidate_proofs}

In this section we formulate local concavity/convexity property of our candidate $B_2$ defined 
in Subsection~\ref{B2_def}.
% Since we will deal only with this candidate we omit index $2$. 
% Then it will be convenient to denote the partial derivatives of $B$ placing the corresponding 
% variable in index: $B_{x_i}=\diff{B}{x_i}$.
Our aim is to prove the following result.
\begin{Th} 
\label{080510}
Let $B_2$ be the function defined in Subsection~\textup{\ref{B2_def}}. It is $C^1$-smooth and locally concave 
if \hbox{$(r-2)(r-p)<0$} and locally convex if $(r-2)(r-p)>0$.
\end{Th}

The required calculations are rather long. By this reason we decided to move the proof of the theorem to Appendix. The $C^1$-smoothness of $B_2$ is proved in Appendix~\ref{Ap1}, and the concavity/convexity of $B_2$ is proved in Appendix~\ref{Ap2}. 
% , we and we begin with some preparatory lemmas. We start
% with $C^1$-smoothness of $B$.

\section{Optimizers}\label{SecOptim}
In this section we provide optimizers for $B_2$. Namely, for any point $x \in \Omega_\eps^3$ we find a function $\vf_x$ (called optimizer for $B_2$ at $x$) that satisfies~\eqref{eqphi_x}. The construction of optimizers depends on the foliation, therefore, we consequently investigate different subdomains of $\Omega_\eps^3$.

\subsection{Optimizers for  points in $\Xi\cii{\mathrm L+}$}
We start with integration by parts and use the change of variable to prove the following identity:
$$
\eps m_s(v) \buildrel\eqref{mp+}\over =
s\! \int\limits_{v}^{+\infty}\!\! e^{\frac{v-t}\eps} t^{s-1} dt=
- v^s + e^{\frac{v}{\eps}}\!\!\!\int\limits_0^{\ e^{-\frac{v}{\eps}}}\! (-\eps\ln\tau)^s d \tau\,,
$$
and use this identity to rewrite the function~$\wL$ from~\eqref{w_L_def}:
$$
\wL(\xi;s,\eps) = e^{-\frac\xi\eps}\cdot\frac{(\xi+\eps)^s+(\xi-\eps)^s}{2}- e^{-1}\!\!\!\!\!\int\limits_0^{e^{-\frac{\xi-\eps}{\eps}}}\!\!\! (-\eps\ln\tau)^s d \tau\,,
\qquad \xi\ge\eps\,.
$$

Now, using $u=v+\eps$, the function $w_2$ defined in~\eqref{141002} can be rewritten in the following form: 
\begin{align*}
w_2(\xi;s,\eps,x_1,x_2) 
&= v^s + \frac{\Delta_-}{\eps}\Big[\eps m_s(v) + e^{\frac{u}{\eps}} \wL(\xi;s,\eps)\Big]
\\
&=\frac{\eps - \Delta_-}{\eps} v^s + \frac{\Delta_-}{\eps}\bigg[e^{\frac{u-\xi}{\eps}}\cdot\frac{(\xi+\eps)^s+(\xi-\eps)^s}{2} 
+e^{\frac{v}{\eps}}\!\!\!\!\! \int\limits_{e^{-\frac{\xi-\eps}{\eps}}}^{\  e^{-\frac{v}{\eps}}}\!\!\! (-\eps \ln \tau)^s d\tau\bigg]
\\
&=\frac{1}{l} \int\limits_{0\;}^{\;l}\vf_x(\tau)^s d\tau\,,
\end{align*}
where $l = \frac{\eps}{\Delta_-}e^{\frac{v}{\eps}}$ and 
$$
\vf_x(\tau) = 
\begin{cases}
\xi+\eps,  &0\leq \tau < \tau_1 \df \frac{1}{2}e^{-\frac{\xi-\eps}{\eps}},
\\
\xi-\eps, &\tau_1\leq \tau<  \tau_2\df e^{-\frac{\xi-\eps}{\eps}},
\\
-\eps \ln \tau, & \tau_2\leq \tau < \tau_3\df e^{-\frac{v}{\eps}},
\\
v, & \tau_3\leq \tau\leq \frac{\eps}{\Delta_-}e^{-\frac{v}{\eps}} = l\,.
\end{cases}
$$

Take $s = 1$. Then, we trivially have $m_1(v) =1$, $\wL(\xi;1,\eps) =0$, and 
$$
\av{\vf_x}{[0,l]} = w_2(\xi;1,\eps,x_1,x_2) = v+\Delta_- = x_1\,.
$$

Take $s=2$. We have $m_2(v) =2(v+\eps)$, $\wL(\xi;2,\eps) =0$, and 
$$
\av{\vf_x^2}{[0,l]} = w_2(\xi;2,\eps,x_1,x_2) = v^2+2(v+\eps)\Delta_- = x_2\,.
$$

According to~\eqref{141001} and~\eqref{eq080505} we have 
$$
\av{\vf_x^p}{[0,l]}  = x_3, \qquad \av{\vf_x^r}{[0,l]}  = B_2(x_1,x_2,x_3)\,. 
$$

Let us verify that the $\BMO$-norm of $\vf_x$ does not exceed $\eps$. To do that we use the techniques of the so-called delivery curves, see Chapter~5 in~\cite{ISVZ}. Consider the two-dimensional curve $\gamma$:
\eqlb{eqdelcurve}{
\gamma(\tau) = \Big(\av{\vf_x}{[0,\tau]},\av{\vf_x^2}{[0,\tau]}\Big), \qquad 0 <\tau<l\,. 
}
This curve is called the delivery curve generated by $\vf_x$. We see that this curve starts at the point $\big(\xi+\eps, (\xi+\eps)^2\big)$ 
on the lower boundary of $\Omega_\eps^2$ (when $\tau \in (0,\tau_1)$), 
then goes along the tangent line~$S_-(\xi+\eps)$ and arrives at 
$\big(\xi, \xi^2+\eps^2\big)$ when $\tau=\tau_2$ (see Fig.~\ref{fig161101} for the notation $S_\pm$). The point $\gamma(\tau)$
goes along the upper boundary of $\Omega_\eps^2$ when $\tau \in (\tau_2,\tau_3)$ 
and arrives at $\big(v+\eps, (v+\eps)^2+\eps^2\big)$ when $\tau = \tau_3$.
Then it goes along the tangent line $S_+(v)$ till the point $(x_1,x_2)$. 
We see that this curve is a graph of a convex function and that any 
tangent line to this curve does not cross the upper boundary of $\Omega_\eps^2$, only touches it. Therefore, by Corollary~5.1.6
from~\cite{ISVZ} we have  
$$
\|\vf_x\|_{\BMO} \leq \eps\,.
$$

We summarize: for any $x = (x_1,x_2,x_3) \in \Xi\cii{\mathrm L+}$ we have constructed the desired optimizer $\vf_x$ for $B_2$ at $x$ (see~\eqref{eqphi_x}). 

\subsection{Optimizers for points in $\Xi\cii{\mathrm R+}$}
\label{OptXiR}

We again integrate by parts and use the change of variable to prove the following identity:
\eqlb{eq270801}{
\eps k_s(v) \buildrel{\eqref{mp-}}\over= 
s\int\limits_\eps^u e^{(t-u)/\eps} t^{s-1} dt =
v^s -\eps^s e^{\frac{\eps-v}{\eps}} - 
e^{-\frac{v}{\eps}}\!\!\!\int\limits_e^{\ e^{\frac{v}{\eps}}}\!(\eps\ln\tau)^s d\tau,
}
and use this relation to rewrite the function~$\wR$ from~\eqref{220904}. Consider first the case $\xi \ge \eps$. We obtain
\begin{align*}
\wR(\xi;s,\eps)&= e^{\frac\xi\eps}\big[\half\big((\xi-\eps)^s-(\xi+\eps)^s\big)+\eps k_s(\xi+\eps)\big]
\\
&=e^{\frac\xi\eps}\cdot \frac{(\xi-\eps)^s+(\xi+\eps)^s}{2} - \eps^s - e^{-1}\!\!\!\!\int\limits_e^{\ e^{\frac{\xi+\eps}{\eps}}}\!\! (\eps\ln\tau)^s d \tau.
\end{align*}

By the two preceding formulas, we may represent the function~$w_2$ defined in~\eqref{141004} in the following form:
\begin{align*}
w_2(\xi;s,\eps,x_1,x_2) 
&=  v^s - \frac{\Delta_-}{\eps}\Big[\eps k_s(v) - e^{-\frac{u}{\eps}} \wR(\xi;s,\eps)\Big]
\\
&=\frac{\eps - \Delta_-}{\eps} v^s + \frac{\Delta_-}{\eps}\bigg[e^{\frac{\xi-u}{\eps}}\cdot\frac{(\xi+\eps)^s+(\xi-\eps)^s}{2}
+e^{-\frac{v}{\eps}}\!\!\!\! \int\limits_{\ e^{\frac{\xi+\eps}{\eps}}}^{e^{\frac{v}{\eps}}}\!\! (\eps \ln \tau)^s d\tau\bigg]
\\
&=\frac{1}{l} \int\limits_{0\;}^{\;l}\vf_x(\tau)^s d\tau\,,
\end{align*}
where $l = \frac{\eps}{\Delta_-}e^{\frac{v}{\eps}}$ and 
$$
\vf_x(\tau) = 
\begin{cases}
\xi-\eps,  &0\leq \tau < \tau_1 = \frac{1}{2}e^{\frac{\xi+\eps}{\eps}},
\\
\xi+\eps, &\tau_1\leq \tau< \tau_2 =  e^{\frac{\xi+\eps}{\eps}},
\\
\eps \ln \tau, & \tau_2\leq \tau < \tau_3 =  e^{\frac{v}{\eps}},
\\
v, & \tau_3\leq \tau\leq \frac{\eps}{\Delta_-}e^{\frac{v}{\eps}} = l\,.
\end{cases}
$$
Let us verify that the function $\vf_x$ is an optimizer for $B_2$ at $x$.
For $s=1$ we have 
$$
k_1(v) = 1 - e^{1-\frac{v}{\eps}}; \qquad \wR(\xi;1,\eps) =e^{\frac\xi\eps} \bigl(-\eps + \eps (1 - e^{1-\frac{\xi+\eps}{\eps}})\bigr) = -\eps\,.
$$
Therefore,
$$
\av{\vf_x}{[0,l]} = w_2(\xi;1,\eps,x_1,x_2) = v-\Delta_- (1 -  e^{1-\frac{v}{\eps}} +  e^{-\frac{u}{\eps}}) = v-\Delta_- = x_1\,.
$$
For $s=2$ we have 
$$
k_2(v) = 2(v-\eps); \qquad \wR(\xi;2,\eps)  =e^{\frac\xi\eps} \big(-2\xi\eps + \eps k_2(\xi+\eps)\big) = 0\,, 
$$
therefore
$$
\av{\vf_x^2}{[0,l]} = w_2(\xi;2,\eps,x_1,x_2) = v^2-2(v-\eps)\Delta_-  = x_2\,.
$$
Applying~\eqref{141003} and~\eqref{eq080507}, we obtain
\eqlb{eq270802}{
\av{\vf_x^p}{[0,l]}  = w_2(\xi;p,\eps,x_1,x_2)  = x_3, \qquad \av{\vf_x^r}{[0,l]}  = w_2(\xi;r,\eps,x_1,x_2) = B_2(x_1,x_2,x_3)\,.
}
Thus, to prove that $\vf_x$ is an optimizer for $B_2$ at $x$ it suffices to verify   
$
\|\vf_x\|_{\BMO} \leq \eps\,.
$
We will again argue using delivery curves. Consider the two-dimensional curve $\gamma$ defined by~\eqref{eqdelcurve}. In this case it starts at the point $\big(\xi-\eps, (\xi-\eps)^2\big)$ on the lower boundary of $\Omega_\eps^2$ (when $\tau \in (0,\tau_1)$), then goes along the tangent line~$S_+(\xi-\eps)$ and arrives at $\big(\xi, \xi^2+\eps^2\big)$ when $\tau=\tau_2$. The point $\gamma(\tau)$ goes along the upper boundary of~$\Omega_\eps^2$ when $\tau \in (\tau_2,\tau_3)$ and arrives at $\big(v-\eps, (v-\eps)^2+\eps^2\big)$ when $\tau = \tau_3$. Then it follows the tangent line~$S_-(v)$ till the point $(x_1,x_2)$. We see that this curve is a graph of a convex function and that any tangent line to this curve does not cross the upper boundary of $\Omega_\eps^2$ transversally. Therefore, by Corollary~5.1.6 from~\cite{ISVZ} we have  
$$
\|\vf_x\|_{\BMO} \leq \eps\,.
$$

We now turn to the case $0 \le \xi \le \eps$. By~\eqref{eq270801} we have
\begin{align*}
\wR(\xi;s,\eps)& \buildrel{\eqref{220904}}\over{=} e^{\frac\xi\eps}\big(\eps k_s(\xi+\eps) -2\xi \eps (\xi+\eps)^{s-2}\big)
\\
&=e^{\frac\xi\eps} \big((\xi+\eps)^s -2\xi \eps (\xi+\eps)^{s-2}\big) - \eps^s - e^{-1}\!\!\!\!\int\limits_{e\;}^{\ e^{\frac{\xi+\eps}{\eps}}}\!\! (\eps\ln\tau)^s d \tau\,.
\end{align*}

Thus, the function~$w_2$ defined in~\eqref{141004} has the following representation (recall that here $u = v-\eps$):
\begin{align*}
w_2(\xi;s,\eps,x_1,x_2) 
&= v^s - \frac{\Delta_-}{\eps}\Big[\eps k_s(v) - e^{-\frac{u}{\eps}} \wR(\xi;s,\eps)\Big]
\\
&\buildrel{\eqref{eq270801}}\over=\frac{\eps - \Delta_-}{\eps} v^s + \frac{\Delta_-}{\eps}\bigg[e^{\frac{\xi-u}{\eps}}\big((\xi+\eps)^s - 2 \xi \eps (\xi+\eps)^{s-2}\big)
+e^{-\frac{v}{\eps}}\!\!\!\! \int\limits_{e^{\frac{\xi+\eps}{\eps}}}^{\  e^{\frac{v}{\eps}}}\!\! (\eps \ln \tau)^s d\tau\bigg]
\\
&=\frac{1}{l} \int\limits_{0\;}^{\;l} |\vf_x(\tau)|^s d\tau\,,
\end{align*}
where $l = \frac{\eps}{\Delta_-}e^{\frac{v}{\eps}}$ and 
\eqlb{eq161001}{
\vf_x(\tau) = 
\begin{cases}
 - (\xi+\eps),  &0\leq \tau <\tau_1\df \alpha_- e^{\frac{\xi+\eps}{\eps}},
 \\
0, & \tau_1\leq \tau <\tau_2\df (1-\alpha_+)e^{\frac{\xi+\eps}{\eps}},
\\
\xi+\eps,  &\tau_2\leq \tau<\tau_3\df e^{\frac{\xi+\eps}{\eps}},
\\
\eps \ln \tau, & \tau_3\leq \tau <\tau_4\df e^{\frac{v}{\eps}},
\\
v, & \tau_4\leq \tau\leq \frac{\eps}{\Delta_-}e^{\frac{v}{\eps}} = l\,,
\end{cases}
}
here $\alpha_-$ and $\alpha_+$ are any non-negative numbers that satisfy $\alpha_-+\alpha_+ = \frac{\xi^2+\eps^2}{(\xi+\eps)^2}$. As before, one can easily verify 
that for $s=2$ we have $\wR(\xi;2,\eps) = 0$ and
$$
\av{\vf_x^2}{[0,l]} = w_2(\xi;2,\eps,x_1,x_2) = v^2-2(v-\eps)\Delta_-  = x_2\,.
$$
We also have~\eqref{eq270802}. The function $\vf_x$ is no longer non-negative, therefore, $\av{\vf_x}{[0,l]} \ne \av{|\vf_x|}{[0,l]}$, thus we cannot argue as before to calculate the average $\av{\vf_x}{[0,l]}$. On the other hand, we can choose $\alpha_-$ and $\alpha_+$ in such a way that $\av{\vf_x}{[0,l]} = x_1$. Let us verify that we may take
$$
\alpha_- = \frac{\eps^2-\eps\xi}{2(\xi+\eps)^2},\qquad \alpha_+ = \frac{\eps^2+\eps\xi+2\xi^2}{2(\xi+\eps)^2}\,.
$$
Indeed,
\begin{align*}
\int\limits_{0\;}^{\;l} \vf(\tau)\,d\tau &= (\alpha_+ - \alpha_-)(\xi+\eps) e^{\frac{\xi+\eps}{\eps}} + 
\eps\!\!\! \int\limits_{e^{\frac{\xi+\eps}{\eps}}}^{\ e^{\frac{v}{\eps}}}\!\! \ln \tau\,d \tau +
ve^{\frac{v}{\eps}} \big(\frac{\eps}{\Delta_-} -1\big)
\\
&=\xi e^{\frac{\xi+\eps}{\eps}} +\eps \Big(\frac{v}{\eps} e^{\frac{v}{\eps}} - e^{\frac{v}{\eps}} -  \frac{\xi+\eps}{\eps} e^{\frac{\xi+\eps}{\eps}} +  e^{\frac{\xi+\eps}{\eps}}\Big) +ve^{\frac{v}{\eps}} \big(\frac{\eps}{\Delta_-} -1\big)
\\
&=e^{\frac{v}{\eps}} \big(v-\eps +\frac{v\eps}{\Delta_-}-v \big) = e^{\frac{v}{\eps}} \frac{\eps}{\Delta_-} (v-\Delta_-)= e^{\frac{v}{\eps}} \frac{\eps}{\Delta_-} x_1 = l x_1\,.
\end{align*}

It remains to prove that 
$
\|\vf_x\|_{\BMO} \leq \eps.
$
Consider the two-dimensional curve $\gamma$ defined by~\eqref{eqdelcurve}. In this case it starts at the point $\big(-(\xi+\eps), (\xi+\eps)^2\big)$ on the lower boundary of $\Omega_\eps^2$ (when $\tau \in (0,\tau_1)$), then goes in the direction of origin and arrives at 
$$
\gamma(\tau_2)  = \frac{\alpha_-}{1-\alpha_+}\big(-(\xi+\eps), (\xi+\eps)^2\big) = \Big(-\frac{\eps^2-\xi^2}{\eps+3\xi}, \frac{(\eps-\xi)(\eps+\xi)^2}{\eps+3\xi} \Big).
$$
After that the point $\gamma(\tau)$ goes along the line segment in the direction of the point $(\xi+\eps,(\xi+\eps)^2)$ and arrives at $\gamma(\tau_3)=(\xi,\xi^2+\eps^2)$. Then it follows the upper boundary of $\Omega_\eps^2$ when $\tau \in (\tau_3,\tau_4)$ and arrives at $\big(v-\eps, (v-\eps)^2+\eps^2\big)$ when $\tau = \tau_4$. Finally, for $\tau \in [\tau_4,l]$ the point $\gamma(\tau)$ goes along the tangent line~$S_-(v)$ till the point $(x_1,x_2)$. Figure~\ref{figDelCurve} illustrates the curve $\gamma$. 
\begin{figure}[h]
    \centering
    \includegraphics[scale = 0.3]{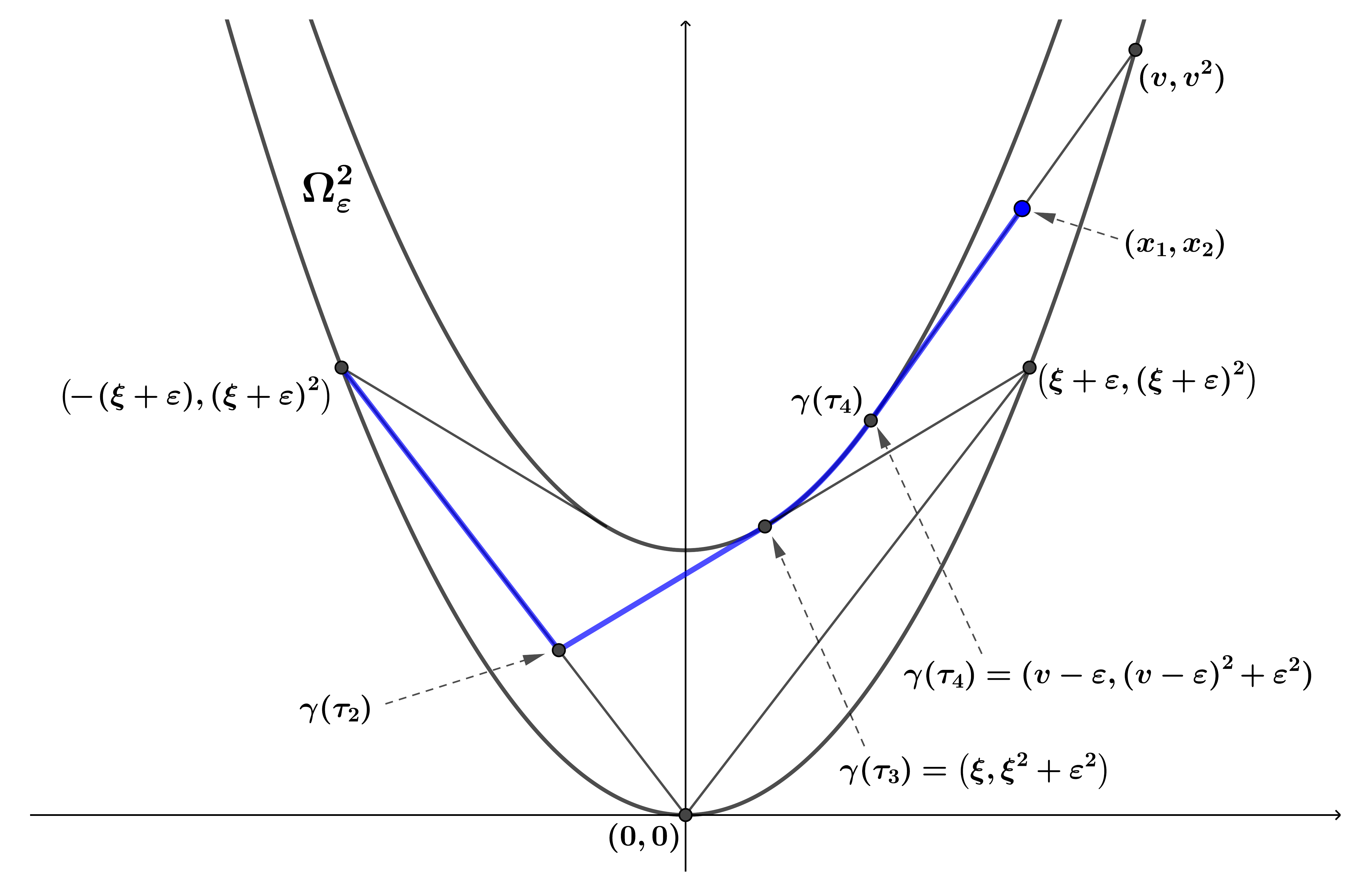}
    \caption{Graph of a delivery curve}
    \label{figDelCurve}
\end{figure}
This curve is a graph of a convex function and any tangent line to this curve does not cross the upper boundary of $\Omega_\eps^2$ transversally. Therefore, by Corollary~5.1.6 from~\cite{ISVZ} we obtain  
$$
\|\vf_x\|_{\BMO} \leq \eps.
$$

\subsection{Description of optimizers for points in other domains}
In this subsection we propose the optimizers at the points of the domains $R$, $F(0)$, and $\Xi\cii{\mathrm{ch}+}$. We omit some details of the proof because  we do not need these optimizers to prove that $B_2$ coincides with either~$\Bell_{p,r;\eps}^+$ or $\Bell_{p,r;\eps}^-$ (depending on $p$ and $r$).

%\subsection{Description of optimizers at the points in $R$ and $\Xi\cii{\mathrm{ch}+}$}
For any $x \in \Xi\cii{\mathrm{ch}+}$ there is a  chord of the form $[U(a), U(b)]$, $0 \leq a \leq b \leq a +2 \eps$, that passes through $x$ such that the function $B_2$ is linear on this chord. The function
$$
\vf_x(\tau) =
\begin{cases}
b, & a \leq \tau < x_1,\\
a, & x_1 \leq \tau \leq b,
\end{cases}
$$
defined on an interval $[a, b]$ is an optimizer for $B_2$ at $x$. Moreover, any function with the same distribution function will also be an optimizer at $x$.

For any $x \in R$ there is $v \in [0,2\eps]$ such that $x \in R(v)$, which is the induced convex hull of three points on the skeleton: $U(v), U(-v),$ and $U(0)=0$, see Subsection~\ref{FolNearZero}. The function $B_2$ is linear on $R(v)$, therefore, it is natural to find an optimizer that attains only three values: $-v, 0,$ and $v$. 
As $x \in R(v)$ lies inside the triangle with the vertices $U(v), U(-v),$ and $U(0)=0$, there are coefficients $\alpha_-, \alpha_0, \alpha_+ \in [0,1]$ such that
$$
x = \alpha_- U(-v)  + \alpha_0 U(0)+ \alpha_+ U(v), \qquad \alpha_- + \alpha_0 + \alpha_+ = 1,
$$ namely, they are
$$
\alpha_- = \frac{x_2 - vx_1}{2v^2}, \qquad \alpha_+ =  \frac{x_2 + vx_1}{2v^2}, \qquad \alpha_0 = \frac{v^2-x_2}{v^2}.
$$
Define the function
$$
\vf_x(\tau) =
\begin{cases}
-v, & 0 \leq \tau < \alpha_-,\\
0, &  \alpha_- \leq \tau < \alpha_-+\alpha_0 = 1 - \alpha_+,\\
v, & 1 - \alpha_+ \leq \tau \leq 1.
\end{cases}
$$
It is obvious that $\av{\vf_x}{[0,1]} = x_1$,  $\av{\vf_x^2}{[0,1]} = x_2$, and $\av{|\vf_x|^p}{[0,1]}=x_3$. The identity $\av{|\vf_x|^r}{[0,1]}=B_2(x)$ follows from the linearity of $B_2$ on  $R(v)$. It is easy  to verify that $\vf_x$ has $\BMO$-norm not greater than $\eps$, therefore, the function  $\vf_x$ constructed above is an optimizer for $B_2$ at~$x$.

%\subsection{Description of optimizers at the points in $F(0)$}

The domain $F(0)$ is foliated by the two-dimensional domains of linearity $F_0(h)$, each such domain is the curvilinear triangle with the vertices 
$U(0) = (0,0,0)$ and $F_\pm(h) = (\pm\eps,2\eps^2,\eps m_p(0)+h)$. The last two vertices, $F_\pm(h)$, lie on the boundary of the domains $\Xi\cii{\mathrm L+}$ and $\Xi\cii{\mathrm R-}$, 
we already know the optimizers for $B_2$ at these points. The optimizer at $x$ can be constructed by gluing those two optimizers together with the zero function (optimizer at the origin) in appropriate proportions. Namely, let  $\alpha_0, \alpha_\pm \in [0,1]$ be coefficients such that
$$
x = \alpha_- F_-(h) + \alpha_0 U(0) + \alpha_+ F_+(h), \qquad \alpha_- + \alpha_0 + \alpha_+ = 1.
$$
Let $\vf_\pm$ be optimizers for $B_2$ at the points $F_\pm(h)$ such that $\vf_+$ is non-decreasing and 
non-negative, $\vf_-$ is non-decreasing and non-positive, both defined on $[0,1]$. Then the function
\eqlb{eq280801}{
\vf_x(\tau) =
\begin{cases}
\vf_-(\frac{\tau}{\alpha_-}), & 0 \leq \tau < \alpha_-,\\
0, & \alpha_- \leq \tau < \alpha_- + \alpha_0,\\
\vf_{+}(\frac{\tau - \alpha_- - \alpha_0}{\alpha_+}), & \alpha_- + \alpha_0 \leq \tau \leq 1
\end{cases}
}
is an optimizer for $B_2$ at $x$. The proof of the fact that $\BMO$-norm of $\vf_x$ does not exceed $\eps$ is more cumbersome in this case. The arguments repeat the proof of~Proposition~5.2.9 in~\cite{ISVZ}.

\section{Proof of Theorem~\ref{BellmanTheorem}}\label{sec_proofthm21}
To prove Theorem~\ref{BellmanTheorem}, we only need to deal with the function $B_2$, because the same theorem concerning~$B_1$ was proved in~\cite{SVZ}. 

Consider the case $(r-2)(r-p)<0$. By Proposition~\ref{BasicProperties}, the function $\Bell_{p,r;\eps}^+$ is the pointwise minimal locally concave function on~$\Omega_\eps^3$ satisfying boundary condition~\eqref{eqBCsceleton} on the skeleton, whereas, by Theorem~\ref{080510}, $B_2$ is locally concave on $\Omega_\eps^3$. Thus, 
$$
B_2 \geq \Bell_{p,r;\eps}^+ \qquad \text{on}~\Omega_\eps^3.
$$ 
Let us now prove the reverse inequality. The whole domain~$\Omega_\eps^3$ is the union of $\Xi\cii{\mathrm R \pm}$, $\Xi\cii{\mathrm L \pm}$, $\Xi\cii{\mathrm{ch}\pm}$, $F(0)$, and~$R$. In Section~\ref{SecOptim}  we have constructed the optimizer $\vf_x$ for each point $x \in \Xi\cii{\mathrm R \pm} \cup \Xi\cii{\mathrm L \pm}\cup \Xi\cii{\mathrm{ch}\pm}$, therefore, by the definition of $\Bell_{p,r;\eps}^+$, we have
$$
\Bell_{p,r;\eps}^+(x) \geq \av{\vf_x}{} = B_2(x).
$$

The domain $R$ is foliated by $R(v)$, $v \in (0,2\eps)$. Each leaf $R(v)$ is a curvilinear triangle, its vertices lie on the skeleton, therefore, $B_2$ and $\Bell_{p,r;\eps}^+$ coincide at the vertices. The function $B_2$ is linear on $R(v)$  whereas $\Bell_{p,r;\eps}^+$ is locally concave there, thus $\Bell_{p,r;\eps}^+ \geq B_2$ on $R(v)$.

The same argument works for the domain $F(0)$. It is foliated by $F_0(h)$. Each $F_0(h)$ is a curvilinear triangle. Two of its vertices lie in $\Xi\cii{\mathrm R -} \cup \Xi\cii{\mathrm L +}$, and the third one is the origin. Therefore, $B_2$ and $\Bell_{p,r;\eps}^+$ coincides at the vertices. The function $B_2$ is linear on $F_0(h)$  whereas $\Bell_{p,r;\eps}^+$ is locally concave there, thus $\Bell_{p,r;\eps}^+ \geq B_2$ on $F_0(h)$.

The case $(r-2)(r-p)>0$ is completely  symmetric, all the inequalities change to the opposite, and concavity changes to convexity. Theorem~\ref{BellmanTheorem} is proved.

\section{Computation of the constant}
\label{const}

Recall that we investigate the optimal constant $C = C(p,r)$ in the inequality
\eqlb{081001}{
\|\vf\|_{L^r(I)}^{\phantom{\frac{p}{r}}}\leq
C(p,r)\|\vf\|_{L^p(I)}^{\frac{p}{r}}\|\vf\|_{\BMO(I)}^{1-\frac{p}{r}},
\qquad\vf\in\BMO,\quad\av{\vf}{I}=0.
}
In this section, we do not provide the explicit formula for the constant $C(r,p)$. However, we show that for given values $r$ and $p$ it can be computed as a unique maximum point of a known function of a single variable. 

Without loss of generality we may assume that $\|\vf\|_{\BMO} = 1$, i.\,e., set $\eps = 1$ throughout this section, and rewrite the inequality above in the form
\eqlb{290701}{
\int\limits_{I}|\varphi|^r \le C^r(p,r) \int\limits_{I}|\varphi|^p.
}

From the definition of the Bellman function (see \eqref{041102}), we get
\eqlb{240801}{
 C^r(p,r) = \sup_{x_2, x_3} \frac{\Bell_{p,r;1}^+(0,x_2,x_3)}{x_3}.
}

In this section we assume that $p > 1$. We postpone the investigation of the case $p=1$ till Section~\ref{p=1}. Till the end of this section we skip indexes $p,r,1$ and symbol $+$ in the notation $\Bell^+_{p,r;1}$, i.\,e., we set $\Bell = \Bell^+_{p,r;1}$.
We have to consider the points $x$ from the domains $F(0)$ and $R(v)$.
We will show that supremum in~\eqref{240801} is attained at the points $(0,x_2,x_3) \in F(0)$.

First, we consider $x \in R$ and show that the maximum of $\Bell / x_3$ is attained at the boundary with $F(0)$, i.\,e, at $x_3 = 2^{p-2} x_2$, see~\eqref{111201}. Indeed, considering  $x_3 \ge 2^{p-2} x_2$, from \eqref{191001} we have that $$ \frac{\Bell}{x_3} = x_3^{\frac{r-2}{p-2} - 1}x_2^{\frac{p-r}{p-2}} =
\left( \frac{x_2}{x_3} \right)^{\frac{r-p}{2-p}} \!\! \le \; 2^{r-p}.$$

Second, we consider $x \in F(0)$.
From~\eqref{220601} and~\eqref{220602} we see that the ratio $\Bell / x_3$ does not depend on the variable $x_2$ and therefore, we may set $x_2 = 1$. 
Recall that the function $\Bell(0,1,\,\cdot\,)$ is concave on the segment 
$$
J \buildrel{\eqref{201202}}\over= [2^{p-2}, \Am{p}(0, 1)]\buildrel{\eqref{eqAmang}}\over=\left[2^{p-2}, \half\Gamma(p+1) \right].
$$ 
Consider the value
\eqlb{150901}{
x_3^2 \left( \frac{\Bell}{x_3}\right)_{x_3} = x_3 \Bell_{x_3} - \Bell.
}
Clearly, this function is decreasing, since its derivative is equal to $x_3 \Bell_{x_3 x_3} < 0$.
Now, we show that the function from~\eqref{150901} attains the values of different signs at the endpoints of $J$, and thereby, it has the unique zero value inside $J$.

It is convenient to use the parameter $\xi$ instead of parameter $x_3$. Recall that $\xi$ varies from $1$ to $+\infty$, see~\eqref{220601} and~\eqref{220602}.
In the domain $F(0)$, using~\eqref{110702} from Appendix~\ref{Ap2} and~\eqref{240403}, we get
\eqlb{130901}{
\Bell_{x_3} =\frac{r(r-1)(r-2)\int\limits_{-1}^1 (1-\lambda^2)(\lambda + \xi)^{r-3} d\lambda}{p(p-1)(p-2)\int\limits_{-1}^1 (1-\lambda^2)(\lambda + \xi)^{p-3} d\lambda}.
}
Therefore, since $1 < p < r < 2$,
for $\xi = +\infty$, we deduce that $\Bell_{x_3} = +\infty$.
From \eqref{220601} and \eqref{220602}, using $\wL(+\infty)=0$, for $\xi = + \infty$ we obtain
\eqlb{130902}{
\frac{\Bell}{x_3} = \frac{\Gamma(r+1)}{\Gamma(p+1)},
}
whence, we obtain $\Bell/x_3 < \Bell_{x_3}$ at the endpoint of $J$ corresponding to $\xi = + \infty$.

At the other endpoint $\xi = 1$, the required quantities are  computed in Appendix~\ref{Ap1}
$$
\Bell_{x_3} \buildrel\eqref{140902}\over= 2^{r-p} \frac{2-r}{2-p}\, \quad\quad \text{ and } \quad\quad \frac{\Bell}{x_3} \buildrel\eqref{140903}\over= 2^{r-p}.
$$
Therefore, we have the opposite inequality $\Bell_{x_3}  < \frac{\Bell}{x_3} $.

Hence, we finished the proof that supremum of $\Bell / x_3$ is attained at the unique interior point of the interval~$J$, where the derivative of $\Bell / x_3$ with respect to $x_3$ vanishes. 

Thus, the sharp constant $C(p,r)$ can be calculated as follows: $C(p,r)^r = \frac{B^{+}_{p,r;1}(0,1,x_3)}{x_3}$, where
$$x_3 \buildrel{\eqref{220601}}\over= \frac{1}{2}\Big[\Gamma(p+1)+e\wL(\xi;p,1)\Big]$$ 
and $\xi$ is the solution of the equation $x_3 \Bell_{x_3} = \Bell$, which can be rewritten as
\eqlb{200901}
{
\frac{r(r-1)(r-2)\int\limits_{-1}^1 (1-\lambda^2)(\lambda + \xi)^{r-3} d\lambda}{p(p-1)(p-2)\int\limits_{-1}^1 (1-\lambda^2)(\lambda + \xi)^{p-3} d\lambda}
\; = \; \frac{\phantom{\int\limits_{-1}^1}\hskip-10pt\Gamma(r+1)+ e\wL(\xi;r,1)}{\phantom{\int\limits_{-1}^1}\hskip-10pt\Gamma(p+1)+ e\wL(\xi;p,1)}\,.
} 

\section{On the circle and on the line: proofs of Theorems~\ref{Ctheorem} and~\ref{Rtheorem}}
\label{sec_circle_line}
It is easy to prove inequalities~\eqref{MultCircle} and~\eqref{eq141201} having Theorem~\ref{MultiplicativeTheoremInterval} at hand, see Subsection~6.1 of~\cite{SVZ} for details. It is much harder to prove the sharpness of these inequalities. Fortunately, we can directly apply lemmas from Subsection~6.2 of~\cite{SVZ} to our problem. Namely, let $\phi_0$ be an optimizer for $B_2$ at the point $(0,1,x_3)$, where supremum in~\eqref{240801} is attained (this function $\phi_0$ can be found explicitly, see~\eqref{eq280801}). We use Lemma~6.3 of~\cite{SVZ} and for any $\delta>0$ find a 1-periodic function $\psi_0$ on $\mathbb{R}$ (i.\,e., a function on $\mathbb{T}$) such that
\begin{gather*}
\av{\psi_0}{[0,1]} = 0, \qquad \av{\psi_0^2}{[0,1]} = 1,
\\
\av{|\psi_0|^p}{[0,1]} = \av{|\phi_0|^p}{[0,1]}  + O(\delta), \qquad \delta \to 0+,
\\
\av{|\psi_0|^r}{[0,1]} = \av{|\phi_0|^r}{[0,1]}  + O(\delta), \qquad \delta \to 0+,
\\
\|\psi_0\|_{\BMO(\mathbb{T})} \leq 1 + \delta.
\end{gather*}
This means that the constant in~\eqref{MultCircle} is attained on a sequence of functions $\psi_0$ when $\delta$ tends to zero. This proves Theorem~\ref{Ctheorem}.

We use Lemma~6.4 of~\cite{SVZ} to prove the sharpness of~\eqref{eq141201} (and thereby Theorem~\ref{Rtheorem}). For any $\delta>0$ we find a function $\psi$ on $\mathbb{R}$ such that 
\begin{gather*}
\psi = 0 \quad\text{on} \quad \mathbb{R}\setminus[0,1],
\\
\av{|\psi|^p}{[0,1]} = \av{|\psi_0|^p}{[0,1]} = \av{|\phi_0|^p}{[0,1]}  + O(\delta), \qquad \delta \to 0+,
\\
\av{|\psi|^r}{[0,1]} = \av{|\psi_0|^r}{[0,1]} = \av{|\phi_0|^r}{[0,1]}  + O(\delta), \qquad \delta \to 0+,
\\
\|\psi_0\|_{\BMO(\mathbb{R})} \leq 1 + \delta.
\end{gather*}
We conclude that the constant in~\eqref{eq141201} is attained on a sequence of functions $\psi$ when $\delta$ goes to zero. This proves Theorem~\ref{Rtheorem}.

\section{Special case: $p=1$}
\label{p=1}
We will prove that the Bellman function $\Bell_{1,r;\eps}^+$ is the limit of the functions $\Bell_{p,r;\eps}^+$  when $p \to 1+$. The limit function  coincides with $B_2$ described in Subsection~\ref{B2_def}. 

First, we note that the domain $\Omega_\eps^3 = \Omega_{\eps;p}^3$ for $p=1$ is the limit of the corresponding domains in a natural sense: for any point $(x_1,x_2) \in \Omega^2_\eps$ the values $\Bell_{1;\eps}^\pm(x_1,x_2)$ that define the upper and the lower boundaries of $\Omega_{\eps;1}^3$ are the limits of $\Bell_{p;\eps}^\pm(x_1,x_2)$, when $p \to 1+$, see Subsection~\ref{2dimBell}. For $p=1$ we have $m_1(u)=1$ for $u \geq 0$, therefore, the function $\Am{1}(x_1,x_2)$ has a very simple description:
\begin{equation}
    \Am{1}(x_1,x_2)
    =
    \begin{cases}
     |x_1|, &     (x_1,x_2) \in \Omega_\eps^2 \setminus \big(\omega_0 \cup \omega_{\pm 1}\big);\\
     \frac{x_2}{2\eps}, & (x_1,x_2) \in \omega_0 \cup \omega_{\pm 1}.
    \end{cases}
\end{equation}

Second, the domains $\Xi\cii{\mathrm L+}$, $\Xi\cii{\mathrm{ch}+}$, and $F(0)$ do not exist for $p=1$ (see~\eqref{eq080501},~\eqref{eq080503}, and~\eqref{F0}). This happens because all the chords of the form $[U(a), U(b)]$, $0\leq a \leq b \leq a+2\eps$ lie on the lower boundary of~$\Omega_{\eps;1}^3$. We conclude that the whole domain $\Omega_{\eps;1}^3$ is the union of $\Xi\cii{\mathrm R+}$, its symmetric domain $\Xi\cii{\mathrm L-}$, and the domain $R$ foliated by the two-dimensional leaves $R(v)$, $v \in [0,2\eps]$.
We state that the function $B_2$ defined on these domains in Subsection~\ref{B2_def} for $p=1$ coincides with~$\Bell_{1,r;\eps}^+$.

Note that $w_2(\,\cdot \,, p, \eps, x_1, x_2)$ defined in~\eqref{141004} converges to $w_2(\,\cdot \,, 1, \eps, x_1, x_2)$ uniformly in $\xi \in [0, u]$ as~$p \to 1+$ when $\eps$ and $(x_1,x_2)$ are fixed. The limit function $w_2(\,\cdot \,, 1, \eps, x_1, x_2)$ is continuous and strictly increasing. Therefore, for every $x_3 \in \big( w_2(0, 1, \eps, x_1, x_2), w_2(u, 1, \eps, x_1, x_2))$ we have the convergence
$$
w_2\big(w_2^{-1}(x_3; p,\eps, x_1,x_2); r;\eps,x_1,x_2\big) \to w_2\big(w_2^{-1}(x_3; 1,\eps, x_1,x_2); r;\eps,x_1,x_2\big), \qquad p \to 1+.
$$
Thus, in the interior of $\Xi\cii{\mathrm R+}$ we have the pointwise convergence of the functions $B_2$ to the limit function when $p \to 1+$. Convergence on $R$ is obvious. As a result, we get convergence on the interior of $\Omega_{\eps;1}^3$. All these functions $B_2$ are locally concave for $p>1$, therefore, the limit function is also locally concave on the interior of $\Omega_{\eps;1}^3$. Since it is also continuous on $\Omega_{\eps;1}^3$, it is locally concave on the entire $\Omega_{\eps;1}^3$, therefore, due to item~3) of Proposition~\ref{BasicProperties}
$$
B_2(x) \geq \Bell_{1,r;\eps}^+(x), \qquad x \in \Omega_{\eps;1}^3.
$$

For each point $x \in \Xi\cii{\mathrm R+}$ one can construct an optimizer in the same way as it was done in~\eqref{OptXiR}. The only difference is that for $p=1$ the variable $\xi$ runs from $0$ to $\min(u,\eps)$, and by this reason we only use optimizers given by~\eqref{eq161001}. Therefore,
$$
B_2(x) \leq \Bell_{1,r;\eps}^+(x),  \qquad x \in \Xi\cii{\mathrm R+}.
$$
In the domain $R$ we use the same reasoning as before: the function $B_2$ is linear on $R(v)$ and coincides with the locally concave function $\Bell^+_{1,r;\eps}$ at the vertices of $R(v)$, therefore, $B_2 \leq \Bell_{1,r;\eps}^+$ on $R(v)$.

Thus, we have proved that $B_2 = \Bell_{1,r;\eps}^+$ everywhere on $\Omega_{\eps;1}^3$. As a consequence, we have that the statements of Theorems~\ref{MultiplicativeTheoremInterval},~\ref{Ctheorem},~\ref{Rtheorem},~\ref{Cubedtheorem}, and~\ref{DimHeatTh} and Corollaries~\ref{Rdcor},~\ref{DimCor}  hold for $p=1$.

Finally, we compute the sharp constant in the inequality \eqref{FirstMultiplicative} for $p=1$.
Repeating the argument from Section~\ref{const}, we arrive at
\eqlb{151001}{
 C^r(1,r) = \sup_{x_2, x_3} \frac{\Bell_{1,r;1}^+(0,x_2,x_3)}{x_3}\,.
}
Since the domain $F(0)$ does not exist for the case $p=1$, it suffices to find the maximum of the function $\frac{\Bell^{+}_{1,r;1}}{x_3}$ in the domain $R$. Recall that in this domain our function is defined by a simple formula ~\eqref{191001}, whence
\eqlb{151002}{
 \frac{\Bell_{1,r;1}^+(0,x_2,x_3)}{x_3} = \Big( \frac{x_2}{x_3} \Big)^{r-1}\!\!.
}
From~\eqref{111201} we have $x_3 \ge \half x_2$ and the maximum value is attained on the lower boundary, i.\,e.\,, $x_3 = \half x_2$, which implies
$$
C(1,r) = 2^{1-\frac1r}.
$$
Thus, we have proved the formula stated in Remark~\ref{Rem1} for the case $1 < r \leq 2$. For $r\geq 2$ it was proved in~\cite{SV2}, see~\eqref{eq091101}.

\section{Comments on multidimensional case}
The proof of Theorem~\ref{DimHeatTh} follows the reasoning from~\cite{SlavinZ}. Following the arguments of Lemma~5.2 there it is possible to prove the following assertion.  If $G$ is a non-negative locally concave function on~$\Omega_\eps^3$ and $\phi \in \BMO(\mathbb{R}^d)$, $\|\phi \|_{\scriptscriptstyle \Kerr}<\eps$, then 
$$
G\Big(\phi_{\scriptscriptstyle \Kerr}(y,t), \big(\phi^2\big)_{\scriptscriptstyle \Kerr}(y,t), \big(|\phi|^p\big)_{\scriptscriptstyle \Kerr}(y,t)\Big) \geq 
\big(G(\phi, \phi^2, |\phi|^p)\big)_{\scriptscriptstyle \Kerr}(y,t), \qquad y \in \mathbb{R}^d, \quad t>0.
$$
Fix a point $(y,t) \in \mathbb{R}^d\times\mathbb{R}_+$ and use this inequality for the function $G = \Bell_{p,r;\eps}^+$ and for $\phi = \vf - \vf_{\scriptscriptstyle \Kerr}(y,t)$. Then we obtain
\eqlb{eq150901}{
\Bell_{p,r;\eps}^+\Big(0, \big((\vf - \vf_{\scriptscriptstyle \Kerr}(y,t))^2\big)_{\scriptscriptstyle \Kerr}(y,t), \big(|\vf-\vf_{\scriptscriptstyle \Kerr}(y,t)|^p\big)_{\scriptscriptstyle \Kerr}(y,t)\Big) \geq 
\big(|\vf-\vf_{\scriptscriptstyle \Kerr}(y,t)|^r\big)_{\scriptscriptstyle \Kerr}(y,t).
}
Now, we use the definitions of the Bellman function $\Bell_{p,r;\eps}^+$ and the constant $C(p,r)$ in the corresponding inequality. We estimate the left-hand side of~\eqref{eq150901}:
$$
\Bell_{p,r;\eps}^+\Big(0, \big((\vf - \vf_{\scriptscriptstyle \Kerr}(y,t))^2\big)_{\scriptscriptstyle \Kerr}(y,t), \big(|\vf-\vf_{\scriptscriptstyle \Kerr}(y,t)|^p\big)_{\scriptscriptstyle \Kerr}(y,t)\Big) \leq C(p,r)^r  \eps^{r-p} \big(|\vf-\vf_{\scriptscriptstyle \Kerr}(y,t)|^p\big)_{\scriptscriptstyle \Kerr}(y,t).
$$
Therefore, we obtain
$$
\eps^{r-p} C(p,r)^r  \big(|\vf-\vf_{\scriptscriptstyle \Kerr}(y,t)|^p\big)_{\scriptscriptstyle \Kerr}(y,t)  \geq 
\big(|\vf-\vf_{\scriptscriptstyle \Kerr}(y,t)|^r\big)_{\scriptscriptstyle \Kerr}(y,t).
$$
Rewrite the last inequality in the form
\eqlb{eq150902}{
\eps^{r-p} C(p,r)^r  \int\limits_{\mathbb{R}^d} |\vf(\tilde y)-\vf_{\scriptscriptstyle \Kerr}(y,t)|^p \Kerr_t(y-\tilde y) d \tilde y  \geq 
\int\limits_{\mathbb{R}^d} |\vf(\tilde y)-\vf_{\scriptscriptstyle \Kerr}(y,t)|^r \Kerr_t(y-\tilde y) d \tilde y.
}
At this moment we recall that $\Kerr_t$ coincides either with $\KerP$ or $\KerH$. For each of these kernels there exists the function 
$$
Q(t) =
\begin{cases}
 \frac{\pi^{\frac{d+1}{2}}}{\Gamma(\frac{d+1}{2})} t^d,&  \Kerr=\KerP,\\
  (4\pi t)^{\frac{d}{2}}, &  \Kerr=\KerH,
\end{cases}
$$
such that $Q(t)K_t(y) \to 1$ for any $y \in \mathbb{R}^d$  when $t \to +\infty$. Moreover, this convergence is monotone and uniform on compact subsets of $\mathbb{R}^d$.
We also note that since the function $\vf$ lies in $L^p(\mathbb{R}^d)$, we have $\vf_{\scriptscriptstyle \Kerr}(y,t) \to 0$ for any $y \in \mathbb{R}^d$ when~$t \to +\infty$. Multiplying both sides of~\eqref{eq150902} by $Q(t)$ and passing~$t \to +\infty$, we finally get
\eqlb{eq150903}{
\eps^{r-p} C(p,r)^r  \int_{\mathbb{R}^d} |\vf(\tilde y)|^p  d \tilde y  \geq 
\int_{\mathbb{R}^d} |\vf(\tilde y)|^r d \tilde y.
}
Theorem~\ref{DimHeatTh} is proved.

\begin{appendices}

\section{Smoothness of the Bellman function $B_2$}\label{Ap1}
In this appendix we prove $C^1$-smoothness of $B_2$. Since we will deal only with this candidate we omit index $2$. 
Then it will be convenient to denote the partial derivatives of $B$ placing the corresponding 
variable in index: $B_{x_i}=\diff{B}{x_i}$.

\begin{Le}
\label{130601}
The function $B$ is $C^1$-smooth.
\end{Le}

\begin{proof}
By the definition, it is clear that $B$ is $C^2$-smooth on each subdomain. Therefore, we need to check
that this function is $C^1$-smooth on the borders between subdomains. Since all borders consists of 
extremals, $B$ is linear in corresponding direction. Therefore, it suffices to verify continuity
of the derivatives in any two transversal directions. We will check continuity of $B_{x_3}$ and $B_{x_2}$.
The only exception is the boundary between $\Xi\cii{\mathrm L+}$ and $F(0)$ that goes along $x_3$ (it belongs
to the plane $x_2=2\eps x_1$). In this case we check continuity of $B_{x_1}$ instead of $B_{x_3}$.
 
Let us start with the domain $\Xi\cii{\mathrm L+}$. On the boundary $x_2=2\eps x_1$ (or more exactly $x_1=\Delta_-$)
the function $B$ is continuous by definition: we have defined $B$ on $F(0)$ using this continuity.

{\bf Boundary $\mathbf{\Xi\cii{\mathrm L+}\Big| F(0)}$}. To check $C^1$-smoothness, we calculate the derivative $B_{x_1}$ and show 
that it is zero on the boundary
$x_1=\Delta_-$. It is what we need because the function $B$ does not depend 
on $x_1$ on the domain $F(0)$.
We use the representation of $B$ in it initial form~\eqref{eq3}:
\eqlb{190601}{
B_{x_1}=rv^{r-1}\diff v{x_1}+\Big(\diff\KL{u}\cdot\diff u{x_1}+\diff\KL{h}
\cdot\diff h{x_1}\Big)(x_1-v)+
\KL\Big(1-\diff v{x_1}\Big),
}
where $v=u-\eps=x_1-\Delta_-$, therefore,
$$
\diff v{x_1}=\diff u{x_1}=\frac{d+x_1}d\,,
$$
and $v=0$ on the boundary. To find the derivative $\diff h{x_1}$ we use formula~\eqref{161001} and 
the relation $\eps m_p'(v)=m_p(v)-pv^{p-1}$ (see~\eqref{m-diff}). As a result, we get
$$
\diff h{x_1}\Big|_{x_1=\Delta_-}\!\!\!=\ \frac hd\,.
$$ 
From~\eqref{eq9} we see that 
$$
h\diff\KL{h}\big|_{x_1=\Delta_-} = \KL-\eps\diff\KL{u}.
$$

Gathering all these formulas, we obtain
\eqlb{eq061001}{
B_{x_1}\big|_{x_1=\Delta_-}=\diff\KL{u}\cdot\frac{d+x_1}d\cdot x_1+\Big(\KL-\eps\diff\KL{u}\Big)\cdot\frac{x_1}d
-\frac{x_1}d\KL=\diff\KL{u}\cdot\frac{x_1}d\cdot(x_1+d-\eps)=0\,,
}
since $\eps-d=\Delta_-=x_1$ on this boundary.

To prove continuity of $B_{x_2}$, note that from~\eqref{eq051002} for the points in the intersection $F(0) \cap \Xi\cii{\mathrm L+} = F_L(\eps)$ we have 
$$
B_{x_2}\Big|_{_{F(0)}} \!\!\! =\  \frac{1}{2\eps}B_{x_1}\Big|_{_{\Xi\cii{\mathrm L+}}}\!\!\! + \ B_{x_2}\Big|_{_{\Xi\cii{\mathrm L+}}},
$$
and the first summand at the right-hand vanishes due to~\eqref{eq061001}.

{\bf Boundary $\mathbf{F(0)| R}$}. Now, we check $C^1$-smoothness of $B$ on the boundary $x_3=(2\eps)^{p-2}x_2$ between~$F(0)$ and 
$R$ (see~\eqref{201202}). According to~\eqref{191001} the value of $B$ from $R$ is $(2\eps)^{r-2}x_2$.
The value from~$F(0)$ looks much more difficult (see~\eqref{071011}):
$$
\Big[e\PsiL\Big(\frac{(2\eps)^p}{2e}-\frac{\eps^p}e\Gamma(p+1)\Big)+\eps^{r-1}\Gamma(r+1)\Big]\frac{x_2}{2\eps}\,.
$$
However, the fact that these two expressions are equal is known: this is just formula~\eqref{220901}
for $u=\eps$.

The calculations will be simpler if we use representation~\eqref{220601}--\eqref{220602} for the function
on $F(0)$. Then the boundary $x_3=(2\eps)^{p-2}x_2$ corresponds to $\xi=\eps$ and we have
\eqlb{140903}{
\wL(\eps;s,\eps)=\frac{(2\eps)^s}{2e}-\frac{\eps}e m_s(0)\qquad\text{and}\qquad B(x)=(2\eps)^{r-2}x_2\,.
}

Since this boundary is a piece of two-dimensional plane where $B$ is linear, for continuity of the
gradient it is sufficient to check continuity of $B_{x_3}$ only. On $R$ we have (see~\eqref{191001}):
$$
B_{x_3}=\frac{r-2}{p-2}\Big(\frac{x_3}{x_2}\Big)^{\frac{r-p}{p-2}}\qquad\text{and}\qquad
B_{x_3}\big|_{x_3=(2\eps)^{p-2}x_2}=\frac{r-2}{p-2}(2\eps)^{r-p}\,.
$$
On $F(0)$ we use the same representation~\eqref{220601}--\eqref{220602} and get
\eqlb{110701}{
B_{x_3}=\frac{\wL'(\xi;r,\eps)}{\wL'(\xi;p,\eps)}\,.
}
Formula~\eqref{dwu} yields $\wL'(\eps;s,\eps)=\tfrac1{2e}(s-2)(2\eps)^{s-1}$, whence
\eqlb{140902}
{
B_{x_3}\big|_{\xi=\eps}=\frac{r-2}{p-2}(2\eps)^{r-p}\,.
}

 We have considered all boundaries of $F(0)$ and now we consider boundaries of $\Xi\cii{\mathrm L+}$.
It has two neighbours: already considered $F(0)$ and $\Xi\cii{\mathrm{ch}+}$, which we have to check.

{\bf Boundary $\mathbf{\Xi\cii{\mathrm L+}\Big| \Xi\cii{\mathrm{ch}+}}$}. This
is the boundary
$$
x_3=\frac{\Delta_-(x_1+\Delta_+)^p+\Delta_+(x_1-\Delta_-)^p}{2\eps}
$$
(see~\eqref{201204}--\eqref{201206}). This boundary is foliated by the left halves of the chords 
$[U(u-\eps),U(u+\eps)]$, where~$B$ is linear with the prescribed values at the ends. Therefore, it
takes the same values not depending on what side of the boundary we consider our function. However, we
can verify this formally plugging the values $\xi=u=v+\eps=x_1+d$ into~\eqref{eq080505}:
$$
B(x)=(u-\eps)^r+\Big[m_r(u-\eps)+\frac1{2\eps}\big((u+\eps)^r-(u-\eps)^r\big)-m_r(u-\eps)\Big]\Delta_-=
\frac{\Delta_-(x_1+\Delta_+)^r+\Delta_+(x_1-\Delta_-)^r}{2\eps}\,.
$$
This expression coincides with~\eqref{090601}, since $a=u-\eps=x_1-\Delta_-$ and $b=u+\eps=x_1+\Delta_+$.

Now, we have to check continuity of $B_{x_3}$ and $B_{x_2}$ on this boundary. In $\Xi\cii{\mathrm L+}$ we have:
\eqlb{230602}{
B_{x_3}{\buildrel\eqref{050801}\over=}\eps\PsiL'(\xi)
{\buildrel\eqref{Psi_L_def}\over=}\frac{\wL'(\xi;r,\eps)}{\wL'(\xi;p,\eps)}{\buildrel\eqref{280401}\over=}
\frac{\AAA(\xi,\eps,r)}{\AAA(\xi,\eps,p)}\,.
}
On $\Xi\cii{\mathrm{ch}+}$ formula~\eqref{261211} yields
\begin{align}
\label{130701}
1{\buildrel\eqref{231002}\over=}
w_{x_3}(a,b;p) &= \frac{a_{x_3}(b-x_1)}{(b-a)^2}\AAA(\half(b+a),\half(b-a),p),
\\
\label{130702}
B_{x_3}{\buildrel\eqref{090601}\over=}
w_{x_3}(a,b;r) &= \frac{a_{x_3}(b-x_1)}{(b-a)^2}\AAA(\half(b+a),\half(b-a),r),
\end{align}
whence
\eqlb{230603}{
B_{x_3}=\frac{\AAA(\half(b+a),\half(b-a),r)}{\AAA(\half(b+a),\half(b-a),p)}\,.
}
Therefore, on the line $\xi=u$, $b=u+\eps$, $a=u-\eps$, the expressions from~\eqref{230602} and~\eqref{230603}
are equal to
$$
\frac{\AAA(u,\eps,r)}{\AAA(u,\eps,p)}\,.
$$

On this boundary, it remains to check continuity of $B_{x_2}$. In $\Xi\cii{\mathrm L+}$ we use 
formula~\eqref{eq12} rewritten in terms of $\wL$ with the help of~\eqref{eq081002} and~\eqref{Psi_L_def}:
$$
B_{x_2}(x)=\frac1{2\eps}\Big[\frac1\eps e^{\frac u\eps}\wL(\xi;r,\eps)-
\frac1\eps e^{\frac u\eps}\wL(\xi;p,\eps)\frac{\wL'(\xi;r,\eps)}{\wL'(\xi;p,\eps)}+\eps m_r'(v)-
\eps m_p'(v)\frac{\wL'(\xi;r,\eps)}{\wL'(\xi;p,\eps)}\Big]\,.
$$
On the boundary $\xi=u=v+\eps=x_1+d$ we use relations~\eqref{w_L_def}, \eqref{280401}, \eqref{m-diff} and 
rewrite this formula as follows:
\eqlb{250602}{
B_{x_2}(x)\big|_{\xi=u}=\frac1{2\eps}\Big[\Big(\frac{(u+\eps)^r-(u-\eps)^r}{2\eps}-rv^{r-1}\Big)-
\Big(\frac{(u+\eps)^p-(u-\eps)^p}{2\eps}-pv^{p-1}\Big)\frac{\AAA(u,\eps,r)}{\AAA(u,\eps,p)}\Big]\,.
}

In $\Xi\cii{\mathrm{ch}+}$ we differentiate formula~\eqref{230601} with respect to $x_2$ and as in~\eqref{261211}
we get
\eqlb{260601}{
%\begin{aligned}
w_{x_2}
%=&\frac{\big[a^sb_{x_2}\!\!+\!(b\!-\!x_1)sa^{s-1}a_{x_2}\!\!-\!b^sa_{x_2}\!\!+\!(x_1\!-\!a)sb^{s-1}b_{x_2}\big](b-a)\!-\!
%(b_{x_2}\!\!-\!a_{x_2})\big[(b\!-\!x_1)a^s\!\!+\!(x_1\!\!-\!a)b^s\big]}{(b-a)^2}
%\\
=\frac{a_{x_2}(b-x_1)\big[sa^{s-1}(b-a)+a^s-b^s\big]+b_{x_2}(x_1-a)\big[sb^{s-1}(b-a)+a^s-b^s\big]}{(b-a)^2}\,.
%\end{aligned}
}
To obtain relation between $a_{x_2}$ and $b_{x_2}$ we differentiate~\eqref{231001}:
\eqlb{260602}{
1=b_{x_2}(x_1-a)-a_{x_2}(b-x_1)\,.
}
Now we replace the expression $a_{x_2}(b-x_1)$ in~\eqref{260601} using~\eqref{260602} and get
$$
\begin{aligned}
w_{x_2}
=&\frac{b_{x_2}(x_1-a)\big[s(b-a)(a^{s-1}+b^{s-1})+2a^s-2b^s\big]-\big[sa^{s-1}(b-a)+a^s-b^s\big]}{(b-a)^2}
\\
=&\frac{b_{x_2}(x_1-a)}{(b-a)^2}\AAA\big(\half(b+a),\half(b-a),s\big)-\frac{sa^{s-1}(b-a)+a^s-b^s}{(b-a)^2}\,.
\end{aligned}
$$
Whence,
\begin{align}
\label{260603}
B_{x_2}{\buildrel\eqref{090601}\over=}
w_{x_2}(a,b;r) &=\frac{b_{x_2}(x_1-a)}{(b-a)^2}\AAA\big(\half(b+a),\half(b-a),r\big)-\frac{ra^{r-1}(b-a)+a^r-b^r}{(b-a)^2},
\\
\label{260604}
0{\buildrel\eqref{231002}\over=}
w_{x_2}(a,b;p) &= \frac{b_{x_2}(x_1-a)}{(b-a)^2}\AAA\big(\half(b+a),\half(b-a),p\big)-\frac{pa^{p-1}(b-a)+a^p-b^p}{(b-a)^2}.
\end{align}
Therefore,
\eqlb{260605}{
B_{x_2}=\frac{b^r-a^r-ra^{r-1}(b-a)}{(b-a)^2}-\frac{b^p-a^p-pa^{p-1}(b-a)}{(b-a)^2}\cdot
\frac{\AAA\big(\half(b+a),\half(b-a),r\big)}{\AAA\big(\half(b+a),\half(b-a),p\big)}\,.
}
Since on the boundary we have $b=u+\eps$, $a=u-\eps=v$, we see that the expressions in~\eqref{250602}
and in~\eqref{260605} coincide.

Now we consider two remaining boundaries of $\Xi\cii{\mathrm{ch}+}$. This is the boundary $x_3=x_2^{p-1}x_1^{2-p}$ with
$R$ (if $(x_1,x_2)\in\omega_2\cup\omega_3$, see~\eqref{201204} and~\eqref{201205}) and the boundary
$x_3=\frac{\Delta_-(x_1-\Delta_+)^p+\Delta_+(x_1+\Delta_-)^p}{2\eps}$ with $\Xi\cii{\mathrm R+}$
(if $(x_1,x_2)\in\omega_4$, see~\eqref{201206}).

{\bf Boundary $\mathbf{\Xi\cii{\mathrm{ch}+}\Big| R}$}. This is the boundary $x_3=x_2^{p-1}x_1^{2-p}$. It consists of the chords with $[U(0),U(b)]$, $b\in[0,2\eps]$: $x_2=bx_1$, 
$x_3=b^{p-1}x_1$. The value of our function on this chord is $B=b^{r-1}x_1$ in both domains.

The derivatives on $R$ (see~\eqref{191001}) are:
\begin{align}
\label{270601}
B_{x_3}=\frac{r-2}{p-2}\Big(\frac{x_3}{x_2}\Big)^{\frac{r-p}{p-2}}=\frac{r-2}{p-2}\cdot b^{r-p},
\\
\label{270602}
B_{x_2}=\frac{p-r}{p-2}\Big(\frac{x_3}{x_2}\Big)^{\frac{r-2}{p-2}}=\frac{p-r}{p-2}\cdot b^{r-2}.
\end{align}
Since $\AAA(\half b,\half b,s){\buildrel\eqref{051003}\over=}(s-2)b^s$, on this boundary the value of~\eqref{230603} coincides with~\eqref{270601}
and the value of~\eqref{260605} coincides with~\eqref{270602}.

{\bf Boundary $\mathbf{\Xi\cii{\mathrm{ch}+}\Big| \Xi\cii{\mathrm R+}}$}. This boundary is foliated by the right half of 
the chords $[U(u-\eps),U(u+\eps)]$ ($u\ge\eps$), and $B$ is linear on such chord with the prescribed 
values at the ends. Therefore, the boundary values from both sides of the boundary are the same. 

Now, we check continuity of $B_{x_3}$. On $\Xi\cii{\mathrm{ch}+}$ these derivatives are already calculated, 
we need only to plug the boundary values $b=u+\eps$ and $a=u-\eps$ into~\eqref{230603}:
\eqlb{290601}{
B_{x_3}=\frac{\AAA(u,\eps,r)}{\AAA(u,\eps,p)}\,.
}
On $\Xi\cii{\mathrm R+}$ we get the same expression if we substitute $\xi = u$ to the following formula:
\eqlb{eq091001}{
B_{x_3}\Big|_{\Xi\cii{\mathrm R+}} 
{\buildrel\eqref{250601}\over=} -\eps \PsiR' 
{\buildrel\eqref{Psi_R_def}\over=}  \frac{\wR'(\xi; r,\eps)}{\wR'(\xi; p,\eps)} 
{\buildrel\eqref{w_ww_relation}\over=} \frac{\wL'(\xi; r,\eps)}{\wL'(\xi; p,\eps)}
{\buildrel\eqref{280401}\over=} \frac{\AAA(\xi,\eps,r)}{\AAA(\xi,\eps,p)}, \qquad \xi \geq \eps.
}

Verifying continuity of $B_{x_2}$, we use~\eqref{eq081003} and~\eqref{Psi_R_def} to rewrite the second formula of~\eqref{250601} in terms of $\wR$:
\eqlb{040701}{
B_{x_2}(x)=\frac1{2\eps}\Big[\Big(\eps k_r'(v)+\frac1\eps e^{-\frac u\eps}\wR(\xi;r,\eps)\Big)-
\Big(\eps k_p'(v)+\frac1\eps e^{-\frac u\eps}\wR(\xi;p,\eps)\Big)\frac{\wR'(\xi;r,\eps)}{\wR'(\xi;p,\eps)}\Big]\,.
}
On the boundary $\xi=u=v-\eps=x_1-d$ we apply already used relations~\eqref{220904}, \eqref{w_ww_relation}, \eqref{280401}, 
\eqref{m-diff}, and rewrite this formula as follows:
\eqlb{290602}{
B_{x_2}(x)\big|_{\xi=u}=\frac1{2\eps}\Big[\Big(\frac{(u-\eps)^r-(u+\eps)^r}{2\eps}+rv^{r-1}\Big)-
\Big(\frac{(u-\eps)^p-(u+\eps)^p}{2\eps}+pv^{p-1}\Big)\frac{\AAA(u,\eps,r)}{\AAA(u,\eps,p)}\Big]\,.
}
On the face of it, this expression differs from~\eqref{260605} with $b=u+\eps=v$, $a=u-\eps$. However, in fact
they are equal. It can be verified by the direct calculation using the definition of the function $\AAA$ 
(see~\eqref{051003}). But we check this in other way, deducing an alternative formula
for $B_{x_2}$ on $\Xi\cii{\mathrm{ch}+}$. Namely, in~\eqref{260601} we remove not $a_{x_2}$ but $b_{x_2}$
using the same relation~\eqref{260602}:
$$
\begin{aligned}
w_{x_2}
=&\frac{a_{x_2}(b-x_1)\big[s(b-a)(a^{s-1}+b^{s-1})+2a^s-2b^s\big]+\big[sb^{s-1}(b-a)+a^s-b^s\big]}{(b-a)^2}
\\
=&\frac{a_{x_2}(b-x_1)}{(b-a)^2}\AAA\big(\half(b+a),\half(b-a),s\big)+\frac{ba^{s-1}(b-a)+a^s-b^s}{(b-a)^2}\,.
\end{aligned}
$$
Then instead of~\eqref{260603} and~\eqref{260604} we get
\begin{align}
\label{020701}
B_{x_2}{\buildrel\eqref{090601}\over=}
w_{x_2}(a,b;r) &=\frac{a_{x_2}(b-x_1)}{(b-a)^2}\AAA\big(\half(b+a),\half(b-a),r\big)+\frac{rb^{r-1}(b-a)+a^r-b^r}{(b-a)^2},
\\
\label{020702}
0{\buildrel\eqref{231002}\over=}
w_{x_2}(a,b;p) &=\frac{a_{x_2}(b-x_1)}{(b-a)^2}\AAA\big(\half(b+a),\half(b-a),p\big)+\frac{pb^{p-1}(b-a)+a^p-b^p}{(b-a)^2}.
\end{align}
Therefore,
\eqlb{020703}{
B_{x_2}=-\frac{b^r-a^r-rb^{r-1}(b-a)}{(b-a)^2}+\frac{b^p-a^p-pb^{p-1}(b-a)}{(b-a)^2}\cdot
\frac{\AAA\big(\half(b+a),\half(b-a),r\big)}{\AAA\big(\half(b+a),\half(b-a),p\big)}\,,
}
and this expression evidently coincides with~\eqref{290602} on the boundary $b=u+\eps=v$, $a=u-\eps$.

It remains to check one boundary, namely, between $\Xi\cii{\mathrm R+}$ and $R$. 

{\bf Boundary $\mathbf{\Xi\cii{\mathrm R+}\Big| R}$}. This is the boundary
$x_3=(x_1+\Delta_-)^{p-2}x_2$ or $\xi=u=v-\eps=x_1-d$ and it is foliated by the extremal lines $x_3=v^{p-2}x_2$,
$x_2=2(v-\eps)x_1-v^2+2v\eps$, $\eps\le v\le2\eps$. On this boundary we have to use the first line of definition 
of $\wR$~\eqref{220904} for $\xi\le\eps$. Therefore, (see~\eqref{eq080507}), on any such extremal line we have
$$
B=v^r-\big(k_r(v)-k_r(v)+2uv^{r-2}\big)\Delta_-=v^{r-2}(v^2-2u\Delta_-)=v^{r-2}x_2\,,
$$
that coincides with the value on $R$ (see~\eqref{191002}).

Using as before~\eqref{Psi_R_def} and the first formula in~\eqref{250601} for $B_{x_3}$ on $\Xi\cii{\mathrm R+}$, we get
$$
B_{x_3}=\frac{\wR'(\xi;r,\eps)}{\wR'(\xi;p,\eps)}{\buildrel\eqref{ww_dif}\over=}\frac{r-2}{p-2}(\xi+\eps)^{r-p},
$$
and for $\xi=u=v-\eps$ we get the same expression as in~\eqref{270601} for $x_3=v^{p-2}x_2$.
The second formula in~\eqref{250601} for $B_{x_2}$ yields~\eqref{040701}, where we have use another
expression for $\wR$ (see the first line of~\eqref{220904}). As a result, we get not~\eqref{290602}, but
\begin{align*}
B_{x_2}\big|_{\xi=u}=&\frac1{2\eps}\Big[\big(rv^{r-1}-2uv^{r-2}\big)
-\big(pv^{p-1}-2uv^{p-2}\big)\frac{(r-2)v^{r-3}}{(p-2)v^{p-3}}\Big]
\\
=&\frac1{2\eps}\Big[\big((r-2)v^{r-1}+2\eps v^{r-2}\big)
-\big((p-2)v^{p-1}+2\eps v^{p-2}\big)\frac{r-2}{p-2}v^{r-p}\Big]=\frac{p-r}{p-2}v^{r-2}\,,
\end{align*}
therefore, we get the same expression as in~\eqref{270602} for $x_3=v^{p-2}x_2$.

The only boundary intersecting the plane $x_1=0$ is the boundary between $F(0)$ and $R$ considered above. $C^1$-smoothness of $B$ in the domain $x_1<0$ follows from the symmetry. This completes the proof of $C^1$-smoothness of the function $B$.
\end{proof}

\section{Convexity/concavity of the Bellman function $B_2$}\label{Ap2}
In this appendix we prove convexity/concavity of $B_2$. Since we will deal only with this candidate we omit index $2$, as it was done in Appendix~\ref{Ap1}.

\begin{Le} 
\label{200401}
The equality $\sign B_{x_3x_3}=\sign(r-2)(r-p)$ holds on subdomains $F(0),$ $R,$ $\Xi\cii{\mathrm L\pm},$ $\Xi\cii{\mathrm R\pm},$ $\Xi\cii{\mathrm ch\pm}$.
\end{Le}

\begin{proof}
We start with $F(0)$, where by~\eqref{110701} we have
\eqlb{110702}{
B_{x_3}=\frac{\wL'(\xi;r,\eps)}{\wL'(\xi;p,\eps)}\buildrel\eqref{280401}\over=\frac{\AAA(\xi,\eps,r)}{\AAA(\xi,\eps,p)}.
}
The same expression we have in~$\Xi\cii{\mathrm L+}$, as we have seen in~\eqref{230602}. In~$\Xi\cii{\mathrm R+}$ 
we have the same formula for $\xi\ge\eps$ due to~\eqref{eq091001}.
Therefore, in all three cases we have
$$
B_{x_3x_3}=\left(\frac{\AAA(\xi,\eps,r)}{\AAA(\xi,\eps,p)}\right)'\!\cdot\diff{\xi}{x_3}\,.
$$

Recall that we use symbol $'$ to denote differentiation with respect to $\xi$. In the cases under consideration we have
\begin{align*}
F(0):\qquad x_3'&\buildrel\eqref{220601}\over=\frac{ex_2}{2\eps^2}\wL'(\xi;p,\eps)\,;
\\
\Xi\cii{\mathrm R+}:\qquad x_3'&\buildrel\eqref{141003}\over=\frac{\Delta_-}\eps e^{-\frac u\eps}\wR'(\xi;p,\eps)
\buildrel\eqref{w_ww_relation}\over=\frac{\Delta_-}\eps e^{\frac{2\xi-u}{\eps}}\wL'(\xi;p,\eps), \qquad \xi \geq \eps\,;
\\
\Xi\cii{\mathrm L+}:\qquad x_3'&\buildrel\eqref{141001}\over=\frac{\Delta_-}\eps e^{\frac u\eps}\wL'(\xi;p,\eps)\,.
\end{align*}
Therefore, in all cases Lemma~\ref{lemmaSignW_r} implies
$$
\sign\diff{\xi}{x_3}=\sign x_3'=\sign\wL'(\xi;p,\eps)=\sign(p-2)\,,
$$
and we need to check that
\eqlb{120701}{
\sign\left(\frac{\AAA(\xi,\eps,r)}{\AAA(\xi,\eps,p)}\right)'=
\sign\big(\AAA'(\xi,\eps,r)\AAA(\xi,\eps,p)-\AAA'(\xi,\eps,p)\AAA(\xi,\eps,r)\Big)=
\sign(p-2)(r-2)(r-p)\,.
}

Differentiating~\eqref{240403} with respect to $\alpha$ and then integrating by parts, we get
\eqlb{070501}{
\diff{}\alpha\AAA(\alpha,\beta,s)=
s(s-1)(s-2)\!\!\int\limits^{\;\beta}_{-\beta\ \ }\!\!2\lambda(\lambda+\alpha)^{s-3}d\lambda\,.
}
After dividing the expression in~\eqref{120701} over $r(r-1)(r-2)p(p-1)(p-2)$, we need to verify
that $\sign(r-p)$ is the sign of the following expression:
\begin{align*}
&\int\limits^{\;\beta}_{-\beta\ \ }\!\!2\lambda(\lambda+\alpha)^{r-3}\,d\lambda
\!\!\int\limits^{\;\beta}_{-\beta\ \ }\!\!(\beta^2-\lambda^2)(\lambda+\alpha)^{p-3}d\lambda-
\!\!\int\limits^{\;\beta}_{-\beta\ \ }\!\!2\lambda(\lambda+\alpha)^{p-3}d\lambda
\!\!\int\limits^{\;\beta}_{-\beta\ \ }\!\!(\beta^2-\lambda^2)(\lambda+\alpha)^{r-3}\,d\lambda
\\
&=2\!\!\int\limits^{\;\beta}_{-\beta\ \ }\!\!\!\!\int\limits^{\;\beta}_{-\beta\ \ }
\!\!\big[\lambda(\beta^2-\mu^2)-\mu(\beta^2-\lambda^2)\big]
(\mu+\alpha)^{p-3}(\lambda+\alpha)^{r-3}\,d\lambda\,d\mu
\\
&=2\!\!\int\limits^{\;\beta}_{-\beta\ \ }\!\!\!\!\int\limits^{\;\beta}_{-\beta\ \ }
\!\!(\lambda-\mu)(\beta^2+\lambda\mu)
(\mu+\alpha)^{p-3}(\lambda+\alpha)^{r-3}\,d\lambda\,d\mu\,.
\end{align*}
After symmetrization (interchanging $\mu$ and $\lambda$) we get:
$$
\int\limits^{\;\beta}_{-\beta\ \ }\!\!\!\!\int\limits^{\;\beta}_{-\beta\ \ }
\!\!(\beta^2+\lambda\mu)(\mu+\alpha)^{p-3}(\lambda+\alpha)^{p-3}
\Big[(\lambda-\mu)\big((\lambda+\alpha)^{r-p}-(\mu+\alpha)^{r-p}\big)\Big]\,d\lambda\,d\mu\,.
$$
Since the function $t\mapsto t^s$ is increasing for $s>0$ and decreasing for $s<0$, the sign of expression 
in the square brackets coincides with $\sign(r-p)$, and this is just what we need, because all other terms
are positive.

To complete the proof for these three domains, it remains to consider the subdomain of $\Xi\cii{\mathrm R+}$ with
$\xi<\eps$:
$$
B_{x_3}\Big|_{\Xi\cii{\mathrm R+}} 
{\buildrel\eqref{250601}\over=} -\eps \PsiR' 
{\buildrel\eqref{Psi_R_def}\over=}  \frac{\wR'(\xi; r,\eps)}{\wR'(\xi; p,\eps)} 
{\buildrel\eqref{ww_dif}\over=}\frac{r-2}{p-2}(\xi+\eps)^{r-p},\qquad \ \xi<\eps\,,
$$
whence $\sign B_{x_3\xi}=\sign(p-2)(r-2)(r-p)$. Since $\sign\diff\xi{x_3}=\sign\wR'(\xi; p,\eps) = \sign(p-2)$ by Lemma~\ref{ww_increase}, we obtain
the required property: $\sign B_{x_3x_3}=\sign(r-2)(r-p)$.

Now, we start to calculate the sign of $B_{x_3x_3}$ on the domain $\Xi\cii{\mathrm{ch}+}$. On this domain we will use notation $\alpha=\half(b+a)$ and $\beta=\half(b-a)$.
Due to 
formula~\eqref{230603} we have
$$
B_{x_3x_3}=\diff{}{x_3}\left(\frac{\AAA(\alpha,\beta,r)}{\AAA(\alpha,\beta,p)}\right)=
\frac{\AAA_{x_3}(r)\AAA(p)-\AAA(r)\AAA_{x_3}(p)}{\AAA(p)^2}\,.
$$
For brevity, we omit arguments $\alpha$ and $\beta$
if this does not lead to misunderstanding.
The principal step will be the same as before, but we need some auxiliary calculations here.
First of all, we will use another integral form of the function $\AAA$. Changing the variable of
integration in~\eqref{240403}, we get the following formula
\eqlb{270801}{
\AAA(s)=s(s-1)(s-2)\int\limits_a^b(b-\lambda)(\lambda-a)\lambda^{s-3}d\lambda\,,
}
whence
\eqlb{270802}{
\AAA_{x_3}(s)=s(s-1)(s-2)\int\limits_a^b\big[b_{x_3}(\lambda-a)-
(b-\lambda)a_{x_3}\big]\lambda^{s-3}d\lambda\,.
}
From~\eqref{130701} and~\eqref{eq300403} we have
\eqlb{121002}{
(x_1-a)b_{x_3}=(b-x_1)a_{x_3}=\frac{(b-a)^2}{\AAA(p)},
}
therefore,~\eqref{270802} can be rewritten as follows:
\eqlb{030906}{
\begin{aligned}
\AAA_{x_3}(s)&=\frac{s(s-1)(s-2)(b-a)^2}{(x_1-a)(b-x_1)\AAA(p)}
\int\limits_a^b\big[(b-x_1)(\lambda-a)-(b-\lambda)(x_1-a)\big]\lambda^{s-3}d\lambda
\\
&=\frac{s(s-1)(s-2)(b-a)^3}{(x_1-a)(b-x_1)\AAA(p)}
\int\limits_a^b(\lambda-x_1)\lambda^{s-3}d\lambda\,,
\end{aligned}
}
whence
\begin{align}
B_{x_3x_3}&=\frac{r(r-1)(r-2)p(p-1)(p-2)(b-a)^3}{(x_1-a)(b-x_1)\AAA(p)^3}\times \notag
\\
&\quad\times\Bigg\{
\int\limits_a^b(\lambda-x_1)\lambda^{r-3}d\lambda\int\limits_a^b(b\!-\!\mu)(\mu\!-\!a)\mu^{p-3}d\mu\!-\!
% \\
% &\qquad\qquad
\int\limits_a^b(b\!-\!\lambda)(\lambda\!-\!a)\lambda^{r-3}d\lambda\int\limits_a^b(\mu\!-\!x_1)\mu^{p-3}d\mu\Bigg\} \notag
\\
&=\frac{r(r-1)(r-2)p(p-1)(p-2)(b-a)^3}{(x_1-a)(b-x_1)\AAA(p)^3}%\times 
%\notag
%\\
%&\qquad\times
\int\limits_a^b\!\!\int\limits_a^b\big[\lambda\mu\!-\!ab\!+\!(a\!+\!b\!-\!\lambda\!-\!\mu)x_1\big]
(\lambda\!-\!\mu)\lambda^{r-3}\mu^{p-3}d\lambda\,d\mu \notag
\\
&=\frac{r(r-1)(r-2)p(p-1)(p-2)(b-a)^3}{2(x_1-a)(b-x_1)\AAA(p)^3}\times \notag
\\
&\quad\times\int\limits_a^b\!\!\int\limits_a^b\big[\lambda\mu-ab+(a+b-\lambda-\mu)x_1\big]
(\lambda-\mu)(\lambda^{r-p}-\mu^{r-p})\lambda^{p-3}\mu^{p-3}d\lambda\,d\mu\,. \label{280801}
\end{align}

In the latter equality, we use symmetrization with respect $\lambda$ and $\mu$, as it has been made before. 
Now it is easy to determine the sign of this expression.
Since $a<x_1<b$ and $\sign\AAA(p)=\sign(p-2)$, the sign of the factor in front of the integral is $\sign(r-2)$. 
The expression in the square brackets is positive. To check this it suffices to note that this is a linear
function with respect to $x_1$ and it is positive at both endpoints of the interval $[a,b]$:
\begin{itemize}
\item at $x_1=a$ we have $(\lambda-a)(\mu-a)$;
\item at $x_1=b$ we have $(b-\lambda)(b-\mu)$.
\end{itemize}
Clearly,
$$
\sign(\lambda^{r-p}-\mu^{r-p})(\lambda-\mu)=\sign(r-p)\,.
$$
Therefore, we get what is needed: $\sign B_{x_3x_3}=\sign(r-2)(r-p)$.

It remains to calculate $\sign B_{x_3x_3}$ on $R$. Since the function $B$ has a simple explicit expression there
{\bf(}see~\eqref{191001}{\bf)}
$$
B=x_2^{\frac{p-r}{p-2}}x_3^{\frac{r-2}{p-2}},
$$
it is easy to calculate $B_{x_3x_3}$ directly:
$$
\sign B_{x_3x_3}=\sign\frac{r-2}{p-2}\Big(\frac{r-2}{p-2}-1\Big)x_2^{\frac{p-r}{p-2}}x_3^{\frac{r-2}{p-2}-2}
=\sign(r-2)(r-p)\,.
$$
The consideration of this last case completes the proof of the lemma.
\end{proof}

Since in the domains $F(0)$ and $R$ the function $B$ is linear on two-dimensional planes transversal to the
direction of $x_3$, just proved Lemma~\ref{200401} ensures the required in Theorem~\ref{080510} concavity/convexity
of $B$ inside these domains. The other domains are foliated by one-dimensional extremal lines transversal to $x_2x_3$-plane, therefore, to prove concavity/convexity there we need to check positivity of the minor 
$B_{x_2x_2}B_{x_3x_3}-B_{x_2x_3}^2$.

We start with domain $\Xi\cii{\mathrm{L}+}$.

\begin{Le}
\label{070502}
$\det \{B_{x_i,x_j}\}_{2 \le i,j \le 3} = B_{x_2x_2}B_{x_3x_3}-B_{x_2x_3}^2 \geq 0$ on $\Xi\cii{\mathrm{L}+}$.
\end{Le}

\begin{proof}
To calculate the sign of this determinant we use formula~\eqref{eq15}:
$$
B_{x_2x_2}B_{x_3x_3}-B_{x_2x_3}^2 = 
\frac{e^{-\frac u{\eps}}\PsiL''(e^{-\frac u{\eps}}h)}{4(x_1-u)(x_1-v)}
\Big(e^{\frac u\eps}\PsiL(e^{-\frac u\eps}h)-\big(h+\eps^3m_p''(v)\big)
\PsiL'(e^{-\frac u\eps}h)+\eps^2m_r''(v)\Big).
$$
Since $v<x_1<u=v+\eps$ and $\sign\PsiL''=\sign(r-2)(r-p)$ (see~\eqref{050801} and Lemma~\ref{200401}) we deduce 
that $\sign\det \{B_{x_i,x_j}\}_{2 \le i,j \le 3}$ coincides with
$$
-\sign(r-2)(r-p)\times
\sign\Big(e^{\frac u\eps}\PsiL(e^{-\frac u\eps}h)-
\big(h+\eps^3m_p''(v)\big)\PsiL'(e^{-\frac u\eps}h)+\eps^2m_r''(v)\Big).
$$
So, we need to check that
\eqlb{230701}{
\sign\Big(e^{\frac u\eps}\PsiL(e^{-\frac u\eps}h)
-\big(h+\eps^3m_p''(v)\big)\PsiL'(e^{-\frac u\eps}h)+\eps^2m_r''(v)\Big)
\buildrel?\over=\sign(p-r)(r-2)\,.
}
We return to variable $\xi$ (see~\eqref{021104}) and use the definition of $\PsiL$ in terms of $\wL$ 
(see~\eqref{Psi_L_def}) and rewrite~\eqref{230701} as follows:
$$
\sign\left(\frac{e^{\frac u\eps}}\eps\wL(\xi;r,\eps)-\Big(e^{\frac u\eps}\wL(\xi;p,\eps)+\eps^3m_p''(v)\Big)
\frac1\eps\frac{\wL'(\xi;r,\eps)}{\wL'(\xi;p,\eps)}+\eps^2m_r''(v) \right)\buildrel?\over=\sign (p-r)(r-2),
$$
which is equivalent to
\eqlb{eq040403}{
\sign\left(\frac{\wL(\xi;r,\eps)+e^{-\frac u\eps}\eps^3m_r''(v)}{\wL'(\xi;r,\eps)}-
\frac{\wL(\xi; p,\eps)+e^{-\frac u\eps}\eps^3m_p''(v)}{\wL'(\xi;p,\eps)}\right)\buildrel?\over=\sign(p-r)\,,
}
because $\sign\wL'(\xi;r,\eps)=\sign(r-2)$ by Lemma~\ref{lemmaSignW_r}. In other words, we need to prove that
the function
$$
s\mapsto\frac{\wL(\xi;s,\eps)+e^{-\frac u\eps}\eps^3m_s''(v)}{\wL'(\xi;s,\eps)}
$$
decreases for all $\xi\ge u$.

We split this function into the sum of two
\eqlb{230704}{
\frac{\wL(\xi;s,\eps)+e^{-\frac u\eps}\eps^3m_s''(v)}{\wL'(\xi;s,\eps)}=
\frac{\wL(\xi;s,\eps)+e^{-\frac\xi\eps}\eps^3m_s''(\xi-\eps)}{\wL'(\xi;s,\eps)}+
\frac{e^{-\frac u\eps}\eps^3m_s''(v)-e^{-\frac\xi\eps}\eps^3m_s''(\xi-\eps)}{\wL'(\xi;s,\eps)}
}
and prove that both of them are decreasing.

We begin with a simpler second function rewriting the expression in the numerator. 
By~\eqref{mpp+} we have
$$
e^{-\frac u\eps}\eps^3m_s''(v)-e^{-\frac\xi\eps}\eps^3m_s''(\xi-\eps)=
\eps^2s(s-1)(s-2)\int\limits_u^\xi e^{-\frac t\eps}(t-\eps)^{s-3}dt\,,
$$
and due to~\eqref{280401} and~\eqref{240403} the question is reduced to verification that the function
\eqlb{230702}{
s\mapsto\frac{\int\limits_u^\xi e^{-\frac t\eps}(t-\eps)^{s-3}dt}
{\int\limits_{-\eps\ }^{\;\eps}(\eps^2-\lambda^2)(\lambda+\xi)^{s-3}d\lambda}
}
is decreasing. Indeed, for $\eps<t<\xi$ and $\lambda>-\eps$ we have
$\frac{\lambda+\xi}{t-\eps}>1$, and therefore, the function
$$
s\mapsto\Big(\frac{\lambda+\xi}{t-\eps}\Big)^{s-3}
$$
is increasing. After integration of this family of increasing functions with positive weight we get
the following increasing function:
$$
s\mapsto\int\limits_{-\eps\ }^{\;\eps}(\eps^2-\lambda^2)\Big(\frac{\lambda+\xi}{t-\eps}\Big)^{s-3}d\lambda\,,
$$
or the following decreasing function:
\eqlb{230703}{
s\mapsto\frac{(t-\eps)^{s-3}}
{\int\limits_{-\eps\ }^{\;\eps}(\eps^2-\lambda^2)(\lambda+\xi)^{s-3}d\lambda}\,.
}
We integrate once more family of functions~\eqref{230703} with the positive weight $e^{-\frac t\eps}$
and finally obtain the decreasing function~\eqref{230702}.

Now, we consider the first summand in~\eqref{230704}. Again, we start with simplifying the expression in 
the numerator. Using twice~\eqref{m-diff}, we get
$$
\eps^2 m_s''(v)=\eps\big(m_s(v)-sv^{s-1}\big)'=m_s(v)-sv^{s-1}-\eps s(s-1)v^{s-2}.
$$
Then, we use the definition~\eqref{w_L_def} of $\wL$:
$$
\begin{aligned}
\wL(\xi;&s,\eps)+e^{-\frac\xi\eps}\eps^3m_s''(\xi-\eps)
\\
&=e^{-\frac\xi\eps}\Big(\frac{(\xi+\eps)^s-(\xi-\eps)^s}2-\eps m_s(\xi-\eps)\Big)+
\eps e^{-\frac\xi\eps}\Big(m_s(\xi-\eps)-s(\xi-\eps)^{s-1}-\eps s(s-1)(\xi-\eps)^{s-2}\Big)
\\
&=e^{-\frac\xi\eps}\Big(\frac{(\xi+\eps)^s-(\xi-\eps)^s}2-\eps s(\xi-\eps)^{s-1}-\eps^2s(s-1)(\xi-\eps)^{s-2}\Big)
\\
&=\frac14s(s-1)(s-2)e^{-\frac\xi\eps}\int\limits_{-\eps\ }^{\;\eps}(\eps-\mu)^2(\mu+\xi)^{s-3}d\mu\,.
\end{aligned}
$$
Therefore, we need to check that the function
\eqlb{240701}{
s\mapsto\frac{\int\limits_{-\eps\ }^{\;\eps}(\eps-\mu)^2(\mu+\xi)^{s-3}d\mu}
{\int\limits_{-\eps\ }^{\;\eps}(\eps^2-\lambda^2)(\lambda+\xi)^{s-3}d\lambda}
}
is decreasing. After differentiating this function, we get in numerator the following expression:
$$
\begin{aligned}
\int\limits_{-\eps\ }^{\;\eps}(\eps-\mu)^2&(\mu+\xi)^{s-3}\log(\mu+\xi)\,d\mu
\int\limits_{-\eps\ }^{\;\eps}(\eps^2-\lambda^2)(\lambda+\xi)^{s-3}d\lambda
\\&\qquad\qquad
-\int\limits_{-\eps\ }^{\;\eps}(\eps-\mu)^2(\mu+\xi)^{s-3}d\mu
\int\limits_{-\eps\ }^{\;\eps}(\eps^2-\lambda^2)(\lambda+\xi)^{s-3}\log(\lambda+\xi)\,d\lambda
\\
&=\int\limits_{\!-\eps\ \ }^{\;\eps}\!\!\!\int\limits_{\!-\eps\ }^{\;\eps}
(\eps-\mu)^2(\eps^2-\lambda^2)(\mu+\xi)^{s-3}(\lambda+\xi)^{s-3}
\big(\log(\mu+\xi)-\log(\lambda+\xi)\big)\,d\mu\,d\lambda\,.
\end{aligned}
$$
Now, we symmetrize this expression (interchanging $\mu$ and $\lambda$):
$$
\begin{aligned}
\frac12\!\!\int\limits_{\!-\eps\ \ }^{\;\eps}\!\!\!\int\limits_{\!-\eps\ }^{\;\eps}
&\big[(\eps-\mu)^2(\eps^2-\lambda^2)-(\eps-\lambda)^2(\eps^2-\mu^2)\big](\mu+\xi)^{s-3}(\lambda+\xi)^{s-3}
\big(\log(\mu+\xi)-\log(\lambda+\xi)\big)\,d\mu\,d\lambda
\\
&=\eps\!\!\int\limits_{\!-\eps\ \ }^{\;\eps}\!\!\!\int\limits_{\!-\eps\ }^{\;\eps}
(\lambda-\mu)(\eps-\mu)(\eps-\lambda)(\mu+\xi)^{s-3}(\lambda+\xi)^{s-3}
\big(\log(\mu+\xi)-\log(\lambda+\xi)\big)\,d\mu\,d\lambda\,.
\end{aligned}
$$
Since $\log$ is an increasing function, we have
$$
(\lambda-\mu)\big(\log(\mu+\xi)-\log(\lambda+\xi)\big)<0\,.
$$
This means that the derivative we have calculated is negative and the function~\eqref{240701}
decreases. This completes the proof of the lemma.
\end{proof}

\begin{Le}
\label{250701}
$B_{x_2x_2}B_{x_3x_3}-B_{x_2x_3}^2 \geq 0$ on $\Xi\cii{\mathrm{R}+}$.
\end{Le}

\begin{proof}
To calculate the sign of this determinant, we use formula~\eqref{290902}:
$$
B_{x_2x_2}B_{x_3x_3}-B_{x_2x_3}^2 = 
\frac{e^{-\frac u{\eps}}\PsiR''(e^{\frac u{\eps}}h)}{4(x_1-u)(x_1-v)}
\Big(e^{-\frac u\eps}\PsiR(e^{\frac u\eps}h)-\big(h-\eps^3k_p''(v)\big)\PsiR'(e^{\frac u\eps}h)+\eps^2k_r''(v)\Big).
$$
Since $u<x_1<v=u+\eps$ and $\sign\PsiR''=-\sign(r-2)(r-p)$ (see~\eqref{220906} and Lemma~\ref{200401}), we deduce 
that $\sign\det \{B_{x_i,x_j}\}_{2 \le i,j \le 3}$ coincides with
$$
\sign(r-2)(r-p)\times
\sign\Big(e^{-\frac u\eps}\PsiR(e^{\frac u\eps}h)
-\big(h-\eps^3k_p''(v)\big)\PsiR'(e^{\frac u\eps}h)+\eps^2k_r''(v)\Big).
$$
So, we need to check that
\eqlb{250702}{
\sign\Big(e^{-\frac u\eps}\PsiR(e^{\frac u\eps}h)
-\big(h-\eps^3k_p''(v)\big)\PsiR'(e^{\frac u\eps}h)+\eps^2k_r''(v)\Big)
\buildrel?\over=\sign(r-p)(r-2)\,.
}
As in the proof of the preceding lemma, we return to variable $\xi$ (see~\eqref{250703}) and use the definition 
of $\PsiR$ in terms of $\wR$ 
(see~\eqref{Psi_R_def}) to rewrite~\eqref{250702} as follows:
$$
\sign\left(-\frac{e^{-\frac u\eps}}\eps\wR(\xi;r,\eps)+\Big(e^{-\frac u\eps}\wR(\xi;p,\eps)-\eps^3k_p''(v)\Big)
\frac1\eps\frac{\wR'(\xi;r,\eps)}{\wR'(\xi;p,\eps)}+\eps^2k_r''(v) \right)\buildrel?\over=\sign (r-p)(r-2),
$$
which is equivalent to
\eqlb{250704}{
\sign\left(\frac{\wR(\xi;r,\eps)-e^{\frac u\eps}\eps^3k_r''(v)}{\wR'(\xi;r,\eps)}-
\frac{\wR(\xi; p,\eps)-e^{\frac u\eps}\eps^3k_p''(v)}{\wR'(\xi;p,\eps)}\right)\buildrel?\over=\sign(p-r)\,,
}
because $\sign\wR'(\xi;r,\eps)=\sign(r-2)$ by Lemma~\ref{ww_increase}. In other words, we need to prove that
the function
$$
s\mapsto\frac{\wR(\xi;s,\eps)-e^{\frac u\eps}\eps^3k_s''(v)}{\wR'(\xi;s,\eps)}
$$
decreases for all $\xi\le u$.

We split this function into the sum of two
\eqlb{250705}{
\frac{\wR(\xi;s,\eps)-e^{\frac u\eps}\eps^3k_s''(v)}{\wR'(\xi;s,\eps)}=
\frac{\wR(\xi;s,\eps)-e^{\frac\xi\eps}\eps^3k_s''(\xi+\eps)}{\wR'(\xi;s,\eps)}+
\frac{e^{\frac\xi\eps}\eps^3k_s''(\xi+\eps)-e^{\frac u\eps}\eps^3k_s''(v)}{\wR'(\xi;s,\eps)}
}
and prove that both of them are decreasing.

Again, we begin with a simpler second function rewriting the expression in the numerator. 
From~\eqref{mpp-} we have
$$
e^{\frac\xi\eps}\eps^3k_s''(\xi+\eps)
=s(s-2)\eps^s+\eps^2s(s-1)(s-2)\int\limits_0^\xi e^{\frac t\eps}(t+\eps)^{s-3}dt\,,
$$
and, therefore, since $v = u + \varepsilon$, we have
$$
e^{\frac\xi\eps}\eps^3k_s''(\xi+\eps)-e^{\frac u\eps}\eps^3k_s''(v)=
-\eps^2s(s-1)(s-2)\int\limits_\xi^u e^{\frac t\eps}(t+\eps)^{s-3}dt\,.
$$
If $\xi\ge\eps$ then due to~\eqref{w_ww_relation}, \eqref{280401}, and~\eqref{240403} the question is reduced 
to verification that the function
\eqlb{250706}{
s\mapsto\frac{\int\limits_\xi^u e^{\frac t\eps}(t+\eps)^{s-3}dt}
{\int\limits_{-\eps\ }^{\;\eps}(\eps^2-\lambda^2)(\lambda+\xi)^{s-3}d\lambda}\,
}
is increasing. If $\xi\le\eps$ we refer to formula~\eqref{ww_dif} and check that the function
\eqlb{260701}{
s\mapsto\frac{s(s-1)\int\limits_\xi^u e^{\frac t\eps}(t+\eps)^{s-3}dt}{(\xi+\eps)^{s-3}}
}
increases. Since $s\mapsto s(s-1)a^s$ increases on $(1,\infty)$ for $a>1$, the function~\eqref{260701}
is increasing as well, because $t>\xi$. To prove that~\eqref{250706} increases, we repeat the chain
of arguments we already used for proving that~\eqref{230702} is decreasing. Since $t>\xi$ and $\lambda<\eps$, 
we have
$\frac{\lambda+\xi}{t+\eps}<1$, and therefore, the function
$$
s\mapsto\Big(\frac{\lambda+\xi}{t+\eps}\Big)^{s-3}
$$
is decreasing. After integrating this family of decreasing functions with a positive weight, we get
the following decreasing function:
$$
s\mapsto\int\limits_{-\eps\ }^{\;\eps}(\eps^2-\lambda^2)\Big(\frac{\lambda+\xi}{t+\eps}\Big)^{s-3}d\lambda\,,
$$
or the following increasing function:
\eqlb{270701}{
s\mapsto\frac{(t+\eps)^{s-3}}
{\int\limits_{-\eps\ }^{\;\eps}(\eps^2-\lambda^2)(\lambda+\xi)^{s-3}d\lambda}\,.
}
We integrate once more family of functions~\eqref{270701} with the positive weight $e^{\frac t\eps}$
and finally obtain the increasing function~\eqref{250706}.

Now, we consider the first summand in~\eqref{250705}. First, we simplify the expression in 
the numerator. Using twice~\eqref{m-diff}, we get
$$
\eps^2 k_s''(v)=\eps\big(sv^{s-1}-k_s(v)\big)'=\eps s(s-1)v^{s-2}-sv^{s-1}+k_s(v).
$$
Then we use definition of $\wR$ (see~\eqref{220904}) first for $\xi\le\eps$:
$$
\begin{aligned}
\wR(\xi;&s,\eps)-e^{\frac\xi\eps}\eps^3 k_s''(\xi+\eps)
\\
&=\eps e^{\frac\xi\eps}\Big(k_s(\xi+\eps)-2\xi(\xi+\eps)^{s-2}\Big)-
\eps e^{\frac\xi\eps}\Big(\eps s(s-1)(\xi+\eps)^{s-2}-s(\xi+\eps)^{s-1}+k_s(\xi+\eps)\Big)
\\
&=\eps e^{\frac\xi\eps}(s-2)(\xi-\eps s)(\xi+\eps)^{s-2}\,,
\end{aligned}
$$
and the function
$$
\frac{\wR(\xi;s,\eps)-e^{\frac\xi\eps}\eps^3 k_s''(\xi+\eps)}{\wR'(\xi;s,\eps)}
\buildrel\eqref{ww_dif}\over=
\frac{\eps(\xi-\eps s)(\xi+\eps)}{\xi^2+\eps^2}
$$
decreases in $s$.

For $\xi\ge\eps$ we have:
$$
\begin{aligned}
\wR(\xi;&s,\eps)-e^{\frac\xi\eps}\eps^3 k_s''(\xi+\eps)
\\
&=e^{\frac\xi\eps}\Big(\frac{(\xi-\eps)^s-(\xi+\eps)^s}2+\eps k_s(\xi+\eps)\Big)
-\eps e^{\frac\xi\eps}\Big(\eps s(s-1)(\xi+\eps)^{s-2}-s(\xi+\eps)^{s-1}+k_s(\xi+\eps)\Big)
\\
&=e^{\frac\xi\eps}\Big(\frac{(\xi-\eps)^s-(\xi+\eps)^s}2+\eps s(\xi+\eps)^{s-1}-\eps^2s(s-1)(\xi+\eps)^{s-2}\Big)
\\
&=-\frac14 e^{\frac\xi\eps} s(s-1)(s-2)\!\!\int\limits_{-\eps\ }^{\;\eps}(\eps+\mu)^2(\mu+\xi)^{s-3}d\mu\,.
\end{aligned}
$$
Therefore, we need to check that the function
\eqlb{250708}{
s\mapsto\frac{\int\limits_{-\eps\ }^{\;\eps}(\eps+\mu)^2(\mu+\xi)^{s-3}d\mu}
{\int\limits_{-\eps\ }^{\;\eps}(\eps^2-\lambda^2)(\lambda+\xi)^{s-3}d\lambda}
}
is increasing. After differentiating this function, we get the following expression in the numerator:
$$
\begin{aligned}
\int\limits_{-\eps\ }^{\;\eps}(\eps+\mu)^2&(\mu+\xi)^{s-3}\log(\mu+\xi)\,d\mu
\int\limits_{-\eps\ }^{\;\eps}(\eps^2-\lambda^2)(\lambda+\xi)^{s-3}d\lambda
\\&\qquad\qquad
-\int\limits_{-\eps\ }^{\;\eps}(\eps+\mu)^2(\mu+\xi)^{s-3}d\mu
\int\limits_{-\eps\ }^{\;\eps}(\eps^2-\lambda^2)(\lambda+\xi)^{s-3}\log(\lambda+\xi)\,d\lambda
\\
&=\int\limits_{\!-\eps\ \ }^{\;\eps}\!\!\!\int\limits_{\!-\eps\ }^{\;\eps}
(\eps+\mu)^2(\eps^2-\lambda^2)(\mu+\xi)^{s-3}(\lambda+\xi)^{s-3}
\big(\log(\mu+\xi)-\log(\lambda+\xi)\big)\,d\mu\,d\lambda\,.
\end{aligned}
$$
Now, we symmetrize this expression (interchanging $\mu$ and $\lambda$):
$$
\begin{aligned}
\frac12\!\!\int\limits_{\!-\eps\ \ }^{\;\eps}\!\!\!\int\limits_{\!-\eps\ }^{\;\eps}
&\big[(\eps+\mu)^2(\eps^2-\lambda^2)-(\eps+\lambda)^2(\eps^2-\mu^2)\big](\mu+\xi)^{s-3}(\lambda+\xi)^{s-3}
\big(\log(\mu+\xi)-\log(\lambda+\xi)\big)\,d\mu\,d\lambda
\\
&=\eps\!\!\int\limits_{\!-\eps\ \ }^{\;\eps}\!\!\!\int\limits_{\!-\eps\ }^{\;\eps}
(\mu-\lambda)(\eps+\mu)(\eps+\lambda)(\mu+\xi)^{s-3}(\lambda+\xi)^{s-3}
\big(\log(\mu+\xi)-\log(\lambda+\xi)\big)\,d\mu\,d\lambda\,.
\end{aligned}
$$
Since $\log$ is an increasing function, we have
$$
(\mu-\lambda)\big(\log(\mu+\xi)-\log(\lambda+\xi)\big)>0\,.
$$
This means that the derivative we have calculated is positive and the function~\eqref{250708}
increases. This completes the proof of the lemma.
\end{proof}

\begin{Le}
\label{250707}
$B_{x_2x_2}B_{x_3x_3}-B_{x_2x_3}^2 \geq 0$ on $\Xi\cii{\mathrm{ch}+}$.
\end{Le}

\begin{proof}
Using~\eqref{230603}, we take a half sum of expressions~\eqref{260605} and~\eqref{020703} to obtain:
\eqlb{310801}{
B_{x_2}=\DDD(r)-\DDD(p)B_{x_3}\,,
}
where
$$
\DDD(s)=\DDD(s;a,b)=\frac s2\cdot\frac{b^{s-1}-a^{s-1}}{b-a}\,.
$$
Using this representation, we express the derivatives of $B$ with respect to $x_2$ via derivatives with 
respect to $x_3$:
\begin{align*}
B_{x_2x_3}&=\DDD_{x_3}(r)-\DDD_{x_3}(p)B_{x_3}-\DDD(p)B_{x_3x_3};
\\
B_{x_2x_2}&=\DDD_{x_2}(r)-\DDD_{x_2}(p)B_{x_3}-\DDD(p)B_{x_2x_3}
\\
&=\big[\DDD_{x_2}(r)-\DDD(p)\DDD_{x_3}(r)\big]-\big[\DDD_{x_2}(p)-\DDD(p)\DDD_{x_3}(p)\big]B_{x_3}
+\DDD(p)^2B_{x_3x_3}\,.
\end{align*}
Using these formulas, after some simplifications,  we rewrite the minor under consideration in the following form:
\eqlb{310802}{
B_{x_2x_2}B_{x_3x_3}-B_{x_2x_3}^2
=\big[\EEE(r)-\EEE(p)B_{x_3}\big]B_{x_3x_3}-\big[\DDD_{x_3}(r)-\DDD_{x_3}(p)B_{x_3}\big]^2\,,
}
where $\EEE(s)=\DDD_{x_2}(s)+\DDD(p)\DDD_{x_3}(s)$.

Now, we plan to find some integral representations for each factor in~\eqref{310802}, from where it will be clear that
this expression is positive. We already have an integral representation for $B_{x_3x_3}$, this is 
formula~\eqref{280801}. Let us show that the factor $\EEE(r)-\EEE(p)B_{x_3}$ has almost the same representation,
the difference consists in the factor in front of the integral.

We start with an integral representation for $\DDD$:
\eqlb{030901}{
\DDD(s)=\frac{s(s-1)}{2(b-a)}\int\limits_a^b\lambda^{s-2}d\lambda\,.
}
Whence
\begin{align*}
\DDD_a(s)&=\frac{s(s-1)(s-2)}{2(b-a)^2}\int\limits_a^b(b-\lambda)\lambda^{s-3}d\lambda\,;
\\
\DDD_b(s)&=\frac{s(s-1)(s-2)}{2(b-a)^2}\int\limits_a^b(\lambda-a)\lambda^{s-3}d\lambda\,,
\end{align*}
and
\eqlb{030902}{
\DDD_{x_i}(s)=\frac{s(s-1)(s-2)}{2(b-a)^2}\int\limits_a^b\big[(b-\lambda)a_{x_i}+(\lambda-a)b_{x_i}\big]
\lambda^{s-3}d\lambda\,.
}
We use this expression to get a representation for $\EEE$:
\eqlb{030903}{
\begin{aligned}
\EEE(s)&=\DDD_{x_2}(s)+\DDD(p)\DDD_{x_3}(s)
\\
&=\frac{s(s-1)(s-2)}{2(b-a)^2}\int\limits_a^b
\big[(b-\lambda)(a_{x_2}+\DDD(p)a_{x_3})+(\lambda-a)(b_{x_2}+\DDD(p)b_{x_3})\big]\lambda^{s-2}d\lambda\,.
\end{aligned}
}
We have formula~\eqref{121002} for $a_{x_3}$ and $b_{x_3}$: 
\eqlb{030904}{
(b-x_1)a_{x_3}=(x_1-a)b_{x_3}=\frac{(b-a)^2}{\AAA(p)}\,,
}
and formulas~\eqref{020702} and~\eqref{260604} for $a_{x_2}$ and $b_{x_2}$:
$$
(b-x_1)a_{x_2}=\frac{b^p-a^p-pb^{p-1}(b-a)}{\AAA(p)}\qquad\text{and}\qquad
(x_1-a)b_{x_2}=-\frac{b^p-a^p-pa^{p-1}(b-a)}{\AAA(p)}\,.
$$
If we plug these expressions into the parts of~\eqref{030903}, we get
\begin{align*}
a_{x_2}+\DDD(p)a_{x_3}&=\frac1{b-x_1}\Big[\frac{b^p-a^p-pb^{p-1}(b-a)}{\AAA(p)}+
\frac p2\cdot\frac{b^{p-1}-a^{p-1}}{b-a}\cdot\frac{(b-a)^2}{\AAA(p)}\Big]
\\
&=\frac1{2\AAA(p)(b-x_1)}\big[2b^p-2a^p-p(b^{p-1}+a^{p-1})(b-a)\big]=-\frac1{2(b-x_1)}\,,
\end{align*}
and similarly
$$
b_{x_2}+\DDD(p)b_{x_3}=\frac1{2(x_1-a)}\,.
$$
As a result, we can rewrite~\eqref{030903} as follows
\eqlb{030905}{
\begin{aligned}
\EEE(s)
&=\frac{s(s-1)(s-2)}{4(b-a)^2}\int\limits_a^b
\Big[\frac{\lambda-a}{x_1-a}-\frac{b-\lambda}{b-x_1}\Big]\lambda^{s-3}d\lambda
\\
&=\frac{s(s-1)(s-2)}{4(b-a)(x_1-a)(b-x_1)}\int\limits_a^b(\lambda-x_1)\lambda^{s-3}d\lambda
=\frac{\AAA(p)}{4(b-a)^4}\AAA_{x_3}(s)\,.
\end{aligned}
}
In the last equality we use representation~\eqref{030906}.

Finally, we calculate the first factor in~\eqref{310802}:
\begin{align*}
\EEE(r)-\EEE(p)B_{x_3}&\buildrel\eqref{230603}\over=\frac{\EEE(r)\AAA(p)-\EEE(p)\AAA(r)}{\AAA(p)}
=\frac1{4(b-a)^4}\big(\AAA_{x_3}(r)\AAA(p)-\AAA_{x_3}(p)\AAA(r)\big)
\\
&=\frac{\AAA(p)^2}{4(b-a)^4}\diff{}{x_3}\Big(\frac{\AAA(r)}{\AAA(p)}\Big)\buildrel\eqref{230603}\over=\frac{\AAA(p)^2}{4(b-a)^4}B_{x_3x_3}\,.
\end{align*}
Therefore,
\eqlb{030907}{
\begin{aligned}
\big[\EEE(r)&-\EEE(p)B_{x_3}\big]B_{x_3x_3}=\Big(\frac{\AAA(p)}{2(b-a)^2}B_{x_3x_3}\Big)^2
\\
&=\Bigg(\frac{r(r-1)(r-2)p(p-1)(p-2)(b-a)}{4(x_1-a)(b-x_1)\AAA(p)^2}\times
\\
&\qquad\times\int\limits_a^b\!\!\int\limits_a^b\big[\lambda\mu-ab+(a+b-\lambda-\mu)x_1\big]
(\lambda-\mu)(\lambda^{r-p}-\mu^{r-p})\lambda^{p-3}\mu^{p-3}d\lambda\,d\mu\Bigg)^2.
\end{aligned}
}

Now, we find an integral representation for another term in~\eqref{310802}:
\eqlb{030908}{
\begin{aligned}
\DDD_{x_3}&(r)-\DDD_{x_3}(p)B_{x_3}=\frac{\DDD_{x_3}(r)\AAA(p)-\DDD_{x_3}(p)\AAA(r)}{\AAA(p)}
\\
&\buildrel{\eqref{270801},\eqref{030902}}\over=\frac{r(r-1)(r-2)p(p-1)(p-2)(b-a)}{2\AAA(p)(b-a)^2}\times
\\
&\quad\times\int\limits_a^b\!\!\int\limits_a^b\big[(\lambda-a)(\mu-a)b_{x_3}-(b-\lambda)(b-\mu)a_{x_3}\big]
(\lambda-\mu)\lambda^{r-3}\mu^{p-3}d\lambda\,d\mu
\\
&\buildrel{\eqref{030904}}\over=\frac{r(r-1)(r-2)p(p-1)(p-2)}{4(x_1-a)(b-x_1)\AAA(p)^2}\times
\\
&\quad\times\int\limits_a^b\!\!\int\limits_a^b\big[(\lambda-a)(\mu-a)(b-x_1)-(b-\lambda)(b-\mu)(x_1-a)\big]
(\lambda-\mu)(\lambda^{r-p}-\mu^{r-p})\lambda^{p-3}\mu^{p-3}d\lambda\,d\mu\,,\hskip-20pt
\end{aligned}
}
where we, as usual, used symmetrization to get the last formula in the chain.

It remains to compare the different parts of expression in~\eqref{030907} and~\eqref{030908}. We see that 
$$
(b-a)\big[\lambda\mu-ab+(a+b-\lambda-\mu)x_1\big]\ge\big|(\lambda-a)(\mu-a)(b-x_1)-(b-\lambda)(b-\mu)(x_1-a)\big|
$$
for $\lambda, \mu, x_1 \in [a,b]$. 
Indeed, on the left-hand side we have a linear function in $x_1$ and we have a convex piecewise linear function
on the right-hand side. The values of both functions coincide at the endpoints of the interval $a\le x_1\le b$.
Therefore, inside the interval the left function is strictly larger than the right one. Since the remaining
terms are the same and they do not change its sign inside the interval, we have the required inequalities for
the squares of integrals.
\end{proof}
\end{appendices}

\noindent Vasily Vasyunin\\
Department of Mathematics and Computer Science,\\
St. Petersburg State University, 14-th Line Vasilyevsky Island, 29,\\
199178, St. Petersburg, Russia\\
vasyunin@pdmi.ras.ru\\

\noindent Pavel Zatitskiy\\
Department of Mathematics and Computer Science,\\
St. Petersburg State University, 14-th Line Vasilyevsky Island, 29,\\
199178, St. Petersburg, Russia\\
pavelz@pdmi.ras.ru\\

\noindent Ilya Zlotnikov\\
Department of Mathematics and Computer Science,\\
St. Petersburg State University, 14-th Line Vasilyevsky Island, 29,\\
199178, St. Petersburg, Russia\\
i.zlotnikov@spbu.ru\\

\end{document}